\documentclass[12pt]{amsart}
\usepackage{amssymb}
\usepackage{amsfonts}
\usepackage{amsmath}
\usepackage{graphicx}
\usepackage{xcolor}
\usepackage{mathrsfs}
\usepackage{stmaryrd}
\usepackage{epsfig,color}
\usepackage{blindtext}
\usepackage{enumerate}
\usepackage{hyperref}
\usepackage{url}
\usepackage{bbm}
\usepackage{filecontents}
\usepackage{nicefrac,mathtools}
\usepackage{bm}  
\usepackage[nocompress]{cite}
\DeclareGraphicsExtensions{.pdf,.jpeg,.png}
\usepackage{epstopdf}
\usepackage{cancel} 
\usepackage[normalem]{ulem} 
\usepackage{verbatim} 
\usepackage{enumitem} 
\usepackage{tikz-cd}
\usetikzlibrary{cd}
\usepackage{color}
\usepackage[msc-links, lite]{amsrefs}
\usepackage{geometry}
\usepackage{tikz}
\usetikzlibrary{decorations.markings}
\usetikzlibrary{arrows.meta}

\usepackage{extarrows}
\makeatletter
\renewcommand\subsubsection{\@startsection{subsubsection}{3}{\z@}%
                                     {-3.25ex\@plus -1ex \@minus -.2ex}%
                                     {-1.5ex \@plus -.2ex}
                                     {\normalfont\normalsize\bfseries}}
\makeatother

\newcommand\QQ{\mathbb{Q}}
\newcommand\CC{\mathbb{C}}

\newcommand\RR{\mathbb{R}}

\newcommand\ZZ{\mathbb{Z}}
\newcommand\LL{\mathbb{L}}
\DeclareMathOperator{\Hom}{\mathscr{H}\text{\kern -3pt {\calligra\large om}}\,}
\numberwithin{equation}{section}

\newtheorem{theorem}{Theorem}[section]

\newtheorem{proposition}[theorem]{Proposition}
\newtheorem{lemma}[theorem]{Lemma}
\newtheorem{corollary}[theorem]{Corollary}
\newtheorem{question}[theorem]{Question}

\theoremstyle{definition}
\newtheorem{definition}[theorem]{Definition}
\newtheorem{warning}[theorem]{Warning}

\newtheorem{remark}[theorem]{Remark}

\numberwithin{equation}{section}

\numberwithin{equation}{section}

\DeclareMathOperator{\Tor}{Tor}
\DeclareMathOperator{\homo}{Hom}

\DeclareMathOperator{\Int}{Int}

\DeclareMathOperator{\Sign}{Sign}
\DeclareMathOperator{\dR}{dR}
\DeclareMathOperator{\rank}{rank}
\DeclareMathOperator{\pt}{pt}
\DeclareMathOperator{\odd}{odd}
\DeclareMathOperator{\ph}{ph}

\let\Im\undefined
\DeclareMathOperator{\Im}{Im}
\DeclareMathOperator{\Ker}{Ker}

\def\hocolim{\qopname\relax m{hocolim}}

\title{$L$-theory Characteristic Classes}
\author{Runjie Hu}
\newcommand{\Addresses}{{
  \bigskip
  \footnotesize

  Runjie Hu, \textsc{Department of Mathematics, Stony Brook University,
    Stony Brook, NY 11794}\par\nopagebreak
  \textit{E-mail address}, \texttt{runjie.hu@tamu.edu}
}}
\date{}
\begin{document}

\maketitle

\begin{abstract}
Although the local information of the $L$-spectra is well understood, the problem of whether this local information can be identified with the geometric data for bundles remains open for decades, which was originally raised in the 1960s and 1970s by Sullivan, Brumfiel, Taylor-Williams and others independently. In this paper, we provide an affirmative answer by proving that Levitt-Ranicki's theory of connective $L$-orientations for $TOP$ bundles and spherical fibrations is equivalent to the $2$-local characteristic classes constructed by Brumfiel-Morgan's, Madsen-Milgram's and Morgan-Sullivan's, as well as Sullivan's odd-prime-local real $K$-theory orientation. A key step in our proof involves constructing more geometric homotopy equivalences from the $2$-local quadratic, symmetric and normal connective $L$-spectra to products of Eilenberg-Maclane spectra and those from odd-local  quadratic and symmetric connective $L$-spectra to the connective real $K$-spectra. This approach reproves the known local structure of $L$-spectra.
\end{abstract}

\tableofcontents

\section{Introduction}

A fundamental question in topology is when a space $X$ is homotopy equivalent to a manifold. The first nontrivial requirement is that $X$ should satisfy Poincar\'e duality. Moreover, Spivak (\cite{Spivak1967}) proves that the Poincar\'e duality on $X$ leads to a canonical normal bundle in the homotopy sense, namely, a spherical fibration $\nu:S(\nu)\rightarrow X$.

If $\pi_1(X)=0$ and the dimension is at least $5$, then the existence of a topological manifold homotopy equivalent to $X$ is equivalent to the existence of a topological $\RR^N$ bundle reduction of $\nu$ (which we also call a TOP bundle reduction). Hence, the existence problem of manifolds is reduced to the homotopy problem of bundle liftings. The latter problem has been well understood in many aspects.

\begin{itemize}
    \item Localized at odd primes, the bundle lifting obstruction is equivalent to the existence of a $KO$-theory orientation for $\nu$, by Sullivan (see \cite[Theorem 6.5]{Sullivan2009}).
    \item Localized at prime $2$, Brumfiel-Morgan (\cite{Brumfiel&Morgan}) and Madsen-Milgram (\cite{Madsen&Milgram1974II}) independently discover the obstruction. Brumfiel-Morgan prove that there exist characteristic classes $k^G\in H^{4*+3}(X;\ZZ/2)$ and $l^G\in H^{4*}(X;\ZZ/8)$ for a spherical fibration over $X$. The obstruction to the existence of a $TOP$ bundle structure, localized at prime $2$, is the vanishing of $k^G$ and a $\ZZ_{(2)}$-coefficient lifting of $l^G$.
\end{itemize}

Similar results are known by Quinn (\cite{Quinn1972}) and Jones (\cite{Jones1971}).

In fact, the studies of the uniqueness problem for bundle liftings are earlier than those of the existence problem. Sullivan discovers a sequence of obstructions, which is called Sullivan's obstruction (\cite{Sullivan1965}). Usually in homotopy theory an obstruction can be defined only when all the previous obstructions vanish. However, Sullivan proves that his obstructions are indeed a priori: namely, these obstructions can be defined without knowing whether the previous obstructions vanish. Sullivan's result of a priori obstructions is equivalent to the theorem that the surgery space $G/TOP$ is a product of Eilenberg-Maclane spaces at the prime $2$ and $KO$-theory at odd primes (Sullivan for $G/PL$ \cite{Sullivan1996} and Kirby-Siebenmann for $G/TOP$ \cite[p.~329]{Kirby&Siebenmann}, or see\cite[p.~107]{Weinberger&Chang}). 

A concrete homotopy equivalence from $G/TOP$ to the product of Eilenberg-Maclane spaces is equivalent to constructing explicit characteristic classes for surgery problems. Precisely, Rourke-Sullivan (\cite{Sullivan&Rourke}) construct the $k$-classes of degree $(4*+2)$; Morgan-Sullivan (\cite{Sullivan&Morgan}) construct the $l$-classes of degree $4*$ for $G/TOP$ (Milgram independently \cite{Milgram1972}). 

\vspace{3mm}

Afterwards, Ranicki (\cite{Ranicki1980}\cite{Ranicki1979}\cite{Ranicki1981}, later on Weiss \cite{Weiss1985}) introduces the cobordism theory of chain complexes with Poincar\'e duality which is called algebraic $L$-theory. There are three connective $L$-spectra for the cobordism theories, i.e., $\LL^q,\LL^s,\LL^n$, which are named by the quadratic, symmetric or normal $L$-spectra respectively. These three $L$-spectra form a natural fibration sequence $\LL^q\rightarrow \LL^s\rightarrow \LL^n$.

Using this chain-level language, Levitt-Ranicki (\cite{Ranicki1979},\cite{Levitt&Ranicki}) prove that a spherical fibration admits a canonical $\LL^n$-orientation and a TOP bundle admits a canonical $\LL^s$-orientation. Furthermore, Levitt-Ranicki also prove that the existence of a $TOP$ bundle structure on a spherical fibration is equivalent to the existence of an $\LL^s$-orientation lifting of the canonical connective $\LL^n$-orientation.

Levitt-Ranicki's chain-level theory and Brumfiel-Morgan, Madsen-Milgram, Morgan-Sullivan, Rourke-Sullivan, Sullivan's local theories are all fundamental tools to study manifolds. However, the equivalence relations between these two theories remain mysterious for many years. For example, Brumfiel raises the following question, which is open for about $50$ years:

\begin{question}[Brumfiel]\label{question}
Is Levitt-Ranicki's theory of connective $L$-orientations for TOP bundles and spherical fibrations equivalent to Brumfiel-Morgan's and Morgan-Sullivan's $2$-local characteristic classes and Sullivan's odd-prime-local real $K$-theory orientation?\footnote{``...But at some point I (Brumfiel) can't avoid the question, why is the ad hoc fiber sequence I defined above $L^q\rightarrow L^s\rightarrow L^n$ really the same as the Ranicki sequence? And why are the canonical multiplicative orientations $MSTOP\rightarrow L^s$ and $MSG\rightarrow L^n$ in my ad hoc constructions the same as Ranicki's?..."}
\end{question}

Some people believe in an affirmative answer to this question but a proof never appears in the literature. Taylor-Williams in \cite{Taylor&Williams} partially answer this question (see \ref{corollary}). But they bring up the same question as well: 

\begin{question}[Taylor-Williams, \cite{Taylor&Williams}*{p.~191}]\label{question2}
Is Taylor-Williams's result of $L$-spectra equivalent to Brumfiel-Morgan's chracteristic classes for spherical fibrations?
\end{question}

In this paper, we give a complete proof for an affirmative answer to Question \ref{question} and Question \ref{question2}.

\begin{theorem}\label{Main}
Ranicki-Levitt's theory of connective $L$-orientations for $TOP$ bundles and spherical fibrations is equivalent to Brumfiel-Morgan's, Madsen-Milgram's and Morgan-Sullivan's $2$-local characteristic classes and Sullivan's odd-prime-local $KO$-orientation.
\end{theorem}

To prove this theorem, we show the followings.

\begin{theorem}\label{Main2}
\begin{enumerate}
    \item At prime $2$, Ranicki-Levitt's $\LL^n$-orientation for a spherical fibration is equivalent to three kinds of characteristic classes (Theorem \ref{characteristic-classes-11}). Two of them are Brumfiel-Morgan's $\ZZ/2$ and $\ZZ/8$ classes (Theorem \ref{characteristic-classes-1}, Theorem \ref{characteristic-classes-2}); the third one is a combination of Stiefel-Whiteney classes (Theorem \ref{characteristic-classes-3}).
    \item At prime $2$, Ranicki-Levitt's $\LL^s$-orientation for a $TOP$ bundle is equivalent to two kinds of characteristic classes (Theorem \ref{characteristic-classes-12}). One of them is Morgan-Sullivan's $\ZZ/2$ class (Theorem \ref{characteristic-classes-4}); the other one is also a combination of Stiefel-Whiteney classes (Theorem \ref{characteristic-classes-5}).
    \item At odd primes, Ranicki-Levitt's $\LL^s$-orientation for a $TOP$ bundle is equivalent to Sullivan's $KO$-orientation (Theorem \ref{main-result-3}).
\end{enumerate}
\end{theorem}

Our result also provides an answer to Sullivan's question in 1960's:

\begin{question}[Sullivan]\label{question3}
What is an analogous $2$-local orientation  theory like the odd-prime theory for $TOP$ bundles? (Recall that at odd primes, a $TOP$ bundle means a spherical fibration with a real $K$-theory orientation.)
\end{question}

Our answer is that at prime $2$ a $TOP$ bundle is a spherical fibration with some ``relative'' orientation in terms of vanishing of certain characteristic classes and this is equivalent to saying that a $TOP$ bundle is a spherical fibration with a lifting of the canonical connective normal $L$-orientation.

As a corollary of our arguments, we reprove Taylor-William's, Sullivan's theorem and the connective version of Hebestreit-Land-Nikolaus's theorem (\cite{Taylor&Williams},\cite{Sullivan2009}\cite{Hebestreit-Land-Nikolaus}).

\begin{corollary}\label{corollary}
The connective spectra $\LL^q,\LL^s,\LL^n$ are products of Eilenberg-Maclane spectra at prime $2$.   
\end{corollary}

\begin{corollary}\label{corollary}
The connective spectra $\LL^q,\LL^s$ are homotopy equivalent to the $KO$-theory at odd primes. The connective spectra $\LL^n$ is contractible at odd primes.
\end{corollary}

It seems that the local information of $L$-spectra would establish a direct answer to Questions \ref{question}\ref{question2}\ref{question3}, since the $2$-local information of Levitt-Ranicki's $L$-orientation provides some $2$-local characteristic classes of spherical fibrations and $TOP$ bundles. However, the identification of these $2$-local characteristic classes and the characteristic classes constructed by Brumfiel-Morgan and Morgan-Sullivan remains unclear. In other words, a more refined approach is needed to establish a precise identification of the $2$-local $\LL$-spectra and the products of Eilenberg-Maclane spectra.

\subsubsection*{Organization of This Paper and Outlined Proof} Our method is based on the a-priori-invariant method. Section \ref{Preliminaries} is an overview of $L$-theory and the method of a priori invariants. We also review how Brumfiel-Morgan, Morgan-Sullivan and Sullivan use the a-priori-invariant method method to construct $2$-local characteristic classes for spherical fibrations and $TOP$ bundles and odd-prime-local $KO$-orientations for $TOP$ bundles. Section \ref{Znchain} is devoted to generalize the $\ZZ/n$ manifolds to the chain level and study various bordism invariants for $\ZZ/n$-chains. The technical part is to check the product formulae for ``tensor products'' of $\ZZ/n$-chains. In Section \ref{Mainsection}, we use the tools of Section \ref{Znchain} and a-priori-invariant method to study the local homotopy types of $L$-spectra. In this part, we construct characteristic classes for $L$-theories and prove product formula for these characteristic classes. At last, we prove the Theorem \ref{main-result-1}, Theorem \ref{main-result-2} and Theorem \ref{main-result-3} to answer the equivalence questions.

\subsubsection*{Notations}
In this paper, we use the abbreviation $\LL^s, \LL^q,\LL^n$ to represent the $0$-connective symmetric $L$-spectrum $\LL^s\langle 0\rangle$, the $1$-connective quadratic $L$-spectrum $\LL^q\langle 1\rangle$ and Ranicki's ``$1/2$-connective'' normal $L$-spectrum $\LL^n\langle 1/2\rangle$ respectively. We also use the same symbols $\LL^s,\LL^q,\LL^n$ to represent the $0$-th space of the these connective spectra.

\subsubsection*{Acknowledgements} I am grateful to my thesis advisor, Dennis Sullivan, for bringing up the question of $2$-local orientations for bundles, which is open for over $50$ years. I would like to thank Gregory Brumfiel for mentioning the same question to me. I sincerely appreciate the invaluable insights and commnents from James F. Davis, Jiahao Hu, John Morgan, John Pardon, Shmuel Weinberger, Zhouli Xu and Guozhen Wang on this paper. I wish to express my gratitude to the Simons Foundation International for their financial support during my research on this topic.

\section{Preliminaries}\label{Preliminaries}

For the readers' convenience, we write all the preliminary materials needed for this paper in this section. We add and prove some technical lemmas for our use. We also simply some proofs for some theorems. More preciesly, we first review Ranicki's chain-level constructions of $L$-theory, and then we review the a priori invariant method and their applications to bundle lifting problems.

\subsection{\texorpdfstring{$L$}{lg} Groups}\label{Lgroup}

We review the definition of $L$-groups in this subsection. Throughout this article, by a chain complex we always mean a bounded chain complex whose underlying abelian group at each degree is a finitely generated free abelian group. 

For a chain complex $C_*$, the tensor product $C\otimes C\cong \homo_{\ZZ}(C^{-*},C_*)$ has a natural $\ZZ/2$ action. An $n$-dimensional symmetric structure on a chain complex $C$ is a degree $n$ homotopy $\ZZ/2$ invariant $\phi$. More explicitly, let $W_*\rightarrow \ZZ\rightarrow 0$ be a free resolution of the trivial $\ZZ [\ZZ/2]$-module $\ZZ$. An explicit form of $W_*$ is 
\[\cdots\rightarrow \ZZ[\ZZ/2]\xrightarrow{1-T} \ZZ[\ZZ/2]\xrightarrow{1+T} \ZZ[\ZZ/2]\xrightarrow{1-T}\ZZ[\ZZ/2]\rightarrow 0\rightarrow \cdots\]
where $T$ is the generator in $\ZZ/2$. Each degree $n$ homotopy $\ZZ/2$ invariant $\phi$ is represented by an element $\phi\in Q^s_n(C)=H_n(\homo_{\ZZ[\ZZ/2]}(W,C\otimes C))$. Informally, $\phi$ consists of chain maps $\phi_i:C^{n+i-*}\rightarrow C_*$ for $i\geq 0$ so that each $\phi_i$ is a chain homotopy between $\phi_{i-1}:C^{n+i-1-*}\rightarrow C_*$ and its dual $D(\phi_{i-1}):C^{n+i-1-*}\rightarrow C_*$.

A symmetric chain complex $(C,\phi)$ is Poincar\'e if the chain map $\phi_0:C^{n-*}\rightarrow C_*$ is a chain homotopy equivalence. The notion of Poincar\'e symmetric chains is a derived generalization of nondegenerate symmetric bilinear forms of free abelian groups or nondegenerate linking forms on torsion abelian groups. 

Recall that a homotopy coinvariant is an element in $Q^q_n(C)=H_n(W\otimes_{\ZZ[\ZZ/2]}(C\otimes C))$. Define a quadratic structure on a chain complex $C$ by a homotopy coinvariant $\psi$. Similarly, $\psi$ consists of chain maps $\psi_i:C^{n+i-*}\rightarrow C_*$ for $i\leq 0$ so that each $\psi_i$ is a chain homotopy between $\psi_{i-1}$ and its dual $D(\psi_{i-1}):C^{n+i-1-*}\rightarrow C_*$. A quadratic chain complex $(C,\psi)$ is called Poincar\'e if the chain map $\psi_0:C^{n-*}\rightarrow C_*$ is a chain homotopy equivalence. 

Let $\mathbf{\Delta}$ be the category of totally ordered nonempty finite sets with morphisms order-preserving inclusions. A $\Delta$-set is a contravariant functor from $\mathbf{\Delta}$ to the category of sets. Indeed, a $\Delta$-set has the same hierarchy and boundary maps like a simplicial set, but a $\Delta$-set does not have degeneracy maps. Indeed, a $\Delta$-set $X$ can be viewed as a category with objects the simplices in $X$ and with morphisms the inclusions between simplices. A presheaf of chain complexes over $X$ is a functor $\mathcal{C}$ from the category $X$ to the category of chain complexes. Define the Verdier dual presheaf $D(\mathcal{C})$ by $
D(\mathcal{C})(\sigma)= \hocolim_{\tau\subset \sigma} \mathcal{C}(\tau)^{-*}=\oplus_{\tau\subset \sigma}\mathcal{C}(\tau)^{-*+|\tau|}
$ 
for any simplex $\sigma$ of $X$. 

Now assume $X$ is a finite $\Delta$-set. Let $\mathbf{Hom} (D(\mathcal{C}),\mathcal{C})$ be the differential graded presheaf of homomorphisms. It also has a natural $\ZZ/2$-action, so we can define $n$-dimensional symmetric/quadratic structures on a presheaf. A symmetric/quadratic structure is (locally) Poincar\'e if the chain map $D(\mathcal{C})(\sigma)\rightarrow \mathcal{C}(\sigma)$ is a chain homotopy equivalence for each simplex $\sigma$ of $X$.

Define the assembly (or equivalently, the set of global sections) of a presheaf $\mathcal{C}$ by $\mathcal{C}(X)= \hocolim_{\sigma\in X} \mathcal{C}(\sigma)=\oplus_{\sigma\in X}\mathcal{C}(\sigma)_{*-n+|\sigma|}$, where $n$ is the dimension of $X$.

\begin{remark}
All the definitions above like presheaf, Verdier dual presheaf, assembly and symmetric/quadratic structures on a presheaf can be generalized to the case when $X$ is a regular cell complex (for a definition see \cite[I.1.1]{Steenrod}).
\end{remark}

\begin{lemma}\label{DualAssemblyLemma}
Let $X$ be an $n$-dimensional closed $PL$ manifold with a $PL$ triangulation. Let $\mathcal{C}$ be a presheaf of chain complexes over $X$. Then the assembly $D(\mathcal{C})(X)$ is canonically chain homotopy equivalent to $\mathcal{C}(X)^{n-*}=\Sigma^n\homo(\mathcal{C}(X),\ZZ)$.
\end{lemma}

\begin{proof}
Firstly, 
\begin{eqnarray*}
\mathcal{C}(X)^{n-*} & = & \Sigma^n\homo(\mathcal{C}(X),\ZZ)=\Sigma^n\homo(\hocolim_{\sigma\in X}\mathcal{C}(\sigma),\ZZ)\\
 & = & \Sigma^n\homo(\oplus_{\sigma\in X}\mathcal{C}(\sigma)_{*-n+|\sigma|},\ZZ)=\Sigma^n(\oplus_{\sigma\in X}\mathcal{C}(\sigma)^{-*-n+|\sigma|}) \\
 & = & \oplus_{\sigma\in X}\mathcal{C}(\sigma)^{-*+|\sigma|}=\hocolim_{\sigma\in X} \mathcal{C}(\sigma)^{-*}
\end{eqnarray*}
Notice that $\mathcal{C}(\sigma)^{-*}$ is a precosheaf over $X$. So it suffices to prove that given a precosheaf $\mathcal{D}$ over $X$, the canonical chain map $\hocolim_{\sigma\in X}\mathcal{D}(\sigma) \rightarrow\hocolim_{\sigma\in X}\hocolim_{\tau\subset\sigma}\mathcal{D}(\tau)$ is a chain homotopy equivalence.

Let us write down the chain complexes of two sides explicitly. $\hocolim_{\sigma\in X}\mathcal{D}(\sigma)=\oplus_{\sigma\in X}\mathcal{D}(\sigma)_{*-|\sigma|}$, with the differential decomposed like the following. For each $\sigma$, there is an obvious differential $\mathcal{D}(\sigma)_{r-|\sigma|}\rightarrow \mathcal{D}(\sigma)_{r-1-|\sigma|}$; for each pair $\tau\subset\sigma$ of codimension $1$, there is a map $\mathcal{D}(\sigma)_{r-|\sigma|}\rightarrow \mathcal{D}(\tau)_{r-|\sigma|}=\mathcal{D}(\tau)_{r-1-|\tau|}$ induced by the cosheaf structure.

On the other hand, $\hocolim_{\sigma\in X}\hocolim_{\tau\subset\sigma}\mathcal{D}(\tau)=\oplus_{\sigma\in X}\oplus_{\tau\subset \sigma}\mathcal{D}(\tau)_{*-|\tau|-n+|\sigma|}$, with the differential decomposed into three parts. For each pair $\tau\subset \sigma$, there is an obvious differential $\mathcal{D}(\tau)_{r-|\tau|-n+|\sigma|}\rightarrow \mathcal{D}(\tau)_{r-1-|\tau|-n+|\sigma|}$; for each pair $\gamma\subset \tau$ of codimension $1$, there are a map $\mathcal{D}(\tau)_{r-|\tau|-n+|\sigma|}\rightarrow \mathcal{D}(\gamma)_{r-|\tau|-n+|\sigma|}=\mathcal{D}(\gamma)_{r-1-|\gamma|-n+|\sigma|}$ induced by the cosheaf structure; for each pair $\sigma\subset\delta$ of codimension $1$, there is an identity map $\mathcal{D}(\tau)_{r-|\tau|-n+|\sigma|}\rightarrow \mathcal{D}(\tau)_{r-|\tau|-n+|\sigma|}=\mathcal{D}(\tau)_{r-1-|\tau|-n+|\delta|}$.

Hence, $\hocolim_{\sigma\in X}\hocolim_{\tau\subset\sigma}\mathcal{D}(\tau)$ is also isomorphic to $\oplus_{\alpha\in X}\oplus_{\alpha\subset \beta}\mathcal{D}(\alpha)_{*-|\alpha|-n+|\beta|}$, with the differential decomposed into three parts like above. So we have proved that \\ $\hocolim_{\sigma\in X}\hocolim_{\tau\subset\sigma}\mathcal{D}(\tau)=\hocolim_{\alpha\in X}\hocolim_{\alpha\subset \beta}\mathcal{D}(\alpha)$, where by $\hocolim_{\alpha\subset \beta}\mathcal{D}(\alpha)$ we mean the colimit over the constant diagram $\mathcal{D}(\alpha)$ indexed by all simplices $\beta$ containing the fixed $\alpha$. The map into $\hocolim_{\sigma\in}\mathcal{D}(\sigma)$ is induced by $\hocolim_{\alpha\subset \beta}\mathcal{D}(\alpha)\rightarrow \mathcal{D}(\alpha)$. We only need to prove that this map is a chain homotopy equivalence and the lemma follows from \cite[Proposition 1.14]{Ranicki&Weiss}.

Notice that a PL triangulation on $X$ induces a dual cell decomposition on $X$ since $X$ is a PL manifold. The dual cell decomposition on $X$ is a regular cell decomposition. Let $D(\alpha)$ be the regular cell dual to a simplex $\alpha$. But $\hocolim_{\alpha\subset \beta}\mathcal{D}(\alpha)=\mathcal{D}(\alpha)\otimes C_*(D(\alpha))$, where $C_*(D(\alpha))$ is the cellular chain complex of $D(\alpha)$ (with respect to the dual cell decomposition). It is obvious that the chain map $\mathcal{D}(\alpha)\otimes C_*(D(\alpha))\rightarrow \mathcal{D}(\alpha)$ is a chain homotopy equivalence.
\end{proof}

\begin{corollary}\label{AssemblyLemma}
Let $X$ be a closed $n$-dimensional PL manifold with a $PL$ triangulation. Let $\mathcal{C}$ be an $m$-dimensional Poincar\'e presheaf of symmetric/quadratic chain complexes over $X$. The assembly $\mathcal{C}(X)$ is a Poincar\'e symmetric/quadratic chain complex of dimension $m+n$.
\end{corollary}

\begin{proof}
The lemma showed that $D(\mathcal{C})(X)\rightarrow \mathcal{C}(X)^{n-*}$ is a chain homotopy equivalence, where $n$ is the dimension of $X$. Hence, the symmetric/quadratic structure on the presheaf $\mathcal{C}$ induces a symmetric/quadratic structure on $\mathcal{C}(X)$. Moreover, since $\mathcal{C}$ is locally Poincar\'e, \cite[Proposition 1.14]{Ranicki&Weiss} indicates that $D(\mathcal{C})(X)_{-m+*}\rightarrow C(X)$ is a chain homotopy equivalence and hence the induced map $\mathcal{C}(X)^{m+n-*}\rightarrow\mathcal{C}(X)$ is also a chain homotopy equivalence.
\end{proof}

\begin{remark}
The lemmas \ref{DualAssemblyLemma} and \ref{AssemblyLemma} both hold for a $PL$ regular cell decomposition of a $PL$ manifold, since the only fact we need to use in the proof is that the dual cone decomposition of a $PL$ regular cell decomposition for a $PL$ manifold is still a regular cell decomposition.
\end{remark}

\begin{definition}
Let $X=\Delta^1$, the $\Delta$-set of the unit interval. A presheaf $\mathcal{C}$ over $X$ is indeed a chain map $C\oplus C'\rightarrow D$. Let $\phi$ (or $\psi$) be an $n$-dimensional symmetric (or quadratic) structure on $\mathcal{C}$. Call $D$ a Poincar\'e symmetric (or quadratic) bordism between two Poincar\'e symmetric (or quadratic) chain complexes $C$ and $C'$ if the presheaf $(\mathcal{C},\phi)$ (or $(\mathcal{C},\psi)$) is locally Poincar\'e. Moreover, if $C'=0$, call $C\rightarrow D$ a Poincar\'e symmetric (or quadratic) pair.
\qed
\end{definition}

\begin{remark}
A (Poincar\'e) symmetric (or quadratic) presheaf $\mathcal{C}$ over an $n$-dimensional simplex $\Delta^n$ is also called an $n$-ad of (Poincar\'e) symmetric (or quadratic) chain complexes. A (Poincar\'e) symmetric (or quadratic) presheaf over a finite $\Delta$-set $X$ is indeed a presheaf of ads of (Poincar\'e) symmetric (or quadratic) chain complexes.
\end{remark}

The following lemma will be used in many places.

\begin{lemma}[Ranicki's Miracle Lemma]\cite[Proposition 3.4]{Ranicki1980}\label{Ranickimiracle}
The chain homotopy classes of $n$-dimensional symmetric (or quadratic) chain complexes are in one-to-one
correspondence with the homotopy classes of $n$-dimensional Poincar\'e symmetric (or quadratic) pairs.
\end{lemma}

In \cite{Weiss1985} (also see \cite[Definition 2.2]{Ranicki1992}), Weiss defined a chain bundle on a chain complex $C$ by a $0$-dimensional cycle $\gamma\in \homo_{\ZZ[\ZZ/2]}(\widehat{W},C^{-*}\times C^{-*})$, where $\widehat{W}$ is the Tate complex associated to the group ring $\ZZ[\ZZ/2]$, which is
\[
\cdots\rightarrow \ZZ[\ZZ/2]\xrightarrow{1-T} \ZZ[\ZZ/2]\xrightarrow{1+T} \ZZ[\ZZ/2]\xrightarrow{1-T}\ZZ[\ZZ/2]\xrightarrow{1+T} \ZZ[\ZZ/2]\rightarrow \cdots
\]

$\widehat{Q}_*(C)=H_*(\homo_{\ZZ[\ZZ/2]}(\widehat{W},C\otimes C))$ is like the `$K$-theory of spherical fibrations' for chain complexes. $\homo_{\ZZ[\ZZ/2]}(\widehat{W},C^{-*}\otimes C^{-*})$ has the homotopy invariance property, i.e., the chain maps $f^*,(f')^*:\homo_{\ZZ[\ZZ/2]}(\widehat{W},D^{-*}\otimes D^{-*})\rightarrow \homo_{\ZZ[\ZZ/2]}(\widehat{W},C^{-*}\otimes C^{-*})$ induced by two homotopic chain maps $f,f':C\rightarrow D$ is still chain homotopic,
where the homotopy relies on a choice of a `diagonal' element $\omega\in(C(\Delta^1)\otimes C(\Delta^1))_1$ (see \cite[Proposition 1.1]{Ranicki1980}).

The identity map of $\Sigma C$ represents a chain homotopy from the $0$-map $C\xrightarrow{0} \Sigma C$ to itself. It induces a chain homotopy map $\homo_{\ZZ[\ZZ/2]}(\widehat{W},C^{-1-*}\otimes C^{-1-*})\rightarrow \homo_{\ZZ[\ZZ/2]}(\widehat{W},C^{-*}\otimes C^{-*})$. By shifting the degree, we get a chain map 
\[
S:\Sigma \homo_{\ZZ[\ZZ/2]}(\widehat{W},C^{-*}\otimes C^{-*})\rightarrow \homo_{\ZZ[\ZZ/2]}(\widehat{W},C^{1-*}\otimes C^{1-*})
\]

Moreover, there is a natural long exact sequence connecting homotopy coinvariants $Q^q_n(C)$, homotopy invariants $Q^s_n(C)$ and Tate cohomology $\widehat{Q}_n(C)$ (see \cite[Proposition 1.2]{Ranicki1980} for the construction and proof), i.e.,
\[
\cdots\rightarrow Q^q_n(C)\xrightarrow{1+T} Q^s_n(C)\xrightarrow{J} \widehat{Q}_n(C)\xrightarrow{\partial} Q^q_{n-1}(C)\rightarrow \cdots
\]
where $1+T$ corresponds to the polarization of a quadratic form into a symmetric form.

A normal structure on an $m$-dimensional symmetric chain complex $(C,\phi)$ is defined by a chain bundle $\gamma$ over $C$ and a homology $\zeta\in (\homo_{\ZZ[\ZZ/2]}(\widehat{W},C^{1-*}\otimes C^{1-*}))_{n+1}$ between $J(\phi)$ and $\phi_*(S^m(\gamma))$. 

There are some basic facts about normal chain complexes. Define an $n$-dimensional symmetric-quadratic pair $f:C\rightarrow D$ by a Poincar\'e symmetric pair structure on $C\rightarrow D$ and a Poincar\'e quadratic refinement of the symmetric structure on $C$. There are also notions of bordisms and $k$-ads of symmetric-quadratic pairs.

\begin{proposition}[\cite{Ranicki1992}*{Proposition 2.6(ii), Proposition 2.8(i)}]
(1) Each Poincar\'e symmetric complex has a canonical normal structure.

(2) There is a natural one-to-one correspondence between the homotopy classes of $n$-dimensional Poincar\'e symmetric-quadratic pairs and those of $n$-dimensional normal chains.
\end{proposition}

Like before, we can define a normal structure on a presheaf of chain complexes over an arbitrary finite $\Delta$-set. Consequently $n$-ads and bordisms of normal complexes can also be defined. There is a correspondence between presheaves of normal chains and presheaves of symmetric-quadratic pairs. The lemma \ref{AssemblyLemma} for the normal case or the case of symmetric-quadratic pairs is also true by the same argument.

\begin{definition} (\cite[p.~137]{Ranicki1980} and \cite[p.~40]{Ranicki1992})
Define $L^s_m$, $L^q_m$ and $L^n_m$ by the sets of bordism classes of $m$-dimensional Poincar\'e symmetric chain complexes, Poincar\'e quadratic chain complexes and normal complexes, respectively. They are indeed abelian groups, where the additions are induced by the direct sums of chains. 
\end{definition} 

Because of the equivalence between normal chain complexes and Poincar\'e symmetric-quadratic chain pairs, there is a natural long exact sequence (\cite[p.~45]{Ranicki1992})
\begin{equation}\label{lesLgroup}
\cdots\rightarrow L^q_m\xrightarrow{1+T} L^s_m \xrightarrow{J} L^n_m \xrightarrow{\partial} L^q_{m-1}\rightarrow \cdots    
\end{equation}

For geometric intuitions, one can think of a Poincar\'e symmetric chain as a Poincar\'e space, a Poincar\'e quadratic chain as a degree one normal map between two Poincar\'e spaces and a normal chain as a normal space (see \cite{Quinn1972} for a definition). One can use these intuitions to make algebraic constructions like geometric constructions. For example, we may define an algebraic gluing of two $m$-dimensional Poincar\'e symmetric pairs $D\rightarrow C$ and $(-D)\rightarrow C'$ along the common boundary $D$ and get a $m$-dimensional Poincar\'e symmetric chain complex $C\bigcup_D C'$ (\cite[p.~77]{Ranicki1981}), like the gluing of two manifolds along the common boundary.

Analogously, we can also construct the following chain-level bordism invariants. 

Let $(C,\phi)$ be an $m$-dimensional Poincar\'e symmetric chain complex.  $\phi$ induces a\ nondegenerate symmetric bilinear form on $H_{2k}(C)$ when $m=4k$ and a nondegenerate linking form on the torsion subgroup $\Tor(H_{2k+1}(C))$ when $m=4k+1$. If $m=4k$, the signature $\Sign(C)\in \ZZ$ is defined by the signature of the bilinear form over $H_{2k}(C)\otimes\RR$; if $m=4k+1$, the de Rham invariant $\dR(C,\phi)\in \ZZ/2$ is defined by the $\ZZ/2$ rank of $\Tor(H_{2k+1}(C))\otimes \ZZ/2$. 

For a $4k$-dimensional Poincar\'e symmetric chain pair $D\rightarrow C$, we can also define the signature $\Sign(C)$ by the signature of the nondegenerate symmetric bilinear form on $\Im(H_{2k}(C)\rightarrow H_{2k}(C,D))$. Then we have the Novikov's additive formula of Poincar\'e symmetric pairs, namely, $\Sign(C\bigcup_D C')=\Sign(C)+\Sign(C')$, where $D\rightarrow C$ and $(-D)\rightarrow C'$ are two $4k$-dimensional Poincar\'e symmetric chain pairs.

Similarly, given an $m$-dimensional Poincar\'e quadratic chain complex $(C,\psi)$, $\psi$ induces an nondegenerate quadratic form on $H_{2k}(C)$ when $m=4k$ and a nondegenerate skew-quadratic form on $H_{2k+1}(C)$ when $m=4k+2$. If $m=4k$, we can define the index $I(C)=\frac{1}{8}\Sign( {C})\in \ZZ$, where $\Sign( {C})$ is the signature of the quadratic form on $H_{2k}(C)\otimes\RR$; if $m=4k+2$, the Kervaire invariant $K(C,\phi)\in \ZZ/2$ is the Kervaire-Arf invariant of the quadratic form on $H_{2k+1}(C)\otimes \ZZ/2$. 

Let $C'$ be an $m$-dimensional normal chain complex and let $D\rightarrow C$ be the corresponding Poincar\'e symmetric-quadratic pair.

If $m=4k+3$, define the Kervaire invariant $K(C')\in \ZZ/2$ by the Kervaire invariant of $D$.

If $m=4k+1$, the index of $D$ is $0$ since $D$ is a boundary. Then there exists a Poincar\'e quadratic pair $(-D)\rightarrow \widetilde{C}$. Define the de Rham invariant $\dR(C')\in \ZZ/2$ of $C'$ by the de Rham invariant of the Poincar\'e symmetric complex $C\bigcup_D \widetilde{C}$. 

The de Rham invariant is independent of the choice of the quadratic pair $(-D)\rightarrow \widetilde{C}$. Indeed, let $(-D)\rightarrow \widetilde{C}'$ be another pair. Let $I$ be the chain complex of the unit interval. Due to its dimension, the quadratic complex $(-\widetilde{C})\bigcup_{D\times 0} D\otimes I\bigcup_{D\times 1}\widetilde{C}'$ must be the boundary of some quadratic pair $W$. Then $C\otimes I\bigcup_{D\otimes I} (-W)$ is the symmetric bordism between $C\bigcup_D \widetilde{C}$ and $C\bigcup_D \widetilde{C}'$. Hence their de Rham invariants agree.

If $m=4k$, the quadratic complex $D$ is the boundary of some Poincar\'e quadratic pair $(-D)\rightarrow \widetilde{C}$. Then define the $\ZZ/8$-signature by $\Sign(C',D)=\Sign(C\bigcup_D \widetilde{C})\in\ZZ/8$. It is also invariant under different choices of $\widetilde{C}$ by a similar argument like above.

With all the invariants defined above and the long exact sequence of $L$-groups, one can prove the following.

\begin{proposition}[\cite{Ranicki1981}*{Proposition 4.3.1}]

(1) $L^q_m\cong \ZZ,\ZZ/2,0$ when $m$ is congruent to $0,2$ or odd numbers modulo $4$.

(2) $L^s_m\cong \ZZ,\ZZ/2,0$ when $m$ is congruent to $0,1$ or other numbers modulo $4$.

(3) $L^n_m\cong \ZZ/8,\ZZ/2,0,\ZZ/2$ when $m$ is congruent to $0,1,2,3$ modulo $4$.

Moreover, all the isomorphisms are explicitly given by the invariants defined above.
\end{proposition}

\subsection{\texorpdfstring{$L$}{lg} Spectra}\label{Lspectra}

As we know before, the $L$-groups are the homotopy groups of some $\LL$-spectra. We review the construction of $\LL$-spectra in this subsection.

Let $\LL^s(m)$ be the pointed $\Delta$-set whose $k$-simplices are all the $k$-ads of $(k-m)$-dimensional Poincar\'e symmetric chain complexes, with the obvious face maps. (Strictly speaking, there is a set-theoretic issue here but one can get rid of this). The base $k$-simplex is the $k$-ad of zero chains. Each $\LL^s(m)$ satisfies the Kan condition and $\LL^s(m)=\Omega \LL^s(m+1)$.
Then $(\LL^s(m))$ forms an $\Omega$-spectrum with homotopy groups isomorphic to $L^s_k$. 

The spectra $\LL^q(m)$ and $\LL^n(m)$ are similarly constructed. There is a fibration sequence of spectra $\LL^q\xrightarrow{1+T} \LL^s\xrightarrow{J} \LL^n$, whose associated long exact sequence of homotopy groups is exactly the long exact sequence of $L$-groups \eqref{lesLgroup}.

The $l$-connective cover $\LL^s\langle l\rangle$ of $\LL^s$ is constructed like follows. Let the $m$-th space $\LL^s\langle l\rangle(m)$ be the $\Delta$-set with $k$-simplicies all $k$-ads of $(k-m)$-dimensional Poincar\'e symmetric chain complexes such that the presheaf restricted to each cell of dimension less than $(m+l)$ is an ad of acyclic chains. Again, $\LL^s\langle l\rangle(m)$ is an $(l+m-1)$-connected Kan $\Delta$-set and $\LL^s\langle l\rangle(m)=\Omega \LL^s\langle l\rangle(m+1)$.

The $l$-connective spectra $\LL^q\langle l\rangle$ and $\LL^n\langle l\rangle$ are defined likewise.

In \cite[Proposition 7.1]{Ranicki1992}, Ranicki proved the equivalence of Poincar\'e quadratic complexes for a surgery problem and Wall's definition of surgery obstructions. Hence, by Quinn's construction of $G/TOP$ (\cite{Quinn1970}), there is an obvious homotopy equivalence $G/TOP\rightarrow \LL^q$.

Due to different connectiveness of $\LL^s$ and $\LL^q$, we recall the definition of ``$1/2$-connectiveness'' for normal $L$-spectrum in \cite[Deinition 15.14]{Ranicki1992} in order to get the fibration sequence of three connective spectra. 

A chain complex $C$ is $q$-connective if the truncation at degree $\leq q$ of $C$ is acyclic. A symmetic chain complex $(C,\phi)$ is $q$-Poincar\'e if $\partial C$ is $q$-connective. We can also define a presheaf of $q$-Poincar\'e symmetric chain complexes over a finite $\Delta$-set.

\begin{remark}\label{ConnectiveLemma}
In general, one can define the spectra $\LL^q(R),\LL^s(R),\LL^n(R)$ for any ring $R$. There is an analogous way to take connective covers $\LL^a\langle l\rangle(R)$ like above, for $a=q,s,n$. In \cite[p.~157]{Ranicki1992}, there is an alternative way to make $l$-connective spectra $\LL^a(\langle l\rangle,R)$, for $a=q,s,n$. Take $\LL^q(\langle l\rangle,R)$ for example. Let the $m$-th space $\LL^q(\langle l\rangle,R)(m)$ be the $\Delta$-set with $k$-simplices all $k$-ads of $(k-m)$-dimensional Poincar\'e quadratic $l$-connective chain complexes. In general, $\LL^q(\langle l\rangle,R)$ is canonically homotopy equivalent to $\LL^q\langle l\rangle(R)$, but $\LL^s(\langle l\rangle,R)$ is not homotopy equivalent $\LL^s\langle l\rangle(R)$ unless the homotopy group $\pi_*(\LL^s(R))$ has $4$-periodicity, e.g., $R=\ZZ$ (\cite[Example 15.8]{Ranicki1992}).
\end{remark}

Let $\LL^n\langle 1/2\rangle(m)$ be the $\Delta$-set whose $k$-simplices are all $k$-ads of $(k-m)$-dimensional $0$-connective $1$-Poincar\'e normal chain complex such that the presheaf restricted to each cell of dimension less than $m$ is an ad of contractible chains. For an alternative construction, one can use $k$-ads of symmetric-quadratic pairs. 

The spaces $\LL^n\langle 1/2\rangle(m)$ satisfy the Kan condition and form an $\Omega$-spectrum.

\begin{warning}
From now on, we use the symbol $\LL^q,\LL^s,\LL^n$ to represent the $1$-connective $\LL^q$-spectrum, $0$-connective $\LL^s$-spectrum and $1/2$-connective $\LL^n$-spectrum respectively.
\end{warning}

Therefore, there is a natural fiber sequence of connective spectra (\cite[Proposition 15.16(i)]{Ranicki1992})
\[\LL^q\xrightarrow{1+T} \LL^s\xrightarrow{J} \LL^n\]

Like the mock bundle picture to define the cobordism group over any space (\cite{Rourke&Sanderson}), a cellular map from a finite simplicial complex (or a $\Delta$-complex) $X$ to any of the defined $\LL$-spaces above can be considered as a `mock bundle' over $X$ with `fibers' in chain complexes, which is exactly a presheaf $\mathcal{C}$ of ads of (Poincar\'e) symmetric/quadratic/normal complexes. So the $\LL$-theory cohomology group $(\LL^a)^k(X)$ actually consists of bordism classes of presheaves for $a=s,q,n$.

\subsection{A Priori Invariants by Periods}\label{apriori}

In this subsection, we review the method for defining cohomology classes of a space $X$ by periods. We review the a priori $K$-theory invariants by periods too.

\subsubsection{\texorpdfstring{$\ZZ/n$}{lg}-Manifolds}

We follow \cite[Section 1]{Sullivan&Morgan} to review the discussions of $\ZZ/n$ manifolds. A $\ZZ/n$ manifold $M^m$ is defined by an oriented manifold with boundary $\partial M$ together with an identification of $\partial M$ to a disjoint union of $n$ copies (labelled by $1,\dots, n$) of some oriented $(m-1)$-manifold $\delta M$. $\delta M$ is also called the Bockstein of $M$. Conventionally, when $n=0$, a $\ZZ/n$ manifold means a closed manifold.

A $\ZZ/n$ manifold with boundary is an oriented manifold $N$ with boundary $\partial N$ and an embedding of a disjoint union of $n$ (labelled) copies of an oriented manifold $M'$ (with boundary $\partial M'$) into $\partial N$. Call the manifold with boundary $M'$ the Bockstein of $N$ and call the $\ZZ/n$ manifold $\partial N-\bigsqcup_n \Int(M')$ the boundary of $N$. 

Then one can define a bordism between two $\ZZ/n$ manifolds. $\Omega^{SO}_*(X;\ZZ/n)$ is the bordism group of singular $\ZZ/n$ manifolds in $X$ so that the restricted map on the Bocksteins is identical.

\begin{remark}
We can similarly define $\ZZ/n$ manifolds (with or without boundary) with other structures like $STOP,SPL,U,Spin$ and then get the corresponding bordism groups with coefficient $\ZZ/n$.
\end{remark}

Like the algebraic map $i:\ZZ/n\rightarrow \ZZ/nm$ by multiplication with $m$, one can construct $\Omega^{SO}_*(X;\ZZ/n)\rightarrow \Omega^{SO}_*(X;\ZZ/nm)$ by taking a disjoint union of $m$ copies of a $\ZZ/n$ manifold $M$ and view it as a $\ZZ/nm$ manifold. Define $\Omega^{SO}_*(X;\ZZ/p^{\infty})=\varinjlim_{k} \Omega^{SO}_*(X;\ZZ/p^k)$ for any prime number $p$.

On the other hand, the natural quotient map $\ZZ/nm\rightarrow \ZZ/n$ also induces a map $\Omega^{SO}_*(X;\ZZ/nm)\rightarrow \Omega^{SO}_*(X;\ZZ/n)$ by regrouping $mn$ copies of Bocksteins into $n$ copies.

The product of any two $\ZZ/n$ manifolds $M\times N$ might not be a $\ZZ/n$ manifold. The problem is that the neighborhood $\delta M\times \delta N$ does not look like $n$-sheets coming into the boundary. \cite{Sullivan&Morgan} made a modification like follows. 

A canonical neighborhood of each point in $\delta M\times \delta N$ is a product of a Euclidean space and a cone over $n*n$, the join between two sets of $n$ points. It is the boundary of some $2$-dimensional $\ZZ/n$ manifolds $W$, since $\Omega^{SO}_1(\pt;\ZZ/n)=0$. Then replace the tubular neighborhood of $\delta M\times \delta N$ by $\delta M\times \delta N\times W$. The resulting manifold $M\otimes N$ is a $\ZZ/n$-manifold. The $\ZZ/n$-bordism class of $M\otimes N$ is independent of the choice of $W$, since $\Omega^{SO}_2(\pt;\ZZ/n)=0$. The modified multiplication is associative up to $\ZZ/n$-bordism since $\Omega^{SO}_3(\pt;\ZZ/n)=0$. 

There is a natural map $\rho:M\otimes N\rightarrow M\times N$ which is the identity both away from some neighborhood of $\delta M\times \delta N$ and near $\delta M\times \delta N$.

\begin{lemma}[\cite{Sullivan&Morgan}*{Proposition 1.5}]
The tangent bundle $T(M\otimes N)$ is stably equivalent to $\rho^*(TM\times TN)\oplus \pi^* E$ for some vector bundle $E$ over $W/(n*n)$, where $\pi$ is the quotient map $M\otimes N\rightarrow \delta M\times \delta N\times W/(n*n)\rightarrow W/(n*n)$.
\end{lemma}

\subsubsection{Cohomological and $K$-Theoretical A Priori Invariants}

For details about the theorems in this subsection, we recommend \cite{Anderson1969}\cite{Brumfiel&Morgan}.

Let $L^{\QQ}_M=1+L^{\QQ}_4(M)+L^{\QQ}_8(M)+\cdots$ be the rational Hirzebruch $L$-genus of the tangent bundle $TM$ of a manifold $M$. The rational class has a $\ZZ_{(2)}$-lifting $L_M$, which is proved by the transversality of $\ZZ/2^k$-manifolds (\cite[Theorem 3.3]{Sullivan&Morgan}). Furthermore, the $\ZZ/2$-reduction of $L_M$ is the square of even Wu classes, namely, $V^2_M=1+v^2_{2}(M)+v^2_{4}(M)+\cdots$ (\cite[Corollary 3.2]{Sullivan&Morgan}).

\begin{remark}
The class $L$ in \cite{Sullivan&Morgan} might cause some confusions. To be clear for the readers, in  \cite{Sullivan&Morgan} the class $L\in H^*(BSO;\ZZ_{(2)})$ is the inverse of the Hirzebruch $L$-genus as stated in the \cite[Corollary 3.2]{Sullivan&Morgan}. But the paragraph above the \cite[Theorem 3.3]{Sullivan&Morgan} clarifies that when applied to manifold, the notation $L_M$ means the class $L$ for the normal bundle $\nu_M$, which is equivalent to what we used above.
\end{remark}

To define cohomology classes of $2$-primary coefficients, we need the following lemmas, which essentially follow from Thom's cobordism theory.

\begin{lemma}[\cite{Brumfiel&Morgan}*{Proposition A.3}]
Each graded class $x_*\in H^*(X;\ZZ/2)$ is uniquely determined by  a homomorphism $\sigma:\Omega^{SO}_*(X;\ZZ/2) \rightarrow \ZZ/2$ such that
\[
\sigma((M,f)\cdot N)=\sigma(M,f)\cdot \chi_2(N)\in \ZZ/2
\]
where $\chi_2$ is the mod $2$ Euler charateristic, $(M,f)\in \Omega^{SO}_*(X;\ZZ/2)$ and $N\in \Omega^{SO}_*(\pt;\ZZ/2)$.

The defining equation for $x_*$ is 
\[
\sigma(M^m,f)=\langle V^2_M\cdot f^*(\sum_{i\geq 0}x_{m-4i}),[M^m]\rangle\in \ZZ/2
\]
where $(M,f)\in\Omega^{SO}_{m}(X;\ZZ/2)$.
\end{lemma}

\begin{lemma}[\cite{Brumfiel&Morgan}*{Proposition A.11}]
Each graded class $x_*\in H^*(X;\ZZ/2^r)$ is uniquely determined by a homomorphism $\sigma:\Omega^{SO}_*(X;\ZZ/2^r) \rightarrow \ZZ/2^r$ such that

(1) 
\[
\sigma((M,f)\cdot N)=\sigma_{\QQ}(M,f)\cdot \Sign(N)\in \ZZ/2^r
\]
where $(M,f)\in \Omega^{SO}_*(X;\ZZ/2^r)$ and $N\in \Omega^{SO}_*(\pt)$;

(2) 
\[
\sigma(j_r((M,f)\cdot N))=\sigma(j_r(M,f))\cdot \Sign(N)\in \ZZ/2\xrightarrow{j_r} \ZZ/2^r
\]
where $j_r:\ZZ/2\rightarrow\ZZ/2^r$, $(M,f)\in \Omega^{SO}_*(X;\ZZ/2)$ and $N\in \Omega^{SO}_*(\pt;\ZZ/2)$;

(3) 
\[
\sigma(\delta((M,f)\cdot N))=\sigma(\delta(M,f))\cdot \Sign(N)\in \ZZ/2\xrightarrow{j_r} \ZZ/2^r
\]
where $j_r:\ZZ/2\rightarrow \ZZ/2^r$, $(M,f)\in \Omega^{SO}_*(X;\ZZ/2)$ and $N\in \Omega^{SO}_*(\pt;\ZZ/2)$

The defining equation for $x_*$ is 
\[
\sigma(M^m,f)=\langle L_M\cdot f^*(\sum_{i\geq 0}x_{m-4i}),[M^m]\rangle\in \ZZ/2^r
\]
where $(M,f)\in\Omega^{SO}_{m}(X;\ZZ/2^r)$.
\end{lemma}

\begin{lemma}[\cite{Sullivan&Morgan}*{Theorem 4.1}]
Each graded class $x_*\in H^*(X;\ZZ_{(2)})$ is uniquely determined by a commutative diagram
\[
  \begin{tikzcd}
    \Omega^{SO}_*(X)\otimes \QQ \arrow{r}{\sigma_{\QQ}} \arrow{d}{\pi} & \QQ \arrow{d}{\pi} \\
    \Omega^{SO}_*(X;\ZZ/2^{\infty}) \arrow{r}{\sigma_2} & \ZZ/2^{\infty}
  \end{tikzcd}
\]
where $\pi:\QQ\rightarrow \QQ/(\ZZ[\frac{1}{3},\frac{1}{5},\dots])\cong \ZZ/2^{\infty}$, such that

(1)
\[
\sigma_{\QQ}((M,f)\cdot N)=\sigma_{\QQ}(M,f)\cdot \Sign(N)\in \QQ
\]
where $(M,f)\in \Omega^{SO}_*(X)$ and $N\in \Omega^{SO}_*(\pt)$;

(2) 
\[
\sigma_{2}((M,f)\cdot N)=\sigma_{2}(M,f)\cdot \Sign(N)\in \ZZ/2^r\xrightarrow{j_{\infty}} \ZZ/2^{\infty}
\]
where $(M,f)\in \Omega^{SO}_*(X;\ZZ/2^r)$, $N\in \Omega^{SO}_*(\pt;\ZZ/2^r)$ and $j_{\infty}:\ZZ/2^r\rightarrow \ZZ/2^{\infty}$.

The defining equations for $x_*$ are 
\[
\sigma_{\QQ}(M^m,f)=\langle L_M\cdot f^*(\sum_{i\geq 0}x_{m-4i}),[M^m]\rangle\in \QQ
\]
and
\[\sigma_{2}(N^m,g)=\langle L_N\cdot g^*(\sum_{i\geq 0}x_{m-4i}),[N^m]\rangle\in \ZZ/2^r\xrightarrow{j_{\infty}} \ZZ/2^{\infty}
\]
where $(M,f)\in\Omega^{SO}_{m}(X)$ and $(N,g)\in\Omega^{SO}_{m}(X;\ZZ/2^r)$.
\end{lemma}

For odd primes, we need the following lemma to define elements of real $K$-theory. It is based on the Conner-Floyd theory and the Anderson duality for real $K$-theory.

\begin{lemma}[\cite{Sullivan1996}*{p.~87, Theorem 6}]
Let $X$ be a finite complex. Any commutative diagram like the following with some conditions uniquely determines an element in $KO_{(\odd)}(X)$.
\[
  \begin{tikzcd}
    \Omega^{SO}_*(X)\otimes \QQ \arrow{r}{\sigma_{\QQ}} \arrow{d}{\pi} & \QQ \arrow{d}{\pi} \\
    \Omega^{SO}_*(X;\ZZ/(\odd)) \arrow{r}{\sigma_{\odd}} & \ZZ/(\odd)
  \end{tikzcd}
\]
where $\pi:\QQ\rightarrow \QQ/(\ZZ[\frac{1}{2}])\cong \ZZ/(\odd)$ satisfying that

(1)
\[
\sigma_{\QQ}((M,f)\cdot N)=\sigma_{\QQ}(M,f)\cdot \Sign(N)\in \QQ
\]
where $(M,f)\in \Omega^{SO}_*(X)$ and $N\in \Omega^{SO}_*(\pt)$;

(2) 
\[
\sigma_{\odd}((M,f)\cdot N)=\sigma_{\odd}(M,f)\cdot \Sign(N)\in \ZZ/n\xrightarrow{j_{\infty}} \ZZ/(\odd)
\]
where $n$ is odd, $(M,f)\in \Omega^{SO}_*(X;\ZZ/n)$, $N\in \Omega^{SO}_*(\pt;\ZZ/n)$ and $j_{\infty}:\ZZ/n\rightarrow \ZZ/(\odd)$.
\end{lemma}

\subsubsection{Applications to Surgery Theory}\label{brumfielmorgan}

We sketch the proof regarding the homotopy type of $G/PL$ and the construction of Brumfiel-Morgan's characteristic classes for spherical fibrations.

Sullivan proved that the $PL$-surgery space $G/PL$ localized at prime $2$ splits as a product of Eilenberg-Maclane spaces twisted by a cohomology operation at degree $4$ (\cite{Sullivan1965}\cite[Theorem 4(i)]{Sullivan1996}). The homotopy equivalence map is induced by some characteristic classes for $PL$ surgeries \cite[p.~88]{Sullivan1996}. Later \cite{Sullivan&Rourke}\cite{Sullivan&Morgan} constructed more elaborate classes. 

For any $\ZZ/2$ manifold $M$ and any map $M\rightarrow G/PL$, there is induced a degree $1$ normal map $N\rightarrow M$. The Kervaire invariant of the map $N\rightarrow M$ induces a homomorphism $\Omega^{SO}_*(G/PL;\ZZ/2)\rightarrow \ZZ/2$.

\begin{theorem}[\cite{Sullivan1996}*{p.~88, Corollary 1}\cite{Sullivan&Rourke}*{Theorem 4.6}]
There is a graded class $k=k_2+k_6\cdots\in H^{4*+2}(G/PL;\ZZ/2)$ such that for any $\ZZ/2$ manifold $M$ and any map $f:M\rightarrow G/PL$, the Kervaire invariant for the surgery problem induced by $f$ is $\langle f^*k\cdot V^2_M,[M]\rangle$.
\end{theorem}

For any $4l$-dimensional manifold $\ZZ/2^r$ manifold $M$ and any map $f:M\rightarrow G/PL$, there is a surgery obstruction $\sigma(M,f)$ defined in $\ZZ/2^r$ for the surgery problem induced by $f$ (\cite[9.5]{Sullivan&Morgan}). Define $\sigma'(M,f)=\sigma(M,f)-\langle \beta(f^* k\cdot V_MSq^1V_M),[M]\rangle$ (\cite[p.~540]{Sullivan&Morgan}), where $\beta$ is the $\ZZ/2\rightarrow \ZZ/2^r$ Bockstein. The induced map $\Omega^{SO}_*(G/PL;\ZZ/2^{\infty})\rightarrow \ZZ/2^{\infty}$ satisfies the product formula (\cite[Theorem 8.6]{Sullivan&Morgan}).

\begin{theorem}[\cite{Sullivan&Morgan}*{Theorem 8.7}]
There is a graded class $l=l_4+l_8\cdots\in H^{4*}(G/PL;\ZZ_{(2)})$ such that for any $\ZZ/2^r$ manifold $M^{4m}$ and any map $f:M\rightarrow G/PL$, the surgery obstruction for the surgery problem induced by $f$ is $\langle f^*l\cdot L_M,[M]\rangle+\langle \beta(f^* k\cdot V_MSq^1V_M),[M]\rangle$.
\end{theorem}

To prove the splitting of $G/PL$ at prime $2$, one needs the following two lemmas.

\begin{lemma}[\cite{Cooke1968}*{p.~170 (3.7)}]
The $k$-invariant $X_{i-1}\rightarrow K(\pi_i X,i+1)$ of the Postnikov system of a simple space $X$ vanishes if and only if the Hurewicz map $h_i:\pi_i(X)\rightarrow H_i(X;\ZZ)$ is a splitting monomorphism.
\end{lemma}

\begin{lemma}[\cite{Cooke1968}*{p.~172 (3.8)}]
Let $X$ be a simple space with $\pi_i(X)\cong\ZZ$.
The order $d$ of the $k$-invariant $X_{i-1}\rightarrow K(\pi_i X,i+1)$ of the Postnikov system is the least positive integer $d$ such that there is a cohomology class in $H^i(X)$ with value $d$ on the generator of $\pi_i(X)$.
\end{lemma}

Hence, one can prove the following. For more details, we recommend \cite{Cooke1968}.

\begin{theorem}[\cite{Sullivan1996}*{Theorem 4}]
The classes $k$ and $l$ above induce a homotopy equivalence $(G/PL)_{(2)}\simeq K(\ZZ/2,2)\times_{\beta \cdot Sq^2} K(\ZZ_{(2)},4)\times \prod_{i\geq 2}(K(\ZZ/2,4i-2)\times K(\ZZ_{(2)},4i))$, where $\beta$ is the $\ZZ/2\rightarrow \ZZ_{(2)}$ Bockstein.
\end{theorem}

\begin{remark}
In fact, $(G/TOP)_{(2)}\simeq \prod_{i\geq 1}(K(\ZZ/2,4i-2)\times K(\ZZ_{(2)},4i))$ (\cite[p.~329]{Kirby&Siebenmann}). In this sense, the manifold theory for $TOP$ is simpler than $PL$.
\end{remark}

If we apply the odd-prime a priori invariant method to the space $G/PL$, then we can get the following. For more details, we recommend \cite[Chapter 4]{Madsen&Milgram}.

\begin{theorem}[\cite{Sullivan1996}*{Theorem 4}\cite{Madsen&Milgram}*{Corollary 4.31}]
There is a canonical $H$-space equivalence $(G/PL)_{(\odd)}\simeq BSO^{\otimes}_{(\odd)}$, where the $H$-space structure on $G/PL$  is induced by Whitney sums and the superscript $\otimes$ means that the $H$-space structure on $BSO$ is induced by tensor products of vector bundles.
\end{theorem}

\begin{remark}
$(G/TOP)_{(\odd)}\simeq BSO^{\otimes}_{(\odd)}$ since $G/TOP$ and $G/PL$ are only differed by a $\ZZ/2$ invariant (\cite[p.~246]{Kirby&Siebenmann}).
\end{remark}

Now we sketch the construction of Brumfiel-Morgan's characteristic classes for spherical fibrations.

Let $\nu^k$ be a spherical fibration $S(\nu)\rightarrow X$. Recall that the Thom space $Th(\nu)$ is the union of the mapping cylinder $D(\nu)$ of $\nu$ and a cone $C(\nu)$ along $S(\nu)$. Let $M^{m+k}$ be a manifold with a map $f:M\rightarrow Th(\nu)$. Call $f$ Poincar\'e transversal if $f$ is transversal to $S(\nu)$, $f^{-1}(S(\nu))\rightarrow f^{-1}(D(\nu))$ is a spherical fibration and $f$ induces a map of spherical fibrations. One can directly deduce that if $f$ is Poincar\'e transversal then $f^{-1}(D(\nu))$ is a Poincar\'e duality space of dimension $m$ with the fundamental class $[M]\cap f^*U_{\nu}$, where $U_{\nu}$ is the Thom class for $\nu$.

There is a sequence of obstructions $O_{k+n+1}\in H^{k+n+1}(M;P_n)$ for making $f$ Poincar\'e transversal. The construction is like the following. 

There is an open cover $\{U_i\}$ on $X$ such that $\nu$ over each $U_i$ has a $PL$ bundle structure. Then there is some triangulation on $M$ such that each simplex $\sigma^i$ is mapped into some $Th(\nu\vert_{U_i})$. Suppose for each simplex of dimension at most $k+n$, the restriction of $f$ is already Poincar\'e transversal.

Then for each $(k+n+1)$-dimensional simplex $\sigma^{k+n+1}$, we hope to get a value $O_{k+n+1}(\sigma^{k+n+1})\in P_n$. So the problem is reduced to the case of a map $f:D^{k+n+1}\rightarrow Th(\nu)$ such that $f\vert_{S^{k+n}}$ is Poincar\'e transversal and $\nu$ has some $PL$ bundle structure. Then we can slightly shift $f$ without changing the restriction on $f^{-1}(C(\nu))$ to some map $f'$ such that $f'$ is $PL$ transversal to the zero-section. Let $A$ be the preimage of the zero section under $f'$ and let $B=f^{-1}(D(\nu))$. Then $A$ is a $PL$ manifold with boundary and the inclusion $A\cap S^{k+n}\rightarrow B\cap S^{k+n}$ is a degree $1$ normal map. This is an element in $P_n$. We can slightly homotope $f'$ so that the union $C$ of a collar of $B\cap S^{k+n}$ and a tubular of $A$ becomes $(f')^{-1}(D(\nu))$.  If $n\geq 5$ and the surgery obstruction vanishes, then we can homotope $f'$ such that the inclusion $A\cap S^{k+n}\rightarrow B\cap S^{k+n}$ is a homotopy equivalence and $C$ becomes a Poincar\'e pair in $D^{n+k+1}$. So $f'$ on $D^{n+k+1}$ is Poincar\'e transversal. If $n\leq 4$, one can see \cite[p.~18]{Brumfiel&Morgan}.

Come back to the map $f:M^{m+k}\rightarrow Th(\nu^k)$. We can embed $M$ into some sphere $S^{N+k+m}$ for some $N>2(k+m)$ and consider the Pontryagin-Thom construction, namely, an induced map $F:S^{N+k+m}\rightarrow Th(\nu)\wedge Th(EPL(N))$, where $EPL(N)\rightarrow BPL(N)$ is the universal block $PL$ bundle of dimension $N$. The obstruction theory above gives us only one obstruction class $O\in H^{N+k+m}(S^{N+k+m};P_{m-1})\cong P_{m-1}$ for Poincar\'e transversality. The class $O$ vanishes if and only if we can cobord $(M,f)$ to some Poincar\'e transversal $(M',f')$. For more details see \cite[Section 3]{Brumfiel&Morgan}.

Then for a $\ZZ/2$ manifold $M^{4m+3+k}$ and a map $f:M\rightarrow Th(\nu^k)$, since $P_{4m+1}=0$, we can always assume $f$ is homotopy transversal on $\delta M$. Then the obstruction for cobording $(M,f)$ to some Poincar\'e transversal $(M',f')$ is an element in $P_{4m+2}=\ZZ/2$. This is a homomorphism $\Omega_{4m+3}(Th(\nu);\ZZ/2)\rightarrow \ZZ/2$. One can check that it satisfies the product formula. By the Thom isomorphism for the spherical fibration $\nu^k$ one gets the following.

\begin{theorem}[\cite{Brumfiel&Morgan}*{Theorem 5.4}]
There exists a graded characteristic class $k^G\in H^{4*+3}(X;\ZZ/2)$ for a spherical fibration $\nu^k$ on $X$ such that for any $\ZZ/2$ manifold $M^{4m+3+k}$ and any map $f:M\rightarrow Th(\nu)$ the obstruction for $f$ being Poincar\'e transversal by a bordism is $\langle f^*(k^G\cdot U_{\nu})\cdot V^2_M,[M]\rangle\in \ZZ/2$.
\end{theorem}

To define the $\ZZ/8$ characteristic class for $\nu^k$, consider a $\ZZ/8$ manifold $M^{4m+k}$ and a map $f:M\rightarrow Th(\nu)$. Since $P_{4m-1}=0$, we may assume $f$ is Poincar\'e transversal on $\delta M$.

Let us first consider the case that $f$ is Poincar\'e transversal. Define the valuation on $(M,f)$ by the signature of the Poincar\'e space $f^{-1}(D(\nu))$ in $\ZZ/8$. For the general case, \cite[p.~61]{Brumfiel&Morgan} proved that there exists a map $a:K^{k+4}\rightarrow MSG(k)$ for a $\ZZ/2$ manifold
\[
K^{k+4}=S^{k+3}\times I/ (x,0)\sim (-x,1)
\]
such that

(1) the Kervaire obstruction to the Poincar\'e transversality of $a\vert_{\delta K}$ is $1\in\ZZ/2$;

(2) $\langle a^*(V^2\cdot U),[K]\rangle=0\in\ZZ/2$. 

Then consider the bordism class $[M,f]-j_8[K,a]$, where $j_8:\ZZ/2\rightarrow \ZZ/8$, it is Poincar\'e transversal
by some bordism. So we can apply the previously defined valuation on $[M,f]-j_8[K,a]$. Modify this valuation by minus $j_8\langle f^*(k^G\cdot U)\cdot V_MSq^1V_M,[M]\rangle$. Chenck the product formulae for this valuation (\cite[Theorem 8.4]{Brumfiel&Morgan}). Hence,

\begin{theorem}[\cite{Brumfiel&Morgan}*{8.1}]
There exists a graded characteristic class $l^G\in H^{4*}(X;\ZZ/8)$ for a spherical fibration $\nu^k$ on $X$ satisfying the following properties:

(1) $\rho_2(l^G)=V^2$, where $\rho_2:\ZZ/8\rightarrow \ZZ/2$;

(2) $\beta l^G$ is the obstruction class for Poincar\'e transversalities, where $\beta$ is the $\ZZ/8\rightarrow \ZZ_{(2)}$, namely, for any $\ZZ/2^r$ manifold $M^{4m+1+k}$ and any map $f:M\rightarrow Th(\nu)$ the obstruction $f$ being Poincar\'e transversal by a bordism is $\langle f^*(\beta l^G\cdot U_{\nu})\cdot L_M,[M]\rangle+j_r\langle f^*(k^G\cdot U_\nu)\cdot V_{\delta M}Sq^1V_{\delta M},[\delta M]\rangle\in \ZZ/2^r$, where $j_r:\ZZ/2\rightarrow \ZZ/2^r$.
\end{theorem}

We have to mention that the following theorem is essentially due to Morgan-Sullivan, though their result only states for $PL$ bundles. However, if one knows some transversality theorem for $TOP$ manifolds and $TOP$ bundles, then their result generalizes to $TOP$ bundles.

\begin{theorem}[\cite{Sullivan&Morgan}*{p.~530}]
The $\ZZ_{(2)}$-class $L$ for vector bundles defined above, whose rationalization is the inverse Hirzebruch $L$-genus, can be lifted a $\ZZ_{(2)}$-class $l^{TOP}$ for $TOP$ bundles.
\end{theorem}

The Winkelnkemper's axiom I states that transversality unlocks the secret of manifolds. One can define a $TOP$ transversality theory on a spherical fibration $\nu$ by an assignment to each singular simplex $f:\sigma^i\rightarrow Th(\nu)$ a deformation of $f$ till $f$ is Poincar\'e tranversal in the interior and in each face. Then one can define a concordance equivalence relation between $TOP$ transversality theories. Indeed, the equivalence classes of all $TOP$ bundle structures on $\nu$ is in one-to-one correspondence to the concordance classes of transversality theories on $\nu$ (\cite{Levitt&Morgan}, \cite[Theorem E]{Brumfiel&Morgan}, \cite[Theorem 1.11]{Levitt&Ranicki}). Since $k^G$ and $\beta l^G$ are the obstructions for existence of $TOP$ transversality theory on $\nu$ in the $2$-local sense, so we have that

\begin{theorem}
$k^G$ and $\beta l^G$ are the obstructions for a spherical fibration $\nu$ to have a $TOP$ bundle structure in the $2$-local sense.
\end{theorem}

\begin{remark}
If $\nu$ has a $TOP$ bundle structure, the $\ZZ/8$ reduction of $l^{TOP}(\nu)$ is exactly $l^G(\nu)$.
\end{remark}

By applying the odd-prime period method and the correspondence between transversalities and bundle structures, one can also get that

\begin{theorem}[\cite{Sullivan2009}*{Theorem 6.5}]
The existence of a $KO_{(\odd)}$-orientation is the obstruction for a spherical fibration $\nu$ to have a $TOP$ bundle structure in the odd-prime local sense.
\end{theorem}

\section{\texorpdfstring{$L$}{Lg}-theory with \texorpdfstring{$\ZZ/n$}{Lg} coefficient}\label{Znchain}

We generalize the idea of $\ZZ/n$ manifolds to the chain-level. We put much effort in constructing the bordism invariants for $\ZZ/n$ manifolds and proving the product formulae for these invariants.

\subsection{\texorpdfstring{$\ZZ/n$}{lg} Chains and their bordisms}

Like $\ZZ/n$ manifolds, we can define $\ZZ/n$ symmetric, quadratic or normal chains. The definitions for three cases are similar so we only discuss the symmetric case here. We use the notation $n C$ for the direct sum of $n$ (labelled) copies of a chain complex $C$.

\begin{definition}
An $m$-dimensional $\ZZ/n$ symmetric chain complex $C$ is an $m$-dimensional symmetric pair $n\delta C \rightarrow C$, $n\delta C$ is labelled by $1,\dots,n$. 

$\delta C$ is also called the Bockstein of $C$.
\end{definition}

\begin{definition}
A $\ZZ/n$ symmetric complex $C$ is Poincar\'e if both the chain $\delta C$ and the pair $n\delta C \rightarrow C$ are Poincar\'e.    
\end{definition}

\begin{definition}
A $\ZZ/n$ symmetric chain pair $D\rightarrow C$ consists of symmetric pairs $(-\delta D)\rightarrow \delta C$, $n\delta D\rightarrow D$ and $D\bigcup_{n\delta D}n\delta C\rightarrow C$. It is a $\ZZ/n$ Poincar\'e pair if $n\delta D\rightarrow D$ is a $\ZZ/n$ Poincar\'e symmetric chain and $(-\delta D)\rightarrow \delta C$, $D\bigcup_{n\delta D}n\delta C\rightarrow C$ are both Poincar\'e pairs.
\end{definition}

\begin{remark}
We only use the symmetric-quadratic pair description for $\ZZ/n$ normal complexes. A $\ZZ/n$ Poincar\'e symmetric-quadratic chain pair of dimension $m$ consists of a $\ZZ/n$ Poincar\'e quadratic chain complex $n\delta E\rightarrow E$ of dimension $m-1$, a Poincar\'e symmetric chain pair $(-\delta E)\rightarrow \delta D$ of dimension $m-1$ and a Poincar\'e symmetric chain pair $E\bigcup_{n\delta E}n\delta D\rightarrow D$ of dimension $m$. In abbreviation, we use the symbol $(D,E)$ instead of writing all the maps like above.  
\end{remark}

Then one can define a bordism between two $\ZZ/n$ (Poincar\'e) symmetric chains. One can also define a $k$-ad of $\ZZ/n$ (Poincar\'e) symmetric chains.

\begin{definition}
Let $L^s_m(\ZZ,\ZZ/n)$ be the set of $m$-dimensional bordism classes of $\ZZ/n$ Poincar\'e symmetric chain complexes. 
\end{definition}

Like $L^s_m$, the $\ZZ/n$ $L$-group $L^s_m(\ZZ,\ZZ/n)$ is also an abelian group with the additive structure by the direct sum of chains. Analogously, one can also define the abelian groups $L^q_m(\ZZ,\ZZ/n)$ and $L^n_m(\ZZ,\ZZ/n)$.

Like the long exact sequence of $L$-groups in the last section, we also have

\begin{proposition}
There is a long exact sequence
\[
\cdots\rightarrow L^q_m(\ZZ,\ZZ/n)\rightarrow L^s_m(\ZZ,\ZZ/n)\rightarrow L^n_m(\ZZ,\ZZ/n)\rightarrow L^q_{m-1}(\ZZ,\ZZ/n)\rightarrow \cdots
\]    
\end{proposition}

Like the $\ZZ/n$ manifold case, there are also natural maps $L^s(\ZZ,\ZZ/n)\rightarrow L^s(\ZZ,\ZZ/nm)$ and $L^s(\ZZ,\ZZ/nm)\rightarrow L^s(\ZZ,\ZZ/n)$.

The Bockstein of $\ZZ/n$ chains induces a natural map $L^s_m(\ZZ,\ZZ/n)\xrightarrow{\delta} L^s_{m-1}$.

\begin{proposition}\label{BocksteinLES}
There is a long exact sequence
\[
\cdots\rightarrow L^s_m\xrightarrow{\times n} L^s_m\xrightarrow{\iota} L^s_m(\ZZ,\ZZ/n)\xrightarrow{\delta} L^s_{m-1}\rightarrow \cdots
\]
\end{proposition}

\begin{proof}
$\Im(\delta)=\Ker(\times n)$, $\Ker(\iota)\subset \Im(\times n)$ and $\Im(\iota)\subset \Ker(\delta)$ are obvious. Let us check the rest.

$\Ker(\delta)\subset \Im(\iota)$: Take $C\in \Ker(\delta)$. Then there is a Poincar\'e symmetric chain pair $-\delta C\rightarrow D$. Then $\iota(C\bigcup_{n\delta_C}nD)=C$ since the  $C\bigcup_{n\delta_C}nD\times I$ gives the $\ZZ/n$ bordism between $C$ and $C\bigcup_{n\delta_C}nD$, where $I$ is the cellular chain of a unit interval.

$\Im(\times n)\subset \Ker(\iota)$: Take a $C\in L^s_m$. $nC\times I$ is a $\ZZ/n$ chain with boundary $nC$. So $\iota(nC)=0$.
\end{proof}

By a direct calculation, the followings are immediate.

\begin{proposition} 
(1) $L^s_m(\ZZ,\ZZ/2^k) \cong \ZZ/2^k,\ZZ/2,\ZZ/2,0$ when $m$ is congruent to $0,1,2,3$ modulo $4$.

(2) Let $p$ be an odd prime. $L^s_m(\ZZ,\ZZ/p^k) \cong \ZZ/2^k,0$ when $m$ is congruent to $0$ or other numbers modulo $4$.

(3) The isomorphisms are given by the bordism invariants defined below.
\end{proposition}

To define the bordism invariants, let us only focus on the case when $n=2^k$, since the odd $n$ case is almost the same and much easier.

The signature of a $\ZZ/n$ manifold is defined to be the signature modulo $n$. It is well defined because of Novikov's additivity theorem. It is also a bordism invariant for the same reason. Likewise, define the signature invariant $\sigma^s_0=\Sign \in\ZZ/n$ for $\ZZ/n$ Poincar\'e symmetric chains. 

For a $4l+2$ dimensional $\ZZ/2^k$ Poincar\'e symmetric chain complex $C$, define the bordism invariant $\sigma^s_2(C)=\dR(\delta C)\in\ZZ/2$.

The dimension $4l+1$ case is a bit more complicated. \cite{Sullivan&Morgan} defined the de Rham invariant for $4l+1$ dimensional $\ZZ/n$ oriented manifolds. Here we use the same idea .

Let $C$ be a $4l+1$ dimensional $\ZZ/2^k$ Poincar\'e symmetric chain complex. Let $A$ be the self-annihilating subspace of $H_{2l}(\delta C)$, which has half of the total rank. Let $T$ be the torsion of $H_{2l}(\delta C)$. Let $T_{C,A}$ be the torsion of $H_{2l}(C)/f(2^k\cdot(A+T))$. Define the de Rham invariant by $\dR(C)=\rank_{\ZZ/2}(T_{C,A}\otimes \ZZ/2)\in \ZZ/2$.

There is an alternative way to define de Rham invariant. Since $2^k$ copies of $\delta C$ is a boundary, the signature of $\delta C$ vanishes. So it is the boundary of some Poincar\'e symmetric chain pair, namely, $(-\delta C)\rightarrow C'$. Then define $\sigma^s_1(C)=\dR(C\bigcup_{2^k \delta C}2^k C')\in\ZZ/2$.

The second definition is independent of the choice of $C'$. Indeed, let $C''$ be another choice and then $2^k((-C')\bigcup_{\delta C\otimes 0} \delta C\otimes I\bigcup_{\delta C\otimes 1} C'')$ must be the boundary of some Poincar\'e symmetric pair $W$ since its de Rham invariant must vanish. Then $C\otimes I\bigcup_{2^k \delta C\otimes I}(-W)$ is a bordism between $C\bigcup_{2^k\delta C}2^k C'$ and $C\bigcup_{2^k\delta C}2^k C''$. So they have the same de Rham invariant.

$\sigma^s_1$ is a bordism invariant by a similar argument.

\begin{warning}
Extend $\sigma^s_m$ by $0$ when the chain complex does not have the corresponding dimension. This rule applies to other bordism invariants.
\end{warning}

Analogously, we have the followings.

\begin{proposition}
There is a long exact sequence
\[
\cdots\rightarrow L^q_m\xrightarrow{\times n} L^q_m\xrightarrow{\iota} L^q_m(\ZZ,\ZZ/n)\xrightarrow{\delta} L^q_{m-1}\rightarrow \cdots
\]
\end{proposition}

\begin{proposition}
There is a long exact sequence
\[
\cdots\rightarrow L^n_m\xrightarrow{\times n} L^n_m\xrightarrow{\iota} L^n_m(\ZZ,\ZZ/n)\xrightarrow{\delta} L^n_{m-1}\rightarrow \cdots
\]
\end{proposition}

Again by direct computations, we have

\begin{proposition}
(1) $L^q_m(\ZZ,\ZZ/2^k)\cong  \ZZ/2^k,0,\ZZ/2,\ZZ/2$ when $m$ is congruent to $0,1,2,3$ modulo $4$.

(2) Let $p$ be an odd prime. $L^q_m(\ZZ,\ZZ/p^k) \cong \ZZ/p^k,0$ when $m$ is congruent to $0$ or other numbers modulo $4$.

(3) The isomorphisms are given by the bordism invariants defined below.
\end{proposition}

The bordism invariants for the quadratic case is like the symmetric case. Let $C$ be a $\ZZ/2^k$ Poincar\'e quadratic chain complex. 

If $C$ has dimension $4l$, then $\delta C$ has dimension $4l-1$ and must be the boundary of some Poincar\'e quadratic chain pair, namely, $(-\delta C)\rightarrow C'$. Then $C\bigcup_{2^k \delta C}2^k C'$ is a Poincar\'e quadratic chain and hence its signature is divisible by $8$. Define $\sigma^q_0(C)=\frac{1}{8}\cdot \Sign(C\bigcup_{2^k \delta C}2^k C')\in\ZZ/2^k$. 

If $C$ has dimension $4l+2$, then the dimension argument also shows that $\delta C$ must be the boundary of some quadratic Poincar\'e pair, namely, $(-\delta C)\rightarrow C'$. Define $\sigma^q_2(C)=K(C\bigcup_{2^k\delta C}2^k C')\in\ZZ/2$.

If $C$ has dimension $4l+3$, define $\sigma^q_3(C)=K(\delta C)\in\ZZ/2$.

\begin{proposition}
(1) If $k=1,2$, then $L^n_m(\ZZ,\ZZ/2^k) \cong \ZZ/2^k\oplus \ZZ/2, \ZZ/2\oplus \ZZ/2^k, \ZZ/2, \ZZ/2$ when $m$ is congruent to $0,1,2,3$ modulo $4$.

If $k\geq 3$, then $L^n_m(\ZZ,\ZZ/2^k) \cong \ZZ/8\oplus \ZZ/2, \ZZ/2\oplus \ZZ/2^k, \ZZ/2, \ZZ/2$ when $m$ is congruent to $0,1,2,3$ modulo $4$.

(2) Let $p$ be an odd prime. Then $L^n_m(\ZZ,\ZZ/p^k)=0$.

(3) The isomorphisms are induced by the bordism invariants defined in the next section.
\end{proposition}

\begin{proof}
The result follows from the diagram, where each horizontal and vertical line is an exact sequence.
\[
\begin{tikzcd}
 & \cdots \arrow[d]  & \cdots \arrow[d] & \cdots \arrow[d] & \\
\cdots \arrow[r] & L^s_m \arrow[d] \arrow[r]  & L^s_m \arrow[r]  \arrow[d] & L^s_m(\ZZ,\ZZ/n) \arrow[d] \arrow[r] & \cdots  \\
\cdots \arrow[r] & L^q_m \arrow[d] \arrow[r]  & L^q_m \arrow[r]  \arrow[d] & L^q_m(\ZZ,\ZZ/n) \arrow[d] \arrow[r] & \cdots \\
\cdots \arrow[r] & L^n_m \arrow[d] \arrow[r]  & L^n_m \arrow[r]  \arrow[d] & L^n_m(\ZZ,\ZZ/n) \arrow[d] \arrow[r] & \cdots \\
& \cdots & \cdots & \cdots & 
\end{tikzcd}
\]
\end{proof}

\begin{remark}
All the direct sum decompositions above are unnatural. However, we will find a way for an explicit decomposition in the last part of this section.
\end{remark}

A $\ZZ/n$ (Poincar\'e) symmetric presheaf on a finite $\Delta$-set is a  (Poincar\'e) presheaf of ads of  $\ZZ/n$ symmetric chain complexes, like presheaf of ads of symmetric chain complexes. Later on, we might discard the word `ads' for simplicity and only use the phrase `a presheaf of $\ZZ/n$ (Poincar\'e)  symmetric chain complexes'.

Like the lemma \ref{AssemblyLemma}, the following two lemmas are immediate.

\begin{lemma}
Let $X$ be a closed $n$-dimensional $P$L manifold with a $PL$ triangulation. Let $\mathcal{C}$ be an $m$-dimensional (Poincar\'e) presheaf of $\ZZ/k$ symmetric/quadratic/normal chain complexes over $X$. Then the assembly $\mathcal{C}(X)$ is also a $\ZZ/k$ (Poincar\'e) $(n+m)$-dimensional symmetric/quadratic/normal chain complex.
\end{lemma}

\begin{lemma}
Let $X$ be an oriented $n$-dimensional $\ZZ/k$ $PL$ manifold with a $PL$ triangulation. Let $\mathcal{C}$ be an $m$-dimensional (Poincar\'e) presheaf of symmetric/quadratic/normal chain complexes over $X$. Then the assembly $\mathcal{C}(X)$ is also a $\ZZ/k$ (Poincar\'e) $(n+m)$-dimensional symmetric/quadratic/normal chain complex.
\end{lemma}

\subsection{Product of \texorpdfstring{$\ZZ/n$}{lg} Chains}

We can define the product structure for $\ZZ/n$ Poincar\'e symmetric/quadratic/normal chains with the same idea for defining products of $\ZZ/n$ manifolds. Like the previous section, we only foucs on the symmetric case. The quadratic and normal cases are analogous.

Let $C$ and $D$ be two $\ZZ/n$ Poincar\'e symmetric complexes.

Let $I$ be the chain complex of the unit interval. First, take $(C\otimes n\delta D\otimes I) \bigcup (C\otimes D)\bigcup (n\delta C\otimes D\otimes I)$, where the union is via an identification of $C\otimes n\delta D\otimes {0}$ in $C\otimes n\delta D\otimes I$ with $C\otimes n\delta D$ in $C\otimes D$ and by a similar identification for another part. Then the `singularity' part in the boundary is $n\delta C\otimes n\delta D\otimes I\bigcup n\delta C\otimes n\delta D\otimes I\cong \delta C\otimes \delta D\otimes (n^2I\bigcup n^2I)$, where $n^2I\bigcup n^2I$ is the chain complex of the $\ZZ/n$ manifold $n*n$. Since $n*n$ bounds a $\ZZ/n$ 2-dimensional manifold. Let $W$ be the corresponding cellular chain complex. Define the product of $C$ and $D$ by 
\begin{equation}\label{modifiedproduct}
C\otimes n\delta D\otimes I \bigcup C\otimes D\bigcup n\delta C\otimes D\otimes I\bigcup \delta C\otimes \delta D \otimes W  
\end{equation}

By abuse of notations, we still use $C\otimes D$ to represent their modified product. The existence of such a modification is due to the fact $\Omega_1^{SO}(\pt;\ZZ/n)=0$. The $\ZZ/n$ bordism class of $C\otimes D$ does not depend on the choice of $W$, since $\Omega_2^{SO}(\pt;\ZZ/n)=0$. The modified product is associative up to bordism since $\Omega_3^{SO}(\pt;\ZZ/n)=0$.

\begin{remark}
In fact, the modified product of two $\ZZ/n$ manifold is like our modified chain-level products. That is, the modified product $M\otimes N$ is diffeomorphic to the union
\[
(M\otimes \bigsqcup_n\delta N\times I )\bigcup (M\times N)\bigcup (\bigsqcup_n \delta M\otimes N\times I)\bigcup (\delta M\times \delta N \times W) 
\]
\end{remark}

Consider the graph $n*n$, namely, there are two sets of $n$ points $\{i_1,\cdots, i_n\}$ and $\{j_1,\cdots, j_n\}$ and there is a path $(s,t)$ connecting each pair of vertices $i_s,j_t$. A choice of $W$ gives a permutation $\gamma$ of $\{1,\dots,n\}$. It adds for each $s$ a path $p_s$ into $n*n$ connecting the vertices $i_s,j_{\gamma(s)}$. Call the new graph $\gamma'$. $W$ also regroup the original paths $(s,t)$ of $\gamma$ into $n$ classes $\{(i_{s_k},j_{t_k})\}_{1\leq k\leq n}$, where each vertex $i_t$ and $j_s$ appear only once in the set of vertices of each group of paths, so that each group of paths together with $\{p_s\}_{1\leq s\leq n}$ gives a loop in the new graph. In this way, we csee what $W$ does for  $\delta (C\otimes D)$. $W$ connects the $t$-th copy $\delta C\otimes D$ with $\gamma(t)$-th copy $C\otimes \delta D$ for $1\leq t\leq n$. Considering $\delta(C\otimes D)\otimes I$ we get that

\begin{lemma}\label{boundary}
$\delta (C\otimes D)$ is bordant to $(\delta C\otimes D)\oplus (C\otimes \delta D)$, as $\ZZ/n$ Poincar\'e symmetric chains.
\end{lemma}

For later uses, let us further consider a product of presheaves over two $\ZZ/n$ manifolds. Let $\mathcal{C}$ and $\mathcal{D}$ be presheaves of symmetric chain complexes over two oriented $\ZZ/n$ PL manifolds $M$ and $N$ respectively. 

By adding a collar neighborhood of Bockstein $\delta M$ to $M$, we may assume the presheaf $\mathcal{C}$ on a collar neighborhood of a cell $\delta\sigma$ of $\delta M$ is $\mathcal{C}(\delta\sigma)\otimes C_*(c(n))$, where $c(n)$ is the cone over $n$ points. It is similar for $\mathcal{D}$ over $N$.

$M\times N$ has a natural regular cell decomposition inherited from the cell decompositions of $M$ and $N$, namely, each cell in $M\times N$ has the form $\sigma\times\tau$ for some cell $\sigma$ of $M$ and some cell $\tau$ of $N$. Then it is natural to define the presheaf $\mathcal{C}\times \mathcal{D}$ over $M\times N$ by $(\mathcal{C}\times\mathcal{D})(\sigma\times\tau)=\mathcal{C}(\sigma)\otimes\mathcal{D}(\tau)$.

Consider the modified product $M\otimes N$. For any cell $\delta\sigma\times\delta\tau$ of $\delta M\times \delta N$, we just replace its tubular neighborhood $\delta\sigma\times\delta\tau$ in $M\times N$ by $\delta\sigma\times\delta\tau\times c(n*n)$ by $\delta\sigma\times\delta\tau\times W$, where $c(n*n)$ is the cone over $n*n$. On the chain level, we define $\mathcal{C}\otimes \mathcal{D}$ by the same value as $\mathcal{C}\times\mathcal{D}$ except $(\mathcal{C}\otimes \mathcal{D})(\delta\sigma\times\delta\tau\times W)=\mathcal{C}(\delta\sigma)\otimes \mathcal{D}(\delta\tau)\otimes C_*(W)$.

Comparing products of presheaves over modified products of manifolds and modified products of two symmetric chains, it is obvious that

\begin{proposition}\label{ProducAssembly}
The chain-level modified product of the assemblies $\mathcal{C}(M)\otimes \mathcal{D}(N)$ is isomorphic to the assembly of the presheaf over the modified product of manifolds $(\mathcal{C}\otimes \mathcal{D})(M\otimes N)$.
\end{proposition}

\begin{proposition}
The Bockstein of the assembly of the modified product presheaf $\delta(\mathcal{C}\otimes \mathcal{D})(M\otimes N)$ is bordant to $(\delta(\mathcal{C}(M))\otimes \mathcal{D}(N))\oplus(\mathcal{C}(M)\otimes \delta(\mathcal{D}(N)))$, as $\ZZ/n$ Poincar\'e symmetric chains.
\end{proposition}

\subsection{Product Structure of \texorpdfstring{$L$}{Lg}-Theory with \texorpdfstring{$\ZZ/n$}{Lg} coefficient}

We only consider the case $n=2^k$ in this section, since the odd $n$ case is much easier.

Recall that tensor products of chains induce the ring structures on $L^s_*,L^n_*$ and the $L^s_*$-module structure on $L^q_*$. The map $L^s_*\rightarrow L^n_*$ also induces an $L^s_*$-algebra structure on $L^n_*$.

The product structures on the $L$-groups can be lifted to the spectra level. We recall the construction in \cite[Appendix B]{Ranicki1992} here. 

The simplicial approximation of the diagonal map of $\Delta^n$ gives a $\Delta$-set structure on $\Delta^n$, that is, a $(k+l)$-cell $\sigma_{t_0,t_1,\cdots,t_k,s_0,s_1,\cdots,s_l}$ is indexed by each sequence $0\leq t_0< t_1<\cdots t_k\leq s_0<s_1<\cdots s_l\leq n$.

Let $\mathcal{C}$ and $\mathcal{D}$ be two $n$-ads of chains. Then define a presheaf $\mathcal{C}\otimes\mathcal{D}$ of chains over $\Delta^n$ by $(\mathcal{C}\otimes\mathcal{D})(\sigma_{t_0,\cdots,t_k,s_0,\cdots,s_l})=\mathcal{C}(\Delta^{t_0,\cdots,t_k})\otimes \mathcal{D}(\Delta^{s_0,\cdots,s_l})$.
After taking a union, $\mathcal{C}\otimes\mathcal{D}$ is indeed an $n$-ad of chains.

All these algebraic structures can be defined for the $L$-groups with $\ZZ/n$-coefficient, by the modified products constructed in the previous section. Moreover, the natural map $L^a_*\rightarrow L^a_*(\ZZ,\ZZ/n)$ is a ring map when $a=s,n$ and an $L^s$-module map when $a=q$.

We will calculate the explicit product structures for these $L$-groups with $\ZZ/n$-coefficient. 

\subsubsection{Symmetric and Quadratic Case}

Firstly, like manifolds, the signature of chains is multiplicative.

\begin{lemma}
Let $C$ and $D$ be two Poincar\'e symmetric pairs, then 
\[
\Sign(C\otimes D)=\Sign(C)\cdot \Sign(D)
\]
\end{lemma}

Hence, the multiplicative structure on $L^s_*$ induces the following isomorphisms.

\begin{proposition}
\[
L^s_{4s}\otimes L^s_{4t}\xrightarrow{\cong} L^s_{4(s+t)}
\]
\[
L^s_{4s}\otimes L^q_{4t}\xrightarrow{\cong} L^q_{4(s+t)}
\]
\[
L^n_{4s}\otimes L^n_{4t}\xrightarrow{\cong} L^n_{4(s+t)}
\]
\end{proposition}

Using the long exact sequence \ref{BocksteinLES} and the calculation of $L^s_*(\ZZ,\ZZ/2^k)$, we get that 

\begin{proposition}
\[
L^s_{4s}(\ZZ,\ZZ/2^k)\otimes L^s_{4t}(\ZZ,\ZZ/2^k)\xrightarrow{\cong} L^s_{4(s+t)}(\ZZ,\ZZ/2^k)
\]
\[
L^s_{4s}(\ZZ,\ZZ/2^k)\otimes L^q_{4t}(\ZZ,\ZZ/2^k)\xrightarrow{\cong} L^q_{4(s+t)}(\ZZ,\ZZ/2^k)
\]
\end{proposition}

\begin{lemma}
Let $C$ and $D$ be $\ZZ/2^k$  Poincar\'e symmetric chain complexes. Then
\[
\sigma^s_0(C\otimes D)=\sigma^s_0(C)\cdot \sigma^s_0(D)\in\ZZ/2^k
\]
\end{lemma}

The argument in \cite[Chapter 6]{Sullivan&Morgan} also proves that

\begin{proposition}
\[
L^s_{4s+1}\otimes L^s_{4t}\xrightarrow{\cong} L^s_{4(s+t)+1}
\]
\[
L^s_{4s+1}(\ZZ,\ZZ/2^k)\otimes L^s_{4t}(\ZZ,\ZZ/2^k)\xrightarrow{\cong} L^s_{4(s+t)+1}(\ZZ,\ZZ/2^k)   
\]
\end{proposition}

It follows that

\begin{lemma}
Let $C$ and $D$ be $\ZZ/2^k$ Poincar\'e symmetric chain complexes. Then
\[
\sigma^s_1 (C\otimes D)=\sigma^s_1(C)\cdot \sigma^s_0(D)+\sigma^s_0(C)\cdot \sigma^s_1(D)\in\ZZ/2
\]
\end{lemma}

Analogously,

\begin{proposition}
\[
L^q_{4s+2}\otimes L^s_{4t}\xrightarrow{\cong} L^q_{4(s+t)+2}
\]
\[
L^q_{4s+2}(\ZZ,\ZZ/2^k)\otimes L^s_{4t}(\ZZ,\ZZ/2^k)\xrightarrow{\cong} L^q_{4(s+t)+2}(\ZZ,\ZZ/2^k)   
\]
\end{proposition}

\begin{lemma}
\label{product-quadratic-Kervaire}
Let $C$ be a $\ZZ/2^k$  Poincar\'e symmetric chain complex and let $D$ be a $\ZZ/2^k$ Poincar\'e  quadratic chain complex. Then
\[
\sigma^q_2(C\otimes D)=\sigma^s_0(C)\cdot \sigma^q_2(D)\in\ZZ/2
\]
\end{lemma}

The following technical point is essentially \cite[Theorem 6.1]{Sullivan&Morgan}, which is crucial for the rest discussions.

\begin{lemma}\label{Keyproduct}
Let $S$ and $Q$ be generators of $L^s_{4s+1}$ and $L^q_{4t+2}$. Suppose $S\otimes Q$ is the boundary of some Poincar\'e quadratic pair $S\otimes Q\rightarrow W$. Then $\Sign(W)= 4 \in\ZZ/8$.
\end{lemma}

The proof is based on the following lemma, which is essentially equivalent to \cite[Theorem 5.8]{Sullivan&Morgan}. Let $C$ be a $(4k-1)$ dimensional Poincar\'e quadratic chain complex. Then $C$ is equivalent to a linking form $q$ (valued in $\QQ/\ZZ$) on the torsion group $TH_{2k-1}(C)$. Suppose $C$ is the boundary of some Poincar\'e quadratic pair $C\rightarrow D$.

\begin{lemma}
$\sqrt{|TH_{2k-1}(C)|}\cdot e^{\frac{\pi i}{4}\Sign(D)}=\sum_{x\in TH_{2k-1}(C)} e^{2\pi i q(x)}$.
\end{lemma}

\begin{proof}[Sketch Proof of Lemma \ref{Keyproduct}]
There is a correspondence between bordism classes of Poincar\'e symmetric or quadratic chains and the Witt group of nonsingular (skew-)symmetric or (skew-)quadratic forms on abelian groups (which are free finitely generated or finite abelian group respectively) (\cite[Proposition 4.2.1]{Ranicki1981}).

Hence, $S$ is equivalent to a symmetric linking form $l_1:\ZZ/2\times\ZZ/2\rightarrow \QQ/\ZZ$ and $Q$ is equivalent to a symplectic form $l_2$ on $\ZZ/2\oplus\ZZ/2$ together with a quadratic form $q_2$ so that $q_2(1,0)=q_2(0,1)=q_2(1,1)=\frac{1}{2}\in\ZZ/2\subset \QQ/\ZZ$. Then $S\otimes Q$ is equivalent to the quadratic form $q$ on $\ZZ/2\oplus\ZZ/2$ with $q=q_2$.

Applying the previous lemma, we get that $\Sign(W)=4\in\ZZ/8$ by a direct calulation.
\end{proof}

\begin{remark}
If one wants to see a more elementary proof, we recommend the \cite[Theorem 5.9]{Sullivan&Morgan} and the example in \cite[p.~508]{Sullivan&Morgan}.
\end{remark}

Hence, we also have the followings.

\begin{corollary}
Let $S$ and $Q$ be the generators of $L^s_{4s+1}$ and $L^q_{4t+2}$ respectively. Suppose $2^k$ copies of $S\otimes Q$ is the boundary of the Poincar\'e quadratic pair $S\otimes Q\rightarrow W_k$. Then $\sigma^q_0(W_k)=2^{k-1}\in \ZZ/2^k$.
\end{corollary}

\begin{lemma}
Let $C$ be a $\ZZ/2^k$ Poincar\'e quadratic chain complex and $D$ be a $\ZZ/2^k$ Poincar\'e symmetric chain complex. Then 
\[
\sigma^q_0(C\otimes D)=\sigma^q_0(C)\cdot \sigma^s_0(D)+j_{2^k}(\sigma^q_2(\delta C)\cdot \sigma^s_1(D)+\sigma^q_2(C)\cdot \sigma^s_1(\delta D))\in\ZZ/2^{k}
\]
where $j_{2^k}:\ZZ/2\rightarrow\ZZ/2^k$.
\end{lemma}

\begin{proposition} 
The homomorphisms
\[
L^s_{4s+1}\otimes L^q_{4t+3}(\ZZ,\ZZ/2^k)\xrightarrow{\cong} L^s_{4s+1}(\ZZ,\ZZ/2^k)\otimes L^q_{4t+3}(\ZZ,\ZZ/2^k)\rightarrow L^q_{4(s+t)+4}(\ZZ,\ZZ/2^k)
\]
\[
L^s_{4s+2}(\ZZ,\ZZ/2^k)\otimes L^q_{4t+2}\xrightarrow{\cong}L^s_{4s+2}(\ZZ,\ZZ/2^k)\otimes L^q_{4t+2}(\ZZ,\ZZ/2^k)\rightarrow L^q_{4(s+t)+4}(\ZZ,\ZZ/2^k) 
\]
are isomorphic to the nontrivial homomorphism $j_{2^k}:\ZZ/2\rightarrow\ZZ/2^k$.
\end{proposition}

\subsubsection{Normal Case}

Now we are left with the normal case only. Let $(D,E)$ be a $\ZZ/2^k$ Poincar\'e symmetric-quadratic chain pair of dimension $m$. For the use in the next section, we only consider the case of dimension $m\geq 2$. \hfill \break

\textbf{(1) $m\equiv 3 \pmod{4}$.}

Define $\sigma^n_3(D,E)=\sigma^q_2(E)\in \ZZ/2$. \ref{product-quadratic-Kervaire} also proves the product formula as follows. 

\begin{lemma}
Let $H$ be a $\ZZ/2^k$ Poincar\'e symmetric chain complex. Then
\[
\sigma^n_3(D\otimes H,E\otimes H)=\sigma^n_3(D,E)\cdot \sigma^s_0(H)\in\ZZ/2
\]
\end{lemma}

Due to the surjectivity of $L^s_{4s}\rightarrow L^n_{4s}$, then

\begin{proposition}
\[
L^s_{4s}\otimes L^n_{4t+3}(\ZZ,\ZZ/2^k)\xrightarrow{\cong} L^n_{4s}\otimes L^n_{4t+3}(\ZZ,\ZZ/2^k)\xrightarrow{\cong} L^n_{4(s+t)+3}(\ZZ,\ZZ/2^k)
\]
\end{proposition}\hfill \break

\textbf{(2) $m\equiv 1 \pmod{4}$} 

The first summand $\ZZ/2$ in the unnatural decomposition $L^n_m(\ZZ/2)\cong \ZZ/2\oplus\ZZ/2$ is the image of the homomorphism $L^n_m\rightarrow L^n_m(\ZZ/2)$, which is canonically given. 

Let us construct the generator of the second summand as follows. Let $(\widetilde{D},\widetilde{E})$ be a Poincar\'e symmetric-quadratic pair of dimension $3$ such that the Kervaire invariant $\sigma^q_2(\widetilde{E})=1\in\ZZ/2$. Let $\widetilde{G}$ be a $\ZZ/2$ Poincar\'e symmetric complex of dimension $2$ such that $\sigma^s_1(\delta G)=1$. 

\begin{lemma}\label{Notation}
Each element $(D,E)$ of the complement of the first summand of $L^n_5(\ZZ/2)$ has $\sigma^q_0(E)= 1\in\ZZ/2$.
\end{lemma}

\begin{proof}
The nontrivial class in the first summand has vanishing $\sigma^q_0$ because it can be represented by a Poincar\'e symmetric chain. Therefore, once we find a class $(D,E)$ such that $\sigma^q_0(E)=1$, the other element in the complement also has $\sigma^q_0(E)=1$. Indeed, for $(\widetilde{D},\widetilde{E})\otimes \widetilde{G}$, 
\[
\sigma^q_0((\widetilde{D},\widetilde{E})\otimes \widetilde{G})=\sigma^q_2(\widetilde{E})\cdot \sigma^s_1(\delta \widetilde{G})=1\in \ZZ/2
\]
\end{proof}

We use the class $(\overline{D}_0,\overline{E}_0)=(\widetilde{D},\widetilde{E})\otimes \widetilde{G}\otimes G$ to generate the second summand in $L_m(\ZZ/2)$, where $G$ is a $(m-5)$-dimensional Poincar\'e symmetric chain with signature $1$.

The previous argument also proves that

\begin{proposition}
\[
L^n_{4s+1}(\ZZ,\ZZ/2^k) \otimes L^s_{4t}(\ZZ,\ZZ/2^k)\xrightarrow{\cong} L^n_{4s+1}(\ZZ,\ZZ/2^k) \otimes L^n_{4t}(\ZZ,\ZZ/2^k)\xrightarrow{\cong} L^n_{4s+t+1}(\ZZ,\ZZ/2^k)
\]    
\end{proposition}

\begin{proposition}
\begin{align*}
L^n_{4s+3} \otimes L^s_{4t+2}(\ZZ,\ZZ/2^k)\xrightarrow{\cong} L^n_{4s+3}(\ZZ,\ZZ/2^k) \otimes L^s_{4t+2}(\ZZ,\ZZ/2^k)\rightarrow L^n_{4(s+t+1)+1}(\ZZ,\ZZ/2^k) \\
\cong \ZZ/2\oplus \ZZ/2   
\end{align*}
is an isomorphism onto the second summand.
\end{proposition}

Due to the dimension, $\delta E$ must be the boundary of some Poincar\'e quadratic pair $\delta F$. Then $E\bigcup_{2^k\delta E}2^k\delta F$ is a Poincar\'e  quadratic chain. 

If $\sigma^q_0(E\bigcup_{2^k\delta E}2^k\delta F)=0\in\ZZ/2^k$, then there exists another $\delta F'$ such that \\
$\Sign(E\bigcup_{2^k\delta E}2^k\delta F')=0\in\ZZ$. $E\bigcup_{2^k\delta E}2^k\delta F'$ must be the boundary of some Poincar\'e quadratic pair $F'$. Then $D\bigcup_{E}F'$ is a $\ZZ/2^k$ Poincar\'e symmetric complex. Now define $\sigma^n_1(D,E)=\sigma^s_1(D\bigcup_{E} F')\in\ZZ/2$

If $\sigma^q_0(E\bigcup_{2^k\delta E}2^k\delta F)=1$, then define $\sigma^n_1(D,E)=\sigma^n_1((D,E)-j_k(\overline{D}_0,\overline{E}_0))$,
where $j_k:\ZZ/2\rightarrow \ZZ/2^k$.\hfill \break

\textbf{(3) $m\equiv 0 \pmod{4}$} 

First let $k\geq 3$. Assume that the Kervaire invariant $\sigma^q_2(\delta E)$ vanishes. Then $\delta E$ is the boundary of some Poincar\'e quadratic pair $\delta F$. Due to the dimension, the quadratic complex $2^k\delta F\bigcup_{2^k \delta E} E$ must be the boundary of some Poincar\'e quadratic pair $F$. Now $2^k(\delta F\bigcup_{\delta E}\delta D)\rightarrow D\bigcup_{E}F$ can be thought of as a $\ZZ/2^k$ Poincar\'e symmetric complex and $(D,E)$ is bordant to $(D\bigcup_{E}F,0)$ as $\ZZ/2^k$ Poincar\'e symmetric-quadratic pairs. Define $\sigma^n_0(D,E)=\Sign(D\bigcup_{E} F)\in\ZZ/8$. The value is independent of the choices of $\delta F$ and $F$. Indeed, a replacement of $\delta F$ changes the signature by a multiple of $2^k$ but a change of $F$ only alters the signature by a multiple of $8$.

The first summand in the unnatural decomposition $L^n_m(\ZZ,\ZZ/2^k)\cong \ZZ/8\oplus\ZZ/2$ is also canonical, i.e., it is the image under the homomorphism $L^n_m\rightarrow L^n_m(\ZZ,\ZZ/2^k)$. We construct the generator of the second summand $\ZZ/2$ as follows.

\begin{lemma}
The multiplication 
\[
L^n_4(\ZZ,\ZZ/2)\otimes L^s_1\cong (\ZZ/2\oplus \ZZ/2)\otimes \ZZ/2\rightarrow L^n_5(\ZZ,\ZZ/2)\cong \ZZ/2\oplus\ZZ/2
\]
is an isomorphism.
\end{lemma}

\begin{proof}
The bijectivity on the first summands of both sides follows from $L^n_4\otimes L^s_1\xrightarrow{\cong} L^n_5$.

There exists a $4$-dimensional Poincar\'e symmetric-quadratic pair $(D,E)$ with nonvanishing $\sigma^q_2(E)$. Indeed, take a Poincar\'e quadratic chain $\delta E$ with nonvanishing Kervaire invariant. $2\delta E$ must be the boundary of some Poincar\'e quadratic pair $E$. Because of the dimension, as a Poincar\'e symmetric chain, $\delta E$ must be the boundary of some Poincar\'e symmetric pair $\delta D$. Then $2\delta D\bigcup_{2\delta E}E$ is also a boundary.

Let $G$ be a $1$-dimensional Poincar\'e symmetric chain with $\sigma^s_1(G)=1$. Then 
\[
\sigma^q_0(\delta E\otimes G)=\sigma^q_2(\delta E)\cdot \sigma^s_1(G)=1\in\ZZ/2
\]
Hence, $(D,E)$ is in the complement of the first summand of $L^n_5(\ZZ,\ZZ/2)$ as well.
\end{proof}

Applying the $4$-periodicity, it also proves that

\begin{proposition}
\[
L^n_{4s}(\ZZ,\ZZ/2)\otimes L^s_{4t+1} \xrightarrow{\cong} L^n_{4s}(\ZZ,\ZZ/2)\otimes L^s_{4t+1} (\ZZ,\ZZ/2)\xrightarrow{\cong} L^n_{4(s+t)+1}(\ZZ,\ZZ/2)
\]
\end{proposition}

Thus, there must be a $\ZZ/2$ Poincar\'e symmetric-quadratic pair $(D'_0,E'_0)$ of dimension $4$, unique up to bordism, such that $(D'_0,E'_0)\otimes G'$ is bordant to $(\overline{D}_0,\overline{E}_0)$, where $G'$ is a $1$-dimensional Poincar\'e symmetric complex of de Rham invariant $1$.

We fix the $\ZZ/2$ Poincar\'e symmetric-quadratic pair $(D_0,E_0)=(D'_0,E'_0)\otimes G$ of dimension $m$, where $G$ is a $(m-4)$-dimensional Poincar\'e symmetric complex with signature $1$. Then $j_k(D_0,E_0)$ is the element to produce the decomposition $L^n_m(\ZZ,\ZZ/2^k)\cong \ZZ/8\oplus \ZZ/2$, where $j_k:\ZZ/2\rightarrow \ZZ/2^k$.

For a general $(D,E)$, let $(D',E')=(D,E)-j_k( \sigma^q_2(\delta E)\cdot (D_0,E_0))$. Define $\sigma^n_0(D,E)=\sigma^n_0(D',E')\in \ZZ/8$.

It is immediate that

\begin{lemma}
$\sigma^n_0:L^n_{4k}(\ZZ,\ZZ/2^k)\cong\ZZ/8 \oplus\ZZ/2\rightarrow \ZZ/8$ is a projection.
\end{lemma}

\begin{proposition}
\[
L^n_{4s}(\ZZ,\ZZ/2^k)\otimes L^s_{4t}\xrightarrow{\cong} L^n_{4s}(\ZZ,\ZZ/2^k)\otimes L^n_{4t}\xrightarrow{\cong}L^n_{4(s+t)}(\ZZ,\ZZ/2^k)
\]
\end{proposition}

We also have the following.

\begin{proposition}
\[
L^n_{4s+3} \otimes L^s_{4t+1}\xrightarrow{\cong} L^n_{4s+3}(\ZZ,\ZZ/2^k) \otimes L^s_{4t+1}\rightarrow L^n_{4(s+t+1)}(\ZZ,\ZZ/2^k)\cong \ZZ/8\oplus \ZZ/2 
\]
is an injection into the first summand.
\end{proposition}

\begin{proof}
It is irrelevant of the second summand due to the following diagram.
\[
\begin{tikzcd}
 L^n_{4s+3} \otimes L^s_{4t+1} \arrow[r] \arrow[d] & L^n_{4(s+t+1)} \arrow[d] \\
 L^n_{4s+3}(\ZZ,\ZZ/2^k) \otimes L^n_{4t+1} \arrow[r] & L^n_{4(s+t+1)}(\ZZ,\ZZ/2^k)
\end{tikzcd}
\]

On the other hand, because of the $4$-periodicity, it reduces to consider the following diagram.
\[
\begin{tikzcd}
 L^n_4(\ZZ,\ZZ/2)\otimes L^s_1 \arrow[r] \arrow[d,"\cong"] & L^n_3\otimes L^s_1 \arrow[d] \\
 L^n_5(\ZZ,\ZZ/2) \arrow[r] & L^n_4
\end{tikzcd}
\]

The generator of $L^n_3\otimes L^s_1$ is mapped to $(\widetilde{D},\widetilde{E})\otimes \widetilde{G}\in L^n_5(\ZZ,\ZZ/2)$, with the notation introduced in the proof of \ref{Notation}.
\end{proof}

We have also proved the following.

\begin{lemma}
Let $(D,E)$ be a $\ZZ/2^k$ Poincar\'e symmetric-quadratic pair with $k\geq 3$ and let $G$ be a Poincar\'e symmetric complex. Then
\[
\sigma^n_0((D,E)\otimes G)=\sigma^n_0(D,E)\cdot \sigma^s_0(G)+j_8(\sigma^n_3(D,E) \cdot\sigma^s_1(G))\in \ZZ/8
\]
where $j_8:\ZZ/2\rightarrow \ZZ/8$.
\end{lemma}

Now consider the case for $k=1,2$.

If $k=1$, $(D,E)$ is a $\ZZ/2$ Poincar\'e symmetric-quadratic pair and $4(D,E)$ is a $\ZZ/8$ pair. Define
\[
\sigma^n_{0,2}(D,E)=\sigma^n_0(4(D,E))\in 4\cdot\ZZ/8\cong\ZZ/2
\]

Alternatively, there is a similar way to define $\sigma^n_{0,2}$ like the case of $k\geq 3$. We do not repeat the tedious process to describe it and check the agreement of two definitions.

\begin{proposition}
\[
L^n_{4s+2}(\ZZ,\ZZ/2^k)\otimes L^n_{4t+2}(\ZZ,\ZZ/2^k)\xrightarrow{0} L^n_{4(s+t+1)}(\ZZ,\ZZ/2^k)
\]
\end{proposition}

\begin{proof}
We only consider the case when $k=1$ and the case for general $k$ is analogous.

$L^n_{4(s+t+1)}(\ZZ/2)$ has a decomposition $\ZZ/2\oplus \ZZ/2$ and we want to show that the projection of the image to each direct summand of $L^n_{4(s+t+1)}$ is $0$.

For the first summand, consider the diagram
\[
\begin{tikzcd}
L^s_{4s+2}(\ZZ,\ZZ/2)\otimes L^s_{4t+2}(\ZZ,\ZZ/2) \arrow[r,"0"] \arrow[d,"\cong"] & L^s_{4(s+t+1)}(\ZZ,\ZZ/2) \arrow[d,"\cong"] \\
L^n_{4s+2}(\ZZ,\ZZ/2)\otimes L^n_{4t+2}(\ZZ,\ZZ/2) \arrow[r] & L^n_{4(s+t+1)}(\ZZ,\ZZ/2)
\end{tikzcd}
\]

For the second summand, remember that the direct sum decomposition follows from the isomorphism 
\[
L^n_4(\ZZ,\ZZ/2)\otimes L^s_1\xrightarrow{\cong} L^n_5(\ZZ,\ZZ/2)
\]
However, $L^n_2(\ZZ,\ZZ/2)\otimes L^s_1\rightarrow L^n_3(\ZZ,\ZZ/2)$ is a zero map because of the diagram
\[
\begin{tikzcd}
L^s_2(\ZZ,\ZZ/2)\otimes L^s_1 \arrow[r] \arrow[d,"\cong"] & L^s_3(\ZZ,\ZZ/2)=0 \arrow[d]  \\
L^n_2(\ZZ,\ZZ/2)\otimes L^s_1 \arrow[r] & L^n_3(\ZZ,\ZZ/2)
\end{tikzcd}
\]
\end{proof}

Then we also get that

\begin{lemma}
Let $(D,E)$ be a $\ZZ/2$ Poincar\'e symmetric-quadratic pair and let $G$ be a $\ZZ/2$ Poincar\'e symmetric complex. Then
\[
\sigma^n_{0,2}((D,E)\otimes G)=\sigma^n_{0,2}(D,E)\cdot \sigma^s_0(G)+\sigma^n_3(D,E) \cdot \sigma^s_1(G)\in \ZZ/2
\]
\end{lemma}

\begin{lemma}
Let $(D,E)$ be a $\ZZ/2$ Poincar\'e symmetric-quadratic pair and $G$ be a $\ZZ/2$ Poincar\'e symmetric complex . Then
\[
\sigma^n_{0,2}(\delta((D,E)\otimes G))=\sigma^n_{0,2}(\delta(D,E))\cdot \sigma^s_0(G)+\sigma^n_{0,2}(D,E)\cdot \sigma^s_0(\delta G)\in \ZZ/2
\]
\end{lemma}

\begin{proof} 
$\delta((D,E)\otimes G)$ is bordant to $(\delta D,\delta E)\otimes G\oplus (D,E)\otimes\delta G$ as $\ZZ/2$ Poincar\'e symmetric-quadratic pairs.
\end{proof}

\begin{lemma}
\label{5+3=8normal}
Let $(D,E)$ be a $\ZZ/2$ Poincar\'e symmetric-quadratic pair and let $G$ be a $\ZZ/2$ Poincar\'e symmetric complex. Then
\[
\sigma^n_1((D,E)\otimes G)=\sigma^n_1(D,E)\cdot \sigma^s_0(G)+\sigma^n_{0,2}(D,E)\cdot \sigma^s_1(G)\in\ZZ/2
\]
\end{lemma}

Similarly, if $k=2$, $(D,E)$ is a $\ZZ/4$ Poincar\'e symmetric-quadratic pair and define
\[
\sigma^n_{0,4}(D,E)=\sigma^n_0(2(D,E))\in 2\cdot\ZZ/8\cong\ZZ/4
\]
There are also some multiplicative formulae of bordism invariants. They are pretty similar to the $k=1$ case, so let us skip stating them. \hfill \break

\textbf{(4) Product Structure on $L^n_*(\ZZ,\ZZ/2^k)$}

We already proved the product formula $
L^n_{4s+2}(\ZZ,\ZZ/2^k)\otimes L^n_{4t+2}(\ZZ,\ZZ/2^k)\xrightarrow{0} L^n_{4(s+t+1)}(\ZZ,\ZZ/2^k)$.

\begin{proposition}
\[
L^n_{4s+1}(\ZZ,\ZZ/2^k)\otimes L^n_{4t+2}(\ZZ,\ZZ/2^k)\xrightarrow{0} L^n_{4(s+t)+3}(\ZZ,\ZZ/2^k)
\]
\end{proposition}

\begin{proof}
It follows from the diagram
\[
\begin{tikzcd}
L^s_{4s+1}(\ZZ,\ZZ/2^k)\otimes L^s_{4t+2}(\ZZ,\ZZ/2^k) \arrow[r,"0"] \arrow[d] & L^s_{4(s+t)+3}(\ZZ,\ZZ/2^k) \arrow[d,"\cong"] \\
L^n_{4s+1}(\ZZ,\ZZ/2^k)\otimes L^n_{4t+2}(\ZZ,\ZZ/2^k) \arrow[r] & L^n_{4(s+t)+3}(\ZZ,\ZZ/2^k)
\end{tikzcd}
\]
and the diagram
\[
\begin{tikzcd}
L^n_{4s-2}(\ZZ,\ZZ/2^k)\otimes L^n_3 \otimes L^s_{4t+2}(\ZZ,\ZZ/2^k) \arrow[r,"0"] \arrow[d] & L^n_{4(s+t)}(\ZZ,\ZZ/2^k) \otimes L^n_3 \arrow[d] \\
L^n_{4s+1}(\ZZ,\ZZ/2^k)\otimes L^n_{4t+2}(\ZZ,\ZZ/2^k) \arrow[r] & L^n_{4(s+t)+3}(\ZZ,\ZZ/2^k)
\end{tikzcd}
\]
\end{proof}

\begin{proposition}
\[
L^n_{4s}(\ZZ,\ZZ/2^k)\otimes L^n_{4t+3}(\ZZ,\ZZ/2^k)\rightarrow L^n_{4(s+t)+3}(\ZZ,\ZZ/2^k)
\]
is isomorphic to the map 
\[
(\ZZ/2^l\oplus \ZZ/2)\otimes \ZZ/2\xrightarrow{\pi_1} \ZZ/2^l\otimes \ZZ/2\rightarrow \ZZ/2
\]
where the first arrow is the projection onto the first factor. $l=3$ if $k\geq 3$ and $l=k$ if $k=1,2$.
\end{proposition}

\begin{proof}
The map on the first direct summand $\ZZ/2^l$ is the $4$-periodicity. For the second direct summand $\ZZ/2$, it can be deduced by the diagram 
\[
\begin{tikzcd}
L^n_{4s}(\ZZ,\ZZ/2^k)\otimes L^n_{4t+3} \otimes L^s_1 \arrow[r] \arrow[d] & L^n_{4s}(\ZZ,\ZZ/2^k) \otimes L^n_{4(t+1)} \arrow[d] \\
L^n_{4(s+t)+3}(\ZZ,\ZZ/2^k)\otimes L^s_1 \arrow[r] & L^n_{4(s+t+1)}(\ZZ,\ZZ/2^k)
\end{tikzcd}
\]
and the fact that $(4\ZZ/8)\otimes \ZZ/2$ is $0$ in $\ZZ/8\otimes \ZZ/2$.

\end{proof}

So we have also proved the following.

\begin{lemma}
Let $(D,E)$ and $(D',E')$ be two $\ZZ/2$ Poincar\'e symmetric-quadratic pairs. Then
\[
\sigma^n_3((D,E)\otimes (D',E'))  =\sigma^n_{0,2}(D,E)\cdot \sigma^n_3(D',E')+\sigma^n_3(D,E)\cdot \sigma^n_{0,2}(D',E')\in\ZZ/2
\]
\end{lemma}

\begin{proposition}
\[
L^n_{4s+1}(\ZZ,\ZZ/2^k)\otimes L^n_{4t+3}(\ZZ,\ZZ/2^k)\rightarrow L^n_{4(s+t+1)}(\ZZ,\ZZ/2^k)
\]
is isomorphic to the map 
\[
(\ZZ/2\oplus \ZZ/2^l)\otimes \ZZ/2\xrightarrow{\pi_1} \ZZ/2\otimes \ZZ/2\xrightarrow{\cong}\ZZ/2\xrightarrow{j_l} \ZZ/2^l\xrightarrow{i} \ZZ/2^l\oplus \ZZ/2
\]
where the first arrow is the projection onto the first direct summand and $i$ is the natural inclusion. $l=3$ if $k\geq 3$ and $l=k$ if $k=1,2$.
\end{proposition}

\begin{proof}
The map on the first direct summand $\ZZ/2\otimes \ZZ/2$ is calculated before. The map on the second direct summand can be deduced from the fact that $(4\ZZ/8)\otimes \ZZ/2$ is $0$ in $\ZZ/8\otimes \ZZ/2$ and the diagram chasing on
\[
\begin{tikzcd}
L^n_{4s+1}(\ZZ,\ZZ/2^k)\otimes L^n_{4t+3} \otimes L^s_1 \arrow[r] \arrow[d] & L^n_{4s+1}(\ZZ,\ZZ/2^k) \otimes L^n_{4(t+1)} \arrow[d] \\
L^n_{4(s+t+1)}(\ZZ,\ZZ/2^k)\otimes L^s_1 \arrow[r] & L^n_{4(s+t+1)+1}(\ZZ,\ZZ/2^k)
\end{tikzcd}
\]
\end{proof}

\begin{proposition}
\[
L^n_{4s}(\ZZ,\ZZ/2^k)\otimes L^n_{4t}(\ZZ,\ZZ/2^k)\rightarrow L^n_{4(s+t)}(\ZZ,\ZZ/2^k)
\]
is isomorphic to the map 
\begin{multline*}
(\ZZ/2^l\oplus \ZZ/2)\otimes (\ZZ/2^l\oplus \ZZ/2)\rightarrow (\ZZ/2^l\otimes \ZZ/2^l)\oplus (\ZZ/2^l\otimes \ZZ/2)\oplus (\ZZ/2\otimes \ZZ/2^l) \\
\xrightarrow{(1+0+0,0+i+i)} \ZZ/2^l\oplus \ZZ/2    
\end{multline*}
where the first arrow is the projection and the second arrow means the identity on $\ZZ/2^l$ and the sum of two $\ZZ/2$'s onto $\ZZ/2$. $l=3$ if $k\geq 3$ and $l=k$ if $k=1,2$.
\end{proposition}

\begin{proof}
All the nontrivial maps are $4$-periodicity. The map follows from the diagram
\[
\begin{tikzcd}
L^n_{4s}(\ZZ,\ZZ/2^k)\otimes L^n_{4t}(\ZZ,\ZZ/2^k) \otimes L^s_1 \arrow[r] \arrow[d] & L^n_{4s}(\ZZ,\ZZ/2^k) \otimes L^n_{4t+1}(\ZZ,\ZZ/2^k) \arrow[d] \\
L^n_{4(s+t)}(\ZZ,\ZZ/2^k)\otimes L^s_1 \arrow[r] & L^n_{4(s+t)+1}(\ZZ,\ZZ/2^k)
\end{tikzcd}
\]
\end{proof}

We have also proved the following.

\begin{lemma}
Let $(D,E)$ and $(D',E')$ be two $\ZZ/2^k$ Poincar\'e symmetric-quadratic pairs with $k\geq 3$. Then
\begin{eqnarray*}
\sigma^n_0((D,E)\otimes (D',E'))   & = & \sigma^n_0(D,E)\cdot \sigma^n_0(D',E') \\
& & +j_8(\sigma^n_1(D,E)\cdot \sigma^n_3(D',E')
+\sigma^n_3(D,E)
\cdot \sigma^n_1(D',E'))
\in\ZZ/8
\end{eqnarray*}
where $j_8:\ZZ/2\rightarrow \ZZ/8$.
\end{lemma}

For the case when $k=1,2$, the result is similar.

\begin{proposition}
Let $(D,E)$ and $(D',E')$ be two $\ZZ/2^k$ Poincar\'e symmetric-quadratic pairs with $k\geq 3$. Then
\begin{eqnarray*}
\sigma^n_0(\delta((D,E)\otimes (D',E'))) & = & \sigma^n_0(\delta (D,E))\cdot \sigma^n_0(D',E')+\sigma^n_0(D,E)\cdot \sigma^n_0(\delta (D',E'))\\
 & & +j_8(\sigma^n_3(\delta (D,E))\cdot \sigma^n_1(D',E')+ \sigma^n_1(D,E)\cdot \sigma^n_3(\delta (D',E'))) \\
 & & +j_8(\sigma^n_3(D,E)\cdot \sigma^n_1(\delta(D',E')) +\sigma^n_1(\delta(D,E))\cdot \sigma^n_3(D',E'))
\end{eqnarray*}
\end{proposition}

\begin{proposition}
\[
L^n_{4s}(\ZZ,\ZZ/2^k)\otimes L^n_{4t+1}(\ZZ,\ZZ/2^k)\rightarrow L^n_{4(s+t)+1}(\ZZ,\ZZ/2^k)
\]
is the map isomorphic to the map 
\begin{multline*}
(\ZZ/2^l\oplus \ZZ/2)\otimes (\ZZ/2\oplus \ZZ/2^l)\rightarrow (\ZZ/2^l\otimes \ZZ/2)\oplus(\ZZ/2^l\otimes \ZZ/2^l)\oplus (\ZZ/2\otimes \ZZ/2) \\
\xrightarrow{(1+0+0,0+1+j_l)} \ZZ/2\oplus \ZZ/2^l    
\end{multline*}
where the first arrow is the projection. $l=3$ if $k\geq 3$ and $l=k$ if $k=1,2$.
\end{proposition}

\begin{proof}
The map onto the second direct summand $\ZZ/2^l$ is either already proven. The map onto the first direct summand $\ZZ/2$ follows from the diagram
\[
\begin{tikzcd}
L^n_{4s}(\ZZ,\ZZ/2^k)\otimes L^s_{4t-2}(\ZZ,\ZZ/2^k) \otimes L^n_3 \arrow[r] \arrow[d] & L^n_{4s+3}(\ZZ,\ZZ/2^k) \otimes L^n_{4t-2}(\ZZ,\ZZ/2^k) \arrow[d] \\
L^n_{4(s+t)}(\ZZ,\ZZ/2^k)\otimes L^n_{4t+1}(\ZZ,\ZZ/2^k) \arrow[r] & L^n_{4(s+t)+1}(\ZZ,\ZZ/2^k)
\end{tikzcd}
\]
\end{proof}

The product $L^n_{4s+2}(\ZZ,\ZZ/2^k)\otimes L^n_{4t+3}(\ZZ,\ZZ/2^k)\rightarrow L^n_{4(s+t+1)+1}(\ZZ,\ZZ/2^k)$ is essentially proven above. Hence, we also get that

\begin{lemma}
Let $(D,E)$ and $(D',E')$ be $\ZZ/2$ Poincar\'e symmetric-quadratic pairs. Then
\[
\sigma^n_1((D,E)\otimes(D',E'))=\sigma^n_{0,2}(D,E)\cdot \sigma^n_1(D',E')+\sigma^n_1(D,E)\cdot \sigma^n_{0,2}(D',E')\in\ZZ/2
\]
\end{lemma}

\begin{proposition}
The only nontrivial part in the following maps
\[
L^n_{4s}(\ZZ,\ZZ/2^k)\otimes L^n_{4t+2}(\ZZ,\ZZ/2^k)\rightarrow L^n_{4(s+t)+2}(\ZZ,\ZZ/2^k)
\]
\[
L^n_{4s+1}(\ZZ,\ZZ/2^k)\otimes L^n_{4t+1}(\ZZ,\ZZ/2^k)\rightarrow L^n_{4(s+t)+2}(\ZZ,\ZZ/2^k)
\]
\[
L^n_{4s+3}(\ZZ,\ZZ/2^k)\otimes L^n_{4t+3}(\ZZ,\ZZ/2^k)\rightarrow L^n_{4(s+t+1)+2}(\ZZ,\ZZ/2^k)
\]
is the $4$-periodicity.
\end{proposition}

\begin{proof}
The second and the third map can be essentially deduced from the fact $L^n_{4s+2}=0$ and the product structure on $L^n_*$. 

The rest follows from the diagram
\[
\begin{tikzcd}
L^n_{4s}(\ZZ,\ZZ/2^k)\otimes L^s_{4t+2}(\ZZ,\ZZ/2^k) \otimes L^n_3 \arrow[r] \arrow[d] & L^n_{4s+3}(\ZZ,\ZZ/2^k) \otimes L^n_{4t+2}(\ZZ,\ZZ/2^k) \arrow[d] \\
L^n_{4(s+t)+2}(\ZZ,\ZZ/2^k)\otimes L^n_3 \arrow[r] & L^n_{4(s+t+1)+2}(\ZZ,\ZZ/2^k)
\end{tikzcd}
\]
\end{proof}

\section{\texorpdfstring{$L$}{Lg}-theory Characteristic Classes and Bundle Lifting Problem}\label{Mainsection}

In this section, we first construct some cohomology classes for three $L$-spectra and then prove that these classes induce splittings of the spectra localized at $2$. The method we use is the a priori invariant method for cohomolgy classes introduced in Section \ref{Preliminaries}.

When we have the characteristic classes, we determine the relationship among characteristic classes of different $L$-theories under the fibration $\LL^q\rightarrow \LL^s\rightarrow \LL^n$.

Moreover, there are ring structures and module structures over $\LL$-spectra and we also calculate the induced coproducts of these classes. Note that, for the surgery theory, the ring structure on $\LL^q$ is induced from the Whitney sum structure on $G/TOP$. This ring structure can be induced from $\LL^s$-module structure on $\LL^q$. We show that the coproduct of the characteristic classes of $\LL^q$ induced by surgery theory and the coproduct induced from the module structure are essentially related.

Levitt-Ranicki's $L$-theory orientations of bundles (\cite[Proposition 16.1]{Ranicki1992}) imply that the cohomology classes of $\LL$-spectra we construct in this section induce some characteristic classes for spherical fibrations and for $TOP$-bundles. We prove that these classes are exactly the same classes constructed in \cite{Sullivan&Morgan} and \cite{Brumfiel&Morgan}.

Throughout this section, the symmetric spectrum $\LL^s$ is $0$-connective, the quadratic spectrum $\LL^q$ is $1$-connective and the normal spectrum $\LL^n$ is $1/2$-connective, as we clarified in section $2$. By abuse of notations, we still use $\LL^a$ to represent the $0$-th space for each spectrum ($a=s,q,n$) and a reader can easily distinguish the meanings by the context. 

\subsection{\texorpdfstring{$L$}{Lg}-theory Characteristic Classes and Splittings of \texorpdfstring{$L$}{Lg}-spectra at Prime \texorpdfstring{$2$}{Lg}}

\subsubsection {Quadratic \texorpdfstring{$L$}{Lg}-spectrum}

This subsection is a reproof of the splitting of $\LL^q(\simeq G/TOP)$ at prime $2$, with the technique of $\ZZ/n$ quadratic chains discussed in the previous section.

\begin{lemma}
\label{quadratic2modulo4}
The Hurewicz map $h:\pi_{4k+2}(\LL^q)\rightarrow H_{4k+2}(\LL^q;\ZZ)$ is an injection onto a direct summand.
\end{lemma}

\begin{proof}
It suffices to prove that the mod $2$ Hurewicz map $h_2$
\[
h_2: L^q_{4k+2}\cong\pi_{4k+2}(\LL^q)\xrightarrow{i} \Omega^{SO}_{4k+2}(\LL^q;\ZZ/2) \xrightarrow{j} H_{4k+2}(\LL^q;\ZZ/2)
\] 
is  injective.

Let us construct a splitting inverse of $i$ and then $i$ is obviously an injection. For any $(M^{4k+2},f)\in\Omega^{SO}_{4k+2}(\LL^q;\ZZ/2)$, it corresponds to a presheaf $\mathcal{C}_f$ of Poincar\'e quadratic chains over $M$. The assembly $\mathcal{C}_f(M)$ gives a splitting 
\[
\Omega^{SO}_{4k+2}(\LL^q;\ZZ/2)\rightarrow L^q_{4k+2}(\ZZ,\ZZ/2)\xleftarrow{\cong} L^q_{4k+2}
\]

To prove $h_2$ is injective, let $g:S^{4k+2}\rightarrow \LL^q$ be the generator of the homotopy group. Then the injectivity of $i$ implies that $(S^{4k+2},g)$ does not bound any $\ZZ/2$ singular manifold in $\LL^q$. Hence, there exists some nonvanishing generalized Stiefel-Whitney number of $(S^{4k+2},g)$. Since all the Stiefel-Whitney classes of a sphere vanish, it means that some cohomology class $H^{4k+2}(\LL^q;\ZZ_2)$ evaluates nontrivially on $h_2(S^{4k+2},g)$.
\end{proof}

\begin{proposition}
\label{Kervaireclass}
There exists a graded class 
\[
k^q=k^q_2+k^q_6+\cdots\in H^{4*+2}(\LL^q;\ZZ/2)
\]
which agrees with the Kervaire class of $G/TOP$ constructed in \cite[Theorem 4.6]{Sullivan&Rourke}.
\end{proposition}

\begin{proof}
Let $M$ be a $\ZZ/2$ manifold with a map $f:M\rightarrow \LL^q$. Let $\mathcal{C}_f$ be the associated presheaf of quadratic chains over $M$. Define a map $\sigma^q_2:\Omega^{SO}(\LL^q;\ZZ/2)\rightarrow \ZZ/2$ by $\sigma^q_2(M,f)=\sigma^q_2(\mathcal{C}_f(M))$.

Let $N$ be another $\ZZ/2$ manifold. The product property follows from the chain-level product formula
\[
\sigma^q_2((M,f)\cdot N)=\sigma^q_2(\mathcal{C}_f(M)\otimes C_*(N))=\sigma^q_2(\mathcal{C}_f(M))\cdot \sigma^s_0(C_*(N))=\sigma^q_2(M,f)\cdot \chi_2(N)
\]
\end{proof}

Since $\pi_{4k}(\LL^q)\simeq L^q_{4k}\simeq \ZZ$, to prove the splitting injectivity of the Hurewicz map localized at prime $2$, it suffices to construct some characteristic class $l^q_{4k}\in H^{4k}(\LL^q_{4k},\ZZ_{(2)})$ such that its evaluation on the generator of $L^q_{4k}$ is an odd number.

Let $M$ be a $\ZZ$ or $\ZZ/2^k$ manifold with a map $f:M\rightarrow \LL^q$ and let $\mathcal{C}_f$ be the associated presheaf of quadratic chain complexes over $M$. Define 
\[
\sigma^q_0(M,f)=\sigma^q_0(\mathcal{C}_f(M))-j_k\langle\beta(V_MSq^1 V_M\cdot f^* k^q),[M]\rangle\in\ZZ \,\,\text{or}\,\, \ZZ/2^k
\]

Recall that the de Rham invariant of a $\ZZ/2^k$ manifold $M$ satisfies the equation (\cite[Lemma 8.2]{Sullivan&Morgan})
\[
\dR(M)=\langle V Sq^1 V,[M]\rangle\in\ZZ/2
\]
where $V Sq^1 V=(1+v_2+v_4+\cdots)\cdot Sq^1(1+v_2+v_4+\cdots)$ and $\{v_i\}$ is the Wu class. 

Then
\[
\dR(\delta M)=\langle \beta (V Sq^1 V),[M]\rangle\in\ZZ/2\simeq 2^{k-1}\cdot\ZZ/2^k
\]
where $\beta$ is the $\ZZ/2\rightarrow \ZZ/2^k$ Bockstein homomorphism.

The chain-level product formula implies the following.

\begin{lemma}
Let $M$ and $N$ be $\ZZ/2^k$ manifolds with a map $f:M\rightarrow \LL^q$. Then
\[
\sigma^q_0((M,f)\cdot N)=\sigma^q_0(M,f)\cdot \Sign(N) \in \ZZ \,\,\text{or}\,\, \ZZ/2^k
\]
\end{lemma}

Therefore,

\begin{proposition}
There exists a graded class 
\[
l^q=l^q_4+l^q_8+\cdots\in H^{4*}(\LL^q;\ZZ_{(2)})
\]
such that for any map $f:M\rightarrow \LL^q$,
\[
\sigma^q_0(M,f)=\langle L_M\cdot f^*l^q,[M]\rangle \in \ZZ \,\,\text{or}\,\, \ZZ/2^k
\]
where $M$ is a $\ZZ$ or $\ZZ/2^k$ manifold.
\end{proposition}

Recall that the set $[X,\LL^q]$ classifies the bordism classes of presheaves of Poincar\'e quadratic chains over $X$. The previous results can be restated as follows.

\begin{proposition}
For any presheaf $\mathcal{Q}$ of $0$-connective Poincar\'e quadratic chains over $X$, there exist graded characteristic classes
\[
k^q(\mathcal{Q})=k^q_2+k^q_6+\cdots\in H^{4*+2}(X;\ZZ/2)
\]
\[
l^q(\mathcal{Q})=l^q_4+l^q_8+\cdots\in H^{4*}(X;\ZZ_{(2)})
\]
which are invariant under presheaf bordism.
\end{proposition}

Now let $f:S^{4k}\rightarrow \LL^q$ be a generator of $\pi_{4k}(\LL^q)\simeq \ZZ$. Then the associated presheaf of quadratic chains $\mathcal{C}_f$ has $\sigma^q_0(\mathcal{C}_f(S^{4k}))=1$. Hence,
\[
\langle l^q_{4k},f_{*}[S^{4k}]\rangle=\langle L_{S^{4k}}\cdot f^*l^q_{4k},[S^{4k}]\rangle=\sigma^q_0(\mathcal{C}_f(S^{4k}))=1
\]
Then we have reproved that

\begin{theorem}
Localized at prime $2$, 
\[
\LL^q\simeq \prod_{k>0} (K(\ZZ_{(2)},4k)\times K(\ZZ/2,4k-2))
\]
where the homotopy equivalence is given by the classes $l^q$ and $k^q$.
\end{theorem}

\subsubsection{Symmetric \texorpdfstring{$L$}{Lg}-spectrum}

Applying the argument of \ref{quadratic2modulo4}, we can prove an analogous statement about $\LL^s$.

\begin{lemma}
\label{symmetric2modulo4}
The Hurewicz map $h:\pi_{4k+1}(\LL^s)\rightarrow H_{4k+1}(\LL^s,\ZZ)$ is an injection onto a direct summand.
\end{lemma}

We will define the $4*$-degree integral class before the $(4*+1)$-degree $\ZZ/2$ class.

It is different from $\LL^q$ that the space $\LL^s$ is not connected. Denote the $t$-th component by $\LL^s_t$. Notice that the $0$-cells of $\LL^s_t$ are just $0$-dimensional Poincar\'e symmetric chains with the $0$-th homology of rank $t$.

Let $M$ be a $\ZZ$ or $\ZZ/2^l$ manifold with a map $f:M\rightarrow \LL^s$ and let $\mathcal{C}_f$ be the associated presheaf of symmetric chains. Define 
\[
\sigma^s_0(M,f)=\sigma^s_0(\mathcal{C}_f(M))\in \ZZ \,\,\text{or}\,\, \ZZ/2^l
\]
It is immediate that the product formula holds, i.e.,
\[
\sigma^s_0((M,f)\cdot N)=\sigma^s_0(M,f)\cdot \Sign(N)
\]
where $N$ is another $\ZZ$ or $\ZZ/2^l$ manifold. Hence, we have

\begin{proposition}
There exists a graded class 
\[
l^s_{t}=l^s_{t,0}+l^s_{t,4}+l^s_{t,8}+\cdots\in H^{4*}(\LL^s_t;\ZZ_{(2)})
\]
such that for any map $f:M\rightarrow \LL^s_t$,
\[
\sigma^s_0(M,f)=\langle L_M\cdot f^*l^s_t,[M]\rangle \in \ZZ \,\,\text{or}\,\, \ZZ/2^l
\]
where $M$ is a closed manifold or a $\ZZ/2^l$ manifold.
\end{proposition}

$\LL^s$ also has an additional structure due to the infinite loop space structure, i.e.,
\[
a_{t,t'}:\LL^s_t\times\LL^s_{t'}\rightarrow \LL^s_{t+t'}
\]

The homotopy equivalence between different components is indeed the composition of maps
\[
a_t:\LL^s_0=\LL^s_0\times \pt\rightarrow \LL^s_0\times \LL^s_t\xrightarrow{a_{0,t}} \LL^s_t
\]

Let us write down the map explicitly, let $f:\sigma^m\rightarrow \LL^s_{t}$ and $g:\sigma^n\rightarrow \LL^s_{t'}$ be cellular maps from simplices of dimension $m$ and $n$ respectively. They correspond to presheaves $\mathcal{C}_f$ and $\mathcal{C}_g$ respectively. Now we consider the map $a\circ(f\times g):\sigma^m\times\sigma^n\rightarrow \LL_{t+t'}$ and its associated presheaf $\mathcal{C}_{a\circ (f\times g)}$. $\sigma^m\times\sigma^n$ has a natural regular cell decomposition of the form $\tau^{m'}\times\tau^{n'}$, where $\tau^{m'}$ and $\tau^{n'}$ are faces of $\sigma^m$ and $\sigma^n$ respectively. Then $\mathcal{C}_{a\circ (f\times g)}(\tau^{m'}\times\tau^{n'})\cong (\mathcal{C}_f(\tau^{m'})\otimes C_*(\tau^{n'}))\oplus (C_*(\tau^{m'})\otimes \mathcal{C}_f(\tau^{n'}))$, where $C_*(\tau^{m'})$ and $C_*(\tau^{n'})$ are cellular chain complexes. Hence, we have proved that

\begin{lemma}
Let $M$ and $N$ be two $\ZZ/2^l$ manifolds with maps $f:M\rightarrow \LL^s$ and $g:N\rightarrow \LL^s$. Let $\mathcal{C}_f$ and $\mathcal{C}_g$ be the associated presheves of Poincar\'e symmetric chains over $M$ and $N$ respectively. Let $\mathcal{C}_{a\circ (f\otimes g)}$ be the presheaf associated to $a\circ (f\otimes g)$. Then $\mathcal{C}_{a\circ (f\otimes g)}(M\otimes N)$ is bordant to $(\mathcal{C}_f(M)\otimes C_*(N))\oplus (C_*(M)\otimes\mathcal{C}_g(N))$
\end{lemma}

The $l^s$-classes in different components are connected by the additional structure.

\begin{proposition}
\[
a^*_{t,t'}l^s_{t+t'}=l^s_t\times 1+1\times l^s_{t'}
\]
\end{proposition}

\begin{proof}
Let $M$ and $N$ be $\ZZ/2^l$ manifolds with maps $f:M\rightarrow \LL^s_t$ and $g:N\rightarrow \LL^s_{t'}$ . By the lemma above we know that $\mathcal{C}_{a\circ(f\times g)}(M\times N)$ is bordant to $ (\mathcal{C}_f(M)\otimes C_*(N))\oplus (C_*(M)\otimes\mathcal{C}_g(N))$.

On one hand,
\[
\sigma^s_0(M\otimes N,a\circ (f\otimes g))=\langle L_{M\otimes N}\cdot(f\otimes g)^*a^*(l^s_{t+t'}),[M\otimes N]\rangle
\]
On the other hand,
\begin{eqnarray*}
\sigma^s_0((\mathcal{C}_f(M)\otimes C_*(N))\oplus (C_*(M)\otimes\mathcal{C}_g(N)))
& = & \langle L_M\cdot f^*l^s_t,[M]\rangle\cdot 
\langle L_N,[N]\rangle \\
& & +\langle L_M,[M] \rangle\cdot 
\langle L_N \cdot g^* l^s_{t'}, [N]\rangle    
\end{eqnarray*}

\end{proof}

By a similar argument we also have

\begin{proposition}
\[
a^*_t l^s_t=t+ l^s_0
\]
\end{proposition}

In particular,
\begin{proposition}
\[
l^s_{t,0}=t\in H^0(\LL^s_t;\ZZ_{(2)})\simeq\ZZ_{(2)}
\]
\end{proposition}

Now let $M$ be a $\ZZ/2$ manifold with a map $f:M\rightarrow \LL^s_t$ and let $\mathcal{C}_f$ be the associated presheaf. Define
\[
\sigma^s_1(M,f)=\sigma^s_1(\mathcal{C}_f(M))-\langle V_M Sq^1 V_M\cdot f^* l^s_t,[M]\rangle \in \ZZ/2
\]
Then the product formula follows from the chain-level formula. That is,

\begin{lemma}
Let $M$ and $N$ be $\ZZ/2$ manifolds with a map $f:M\rightarrow \LL_t^s$. Then
\[
\sigma^s_1((M,f)\cdot N)=\sigma^s_1(M,f)\cdot \chi_2(N)
\]
\end{lemma}

Therefore, we have

\begin{proposition}
There exists a graded class 
\[
r^s_t=r^s_{t,1}+r^s_{t,5}+\cdots\in H^{4*+1}(\LL^s_t;\ZZ/2)
\]
such that for any $f:M\rightarrow\LL^s_t$
\[
\sigma^s_1(M,f)=\langle V^2_M\cdot f^* r^s_t,[M]\rangle\in\ZZ/2
\]
where $M$ is a $\ZZ/2$ manifold.
\end{proposition}

With the same argument above, we have
\begin{proposition}
\[
a^*_t r^s_t=r^s_0
\]
\end{proposition}

Similarly to the quadratic case, we can pull back the classes to the characteristic classes of presheaves of Poincar\'e symmetric chains.

\begin{proposition}
For any presheaf $\mathcal{S}$ of $0$-connective Poincar\'e symmetric chains over $X$, there exist graded characteristic classes
\[
r^s(\mathcal{S})=r^s_1+r^s_5+\cdots\in H^{4*+1}(X;\ZZ/2)
\]
\[
l^s(\mathcal{S})=l^s_0+l^s_4+l^s_8+\cdots\in H^{4*}(X;\ZZ_{(2)})
\]
which are invariant under bordism.
\end{proposition}

With the same proof as the quadratic case, we have that

\begin{theorem}\label{characteristic-classes-12}
Localized at prime $2$, 
\[
\LL^s\simeq K(\ZZ,0)\times \prod_{k>0} (K(\ZZ_{(2)},4k)\times K(\ZZ/2,4k-3))
\]
where the homotopy equivalence is given by the $l^s$ and $r^s$ classes.
\end{theorem}

\subsubsection{Normal \texorpdfstring{$L$}{Lg}-spectrum}

We have to notice that $\pi_*(\LL^n)$ has $4$-periodicity for $*\geq 1$ and for the lower degrees, $\pi_0(\LL^n)\cong \ZZ$ and $\pi_i(\LL^n)=0$ for $i<0$.

Since $\LL^n$ is not connected, we let $\LL^n_t$ be the $t$-th component.

Again, we apply the argument of \ref{quadratic2modulo4} and get that

\begin{lemma}
\label{normalodd}
The Hurewicz map $h:\pi_{2k+1}(\LL^n)\rightarrow H_{2k+1}(\LL^n;\ZZ)$ is an injection onto a direct summand.
\end{lemma}

We will define the cohomology classes of $\LL^n$ in the order of $(4*+3)$-degree, $4*$-degree and $(4*+1)$-degree.

Let $M$ be a $\ZZ/2$ manifold with $f:M \rightarrow \LL^n$ and let $(\mathcal{D}_f,\mathcal{E}_f)$ be the associated presheaf  of symmetric-quadratic chain pairs. Define 
\[
\sigma^n_3(M,f)=\sigma^n_3(\mathcal{D}_f(M),\mathcal{E}_f(M))
\]

The product formula is immediate, i.e.,
\[
\sigma^n_3((M,f)\cdot N)=\sigma^n_3(M,f)\cdot \chi_2(N)
\]
where $N$ is a $\ZZ/2$ manifold.

\begin{theorem}
There exists a graded class 
\[
k^n_t=k^n_{t,3}+k^n_{t,7}+\cdots\in H^{4*+3}(\LL^n_t;\ZZ/2)
\]
such that for any map $f:M\rightarrow \LL^n_t$
\[
\sigma^n_3(M,f)=\langle V^2_M\cdot f^*k^n_t,[M]\rangle\in \ZZ/2
\]
where $M$ is a $\ZZ/2$ manifold.
\end{theorem}

Next, let $M$ be a $\ZZ/8$ manifold with a map $f:M\rightarrow \LL^n_t$ and let $(\mathcal{D}_f,\mathcal{E}_f)$ be the associated presheaf. Define 
\[
\sigma^n_0(M,f)=\sigma^n_0((\mathcal{D}_f(M),\mathcal{E}_f(M)))-j_8\langle V_MSq^1V_M\cdot f^*k^n_t,[M]\rangle
\in \ZZ/8
\]
where $j_8:\ZZ/2\rightarrow \ZZ/8$. 

The following product formulae are directly proven by the chain-level formulae.

\begin{lemma}
Let $M$ be a $\ZZ/p$ manifold and $N$ be a $\ZZ/q$ manifold with a map $f:M\rightarrow \LL^n_k$.

(1) If $p=8$ and $q=0$, then 
\[
\sigma^n_0((M,f)\cdot N)=\sigma^n_0(M,f)\cdot \Sign(N)\in \ZZ/8
\]

(2) If $p=2$ and $q=2$, then 
\[
\sigma^n_0(j_8((M,f)\cdot N))=\sigma^n_0(j_8(M,f))\cdot \Sign(N)\in 4\cdot \ZZ/8
\]
where $j_8:\ZZ/2\rightarrow \ZZ/8$.

(3) If $p=2$ and $q=2$, then 
\[
\sigma^n_0(\delta((M,f)\cdot N))=\sigma^n_0(\delta(M,f))\cdot \Sign(N)\in 4\cdot \ZZ/8
\]
\end{lemma}

Then we can deduce the existence of a $\ZZ/8$ class.

\begin{proposition}
There exists a graded class 
\[
l^n_t=l^n_{t,0}+l^n_{t,4}+l^n_{t,8}+\cdots\in H^{4*}(\LL_k^n;\ZZ/8)
\]
such that for any map $f:M\rightarrow \LL^n_t$ we have
\[
\sigma^n_0(M,f)=\langle L_M\cdot f^*l^n_t,[M]\rangle \in \ZZ/8
\]
where $M$ is a $\ZZ/8$ manifold.
\end{proposition} 

Now let $M$ be a $\ZZ/2$ manifold again with a map $f:M\rightarrow \LL_t^n$ and let $(\mathcal{D}_f,\mathcal{E}_f)$ be the associated preheaf. Define 
\[
\sigma^n_1(M,f)=\sigma^n_1((\mathcal{D}_f(M),\mathcal{E}_f(M)))-\langle V_M Sq^1V_M\cdot f^*\rho_2l^n_t,[M]\rangle \in\ZZ/2
\]
where $\rho_2: \ZZ/8\rightarrow\ZZ/2$. 

Because of the chain-level product formula, we have

\begin{lemma}
Let $M$ and $N$ be $\ZZ/2$ manifolds with a map $f:M\rightarrow \LL^n_k$. Then 
\[
\sigma^n_1((M,f)\cdot N)=\sigma^n_1(M,f)\cdot \chi_2(N)\in \ZZ/2
\]
\end{lemma}

Consequently,

\begin{proposition}
There exists a graded class $r^n_t=r^n_{t,1}+r^n_{t,5}+\cdots\in H^{4*+1}(\LL_t^n;\ZZ/2)$, such that for any $f:M\rightarrow\LL_t^n$
\[
\sigma^n_1(M,f)=\langle V^2_M\cdot f^* r^n_t,[M]\rangle\in \ZZ/2
\]
where $M$ is a $\ZZ/2$ manifold.
\end{proposition}

$\LL^n$ has an additional structure like $\LL^s$, namely, 
\[
b_{t,t'}:\LL^n_t\times\LL^n_{t'}\rightarrow \LL^n_{t+t'}
\]

It is like the symmetric case that the homotopy equivalence between different components is also given by a composition of maps
\[
b_t:\LL^n_0=\LL^n_0\times \pt\rightarrow \LL^n_0\times \LL^n_t\xrightarrow{b_{0,t}} \LL^n_t
\]

Then we have that
\begin{proposition}
\[
b^*_t k^n_t=k^n_0
\]
\[
b^*_t l^n_t=\rho_8(t)+l^n_0
\]
\[
b^*_t r^n_t=r^n_0
\]
where $\rho_8:\ZZ_{(2)}\rightarrow \ZZ/8$.    
\end{proposition}

In particular,

\begin{proposition}
\[
l^n_{t,0}=\rho_8(t)\in H^{0}(\LL_t^n;\ZZ/8)\simeq \ZZ/8
\]
\end{proposition}

We can also formulate these classes in terms of characteristic classes for presheaves.

\begin{proposition}
For any presheaf $\mathcal{N}$ of $0$-connective, $1$-Poincar\'e normal chains over $X$, there exist graded characteristic classes
\[
r^n(\mathcal{N})=r^n_1+r^n_5+\cdots\in H^{4*+1}(X;\ZZ/2)
\]
\[
k^n(\mathcal{N})=k^n_3+k^n_7+\cdots\in H^{4*+3}(X;\ZZ/2)
\]
\[
l^n(\mathcal{N})=l^n_0+l^n_4+l^n_8+\cdots\in H^{4*}(X;\ZZ/8)
\]
which are invariant under bordism.
\end{proposition}

Like the case of quadratic and symmetric case, we also have that

\begin{theorem}\label{characteristic-classes-11}
Localized at prime $2$, 
\[
\LL^n\simeq K(\ZZ,0)\times \prod_{k>0} (K(\ZZ/8,4k)\times K(\ZZ/2,4k-3)\times K(\ZZ/2,4k-1))
\]
where the homotopy equivalence is given by the classes $l^n,r^n,k^n$.
\end{theorem}

\subsection{Relations of Characteristic Classes}

We show the relations of characteristic classes among different $L$-theories by the natural fibration
\[
\LL^q\xrightarrow{i} \LL^s\xrightarrow{p} \LL^n 
\]
One needs to be careful that the natural image of $i$ is contained in $\LL^s_0$.

\begin{proposition}
\[
i^* l^s_0=8l^q
\]
\[
i^* r^s_0=0
\]
\[
p^*k^n_t=0
\]
\[
p^* l^n_t=\rho_8 l^s_t
\]
\[
p^*r^n_t=r^s_t
\]

where $\rho_8:\ZZ_{(2)}\rightarrow\ZZ/8$.
\end{proposition}

\begin{proof}
Let $M$ be a $\ZZ$ or $\ZZ/2^k$ manifold.

First let $f:M\rightarrow \LL^q$ be a map and let $\mathcal{C}^q$ be the associated presheaf of quadratic chains. Then
\[
\sigma^s_0(M,i\circ f) =\Sign(\mathcal{C}^q(M)) =8\sigma^q_0(M,f)+8\langle \beta(V_MSq^1 V_M\cdot f^* k^q),[M]\rangle
\]
The term $8\langle \beta(f^* k^qV_MSq^1V_M),[M]\rangle$ vanishes since $\langle \beta(V_MSq^1 V_M\cdot f^* k^q),[M]\rangle$ is a $2$-torsion. Thus we proved the first equation.

The second equality is obvious since the de Rham invariant of a Poincar\'e quadratic chain always vanishes.

Now let $g:M\rightarrow \LL^s_t$ be a map and let $\mathcal{D}^s$ be the associated presheaf of symmetric chains. Notice that $\mathcal{D}^s(M)$ does not have the quadratic part as a symmetric-quadratic pair. Then the third equation is trivial. 

By definition, when $k=3$,
\[
\sigma^n_0(M,p\circ g)=\Sign(\mathcal{D}^s(M))=\sigma^s_0(M,g)\in \ZZ/8
\]
which proves the fourth equation.

When $k=1$.
\begin{eqnarray*}
\sigma^n_1(M,p\circ g) & = & \sigma^n_1(\mathcal{D}^s(M),0)-\langle V_MSq^1V_m\cdot g^*p^*\rho_2l^n_t,[M]\rangle \\
& = & \sigma^s_1(\mathcal{D}^s(M))-\langle V_MSq^1V_m\cdot g^*\rho_2l^s_t,[M]\rangle
=\sigma^s_0(M,g)   
\end{eqnarray*}
where $\rho_2$ means either $\ZZ_{(2)}\rightarrow \ZZ/2$ or $\ZZ/8\rightarrow \ZZ/2$.

Thus the fifth equation holds.

\end{proof}

It comes to the relation of the classes of the quadratic and normal theories.

Let $\LL^q(1)$ be the first space in the spectrum $\LL^q$. $\LL^q(1)$ is also the delooping of $\LL^q$.

Mimic the discussion of the space $\LL^q$. Let $M^m$ be a $\ZZ$ or $\ZZ/2^k$ manifold with a map $f:M\rightarrow \LL^q(1)$ and let $\mathcal{C}_f$ the associated presheaf. Then $\mathcal{C}_f(M)$ is a $\ZZ/2^k$ Poincar\'e quadratic chain of dimension $m-1$.

If $k=1$, then define 
\[
\widetilde{\sigma}^q_3=\sigma^q_2(\mathcal{C}_f(M))\in \ZZ/2
\]

Like the previous section, we have that

\begin{proposition}
There exists a graded class 
\[
\widetilde{k}^q=\widetilde{k}^q_3+\widetilde{k}^q_7+\cdots\in H^{4*+3}(\LL^q(1);\ZZ/2)
\]
such that for any map $f:M\rightarrow \LL^q(1)$,
\[
\widetilde{\sigma}^q_3(M,f)=\langle V^2_M\cdot f^*\widetilde{k}^q,[M]\rangle \in \ZZ/2
\]
where $M$ is a $\ZZ/2$ manifold.
\end{proposition}

Next, define 
\[
\widetilde{\sigma}^q_1(M,f)=\sigma^q_0(\mathcal{C}_f(M))-\langle\beta(V_MSq^1 V_M\cdot f^* \widetilde{k}^q),[M]\rangle\in\ZZ \,\,\text{or}\,\, \ZZ/2^k
\]

\begin{proposition}
There exists a graded class 
\[
\widetilde{l}^q=\widetilde{l}^q_5+\widetilde{l}^q_9+\cdots\in H^{4*+1}(\LL^q(1);\ZZ_{(2)})
\]
such that for any map $f:M\rightarrow \LL^q(1)$,
\[
\widetilde{\sigma}^q_1(M,f)=\langle L_M\cdot f^*\widetilde{l}^q,[M]\rangle \in \ZZ \,\,\text{or}\,\, \ZZ/2^k
\]
where $M$ is a $\ZZ$ or $\ZZ/2^k$ manifold.
\end{proposition}

With the same argument as before, 

\begin{theorem}
Localized at prime $2$, 
\[
\LL^q(1) \simeq \prod_{k>0} (K(\ZZ_{(2)},4k+1)\times K(\ZZ/2,4k-1))
\]
where the homotopy equivalence is given by the $\widetilde{l}^q$ and $\widetilde{k}^q$  classes.
\end{theorem}

Consider the connecting map $\partial_t: \LL^n_t\rightarrow \LL^q(1)$.

\begin{proposition}
\[
\partial^*_t \widetilde{k}^q = k^n_t
\]
\[
\partial^*_t\widetilde{l}^q = - \beta l^n_t
\]
where $\beta$ is the $\ZZ/8\rightarrow \ZZ_{(2)}$ Bockstein.
\end{proposition}

\begin{proof}
Let $M^m$ be $\ZZ/2^p$ manifold with a map $f:M\rightarrow \LL^n_t$ and let $(\mathcal{D},\mathcal{E})$ be the associated presheaf of symmetric-quadratic pairs.

If $p=1$, then $\sigma^n_3(M,f)=\sigma^q_2(\mathcal{E}(M))=\widetilde{\sigma}^q_3(M,\partial\circ f)$. Hence, the first equation holds.

If $p=0$, then $\widetilde{\sigma}^q_1(M,\partial\circ f)=\sigma^q_0(\mathcal{E}(M))=0$, since $\Sign(\mathcal{E}(M))=0$.

If $p=3$, then $\delta \mathcal{E}(M)$ as a Poincar\'e quadratic chain must be the boundary of some $F$ when $m\equiv 1 \pmod{4}$, namely $-\delta \mathcal{E}(M)\rightarrow F$. Thus, $\sigma^q_0(\mathcal{E}(M))=\frac{1}{8}\Sign(\mathcal{E}(M)\bigcup_{8\delta\mathcal{E}(M)} 8F)\in \ZZ/2$.

To define $\beta l^n_t$, we consider 
\[
\sigma^n_0(\delta \mathcal{D}(M),-\delta\mathcal{E}(M))=\Sign(\delta \mathcal{D}(M)\bigcup_{-\delta \mathcal{E}(M)} (-F))\in \ZZ/2
\]
Since $\mathcal{E}(M)\bigcup_{8\delta \mathcal{E}(M)} 8\delta\mathcal{D}(M)$ is the boundary of a Poincar\'e symmetic pair,
\[
8\Sign(\delta \mathcal{D}(M))+\Sign(\mathcal{E}(M))=0
\]
Then
\[
\sigma^n_0(\delta \mathcal{D}(M),-\delta\mathcal{E}(M))+\sigma^q_0(\mathcal{E}(M))=0
\]

It is not hard to check that the modified terms in defining $l$-classes are the same. The case for a general $p$ follows from the $p=3$ case. Therefore, the second equation also holds.
\end{proof}

\subsection{Coproducts of Characteristic Classes}

In this section, we calculate the coproducts of the characteristic classes induced by the ring structure or the module structure of the $L$-theories, that is,
\[
m^s_{t,t'}:\LL^s_{t}\times\LL^s_{t'}\rightarrow \LL^s_{tt'}
\]
\[
m^{q,s}_t:\LL^q\times\LL^s_t\rightarrow \LL^q
\]
\[
m^n_{t,t'}:\LL^n_t\times\LL^n_{t'}\rightarrow \LL^n_{tt'}
\]

Recall that for any two $\ZZ/n$ manifolds $M$ and $N$ there is a natural map $\rho: M\otimes N\rightarrow M\times N$. For any cohomology classes $a\in H^*(M;R),b\in H^*(N;R)$, where $R$ is a commutative ring, define $a\otimes b=\rho^*(a\times b)$. We use the notation $f\otimes g:M\otimes N\xrightarrow{\rho} M\times N\xrightarrow{f\times g} X\times Y $ for any two maps $f:M\rightarrow X,g:N\rightarrow Y$.

Recall that

\begin{proposition}[\cite{Sullivan&Morgan}*{Proposition 3.1, Proposition 8.3}]
\label{productcharacteristic}
\[
L_{M\otimes N}=L_M\otimes L_N
\]
\[
V_{M\otimes N}Sq^1V_{M\otimes N}=V^2_M\otimes V_{N}Sq^1V_{N}+V_{M}Sq^1V_{M}\otimes V^2_N
\]
In particular,
\[
V^2_{M\otimes N}=V^2_M\otimes V^2_N
\]
\end{proposition}

\begin{proposition}
\[
(m^s_{t,t'})^* l^s_{tt'}=l^s_t\times l^s_{t'}
\]
\[
(m^s_{t,t'})^* r^s_{tt'}=r^s_t\times l^s_{t'}+l^s_t\times r^s_{t'}
\]
\end{proposition}

\begin{proof}
The methods to prove the equations are the same. So we only give the proof for the first one.

Let $M$ and $N$ be $\ZZ/2^q$ manifolds with maps $f:M\rightarrow \LL^s_t$ and $g:N\rightarrow\LL^s_{t'}$. Let $\mathcal{C}_f$ and $\mathcal{C}_g$ be the associated presheaves. Let $\mathcal{C}_{f\otimes g}$ be the presheaf associated to $m^s_{t,t'}\circ (f\otimes g): M\otimes N\rightarrow \LL^s_{tt'}$. By \ref{ProducAssembly}, $\mathcal{C}_{f\otimes g}(M\otimes N)$ is bordant to
$\mathcal{C}_f(M)\otimes\mathcal{C}_g(N)$.

On the one hand,
\[
\sigma^s_0(M\otimes N,m^s_{t,t'}\circ(f\otimes g))=\Sign(\mathcal{C}_{f\otimes g}(M\otimes N))=\langle L_{M\otimes N}\cdot(f\otimes g)^*(m^s_{t,t'})^*l^s_{tt'},[M\otimes N]\rangle
\]
On the other hand,
\[
\Sign(\mathcal{C}_{f\otimes g}(M\otimes N))=\Sign(\mathcal{C}_f(M))\cdot \Sign(\mathcal{C}_g(N))=\langle L_M\cdot f^*l^s_t,[M]\rangle\cdot \langle L_N\cdot g^*l^s_{t'},[N]\rangle
\]
Hence, the first equation follows.
\end{proof}

\begin{proposition}
\[
(m^{q,s}_t)^* k^q=k^q\times l^s_t
\]
\[
(m^{q,s}_t)^* l^q=l^q\times l^s_t+\beta(k^q\times r^s_t)
\]
where $\beta$ is the $\ZZ/2\rightarrow \ZZ_{(2)}$ Bockstein. 
\end{proposition}

\begin{proof}
The proof is like the previous proposition. We only focus on the second one with the assumption that the first equation holds.

Let $M$ and $N$ be $\ZZ/2^q$ manifolds with maps $f:M\rightarrow \LL^q$ and $g:N\rightarrow\LL^s_t$ and let $\mathcal{C}_f$ and $\mathcal{C}_g$ be the associated presheaves. Again we also have the presheaf $\mathcal{C}_{f\otimes g}$.

On one hand,
\[
\sigma^q_0(M\otimes N,m^{s,q}_t\circ (f\otimes g)) = \langle L_{M\otimes N}\cdot (f\otimes g)^*(m^{q,s}_t)^*k^q,M\otimes N\rangle
\]
On the other hand,
\begin{eqnarray*}
& &\sigma^q_0(\mathcal{C}_f(M)\otimes\mathcal{C}_g(N)) -\langle\beta(V_{M\otimes N}Sq^1V_{M\otimes N}\cdot (f\otimes g)^*(m^{q,s}_t)^*k^q),[M\otimes N]\rangle\\
 & = &\sigma^q_0(\mathcal{C}_f(M))\cdot \sigma^s_0(\mathcal{C}_g(N)) + j_q(\sigma^q_2(\delta\mathcal{C}_f(M))\cdot \sigma^s_1(\mathcal{C}_g(N)) +\sigma^q_2(\mathcal{C}_f(M))\cdot \sigma^s_1(\delta\mathcal{C}_g(N)))\\
 & & - \langle \beta(V_MSq^1V_M\cdot f^*k^q),[M]\rangle\cdot \langle L_N\cdot g^*l^s_t,[N]\rangle  - \langle L_M\cdot\beta f^*k^q,[M]\rangle \cdot \langle V_NSq^1V_N\cdot g^*l^s_t,[N]\rangle \\
& & - \langle L_M\cdot f^*k^q,[M]\rangle\cdot\langle \beta(V_NSq^1V_N)\cdot g^*l^s_t,[N]\rangle
\end{eqnarray*}
where $j_q:\ZZ/2\rightarrow \ZZ/2^q$.

Furthermore,
\[
\sigma^q_0(\mathcal{C}_f(M))\cdot \sigma^s_0(\mathcal{C}_g(N)) 
=\langle L_M\cdot f^*l^q+  \beta(V_MSq^1V_M\cdot f^*k^q),[M]\rangle\cdot\langle L_N\cdot g^*l^s_t,[N]\rangle
\]
\[
j_q(\sigma^q_2(\delta\mathcal{C}_f(M))\cdot \sigma^s_1(\mathcal{C}_g(N)) )
=\langle  L_M\cdot \beta f^*k^q,[M]\rangle \cdot \langle  L_N\cdot g^*r^s_t +  V_NSq^1V_N\cdot g^*l^s_t,[N]\rangle 
\]
\[
\sigma^q_2(\mathcal{C}_f(M))\cdot \sigma^s_1(\delta\mathcal{C}_g(N)) 
=\langle  L_M\cdot f^*k^q,[M]\rangle \cdot \langle  L_N\cdot \beta g^*r^s_t +  \beta (V_NSq^1V_N)\cdot g^*l^s_t,[N]\rangle 
\]

Carefully comparing all the terms, then the second equation holds.
\end{proof}

There is a product structure on $\LL^q$ defined by the tensor product of chains
\[
m^q:\LL^q\times\LL^q\xrightarrow{1\times i} \LL^q\times\LL^s_0\xrightarrow{m^{s,q}_0} \LL^q
\]

Then

\begin{lemma}
(1) $(m^q)^* k^q=0$

(2) $(m^q)^* l^q=8\cdot l^q\times l^q$
\end{lemma}

The infinite loop structure on $\LL^q$ induces an addition $a^q:\LL^q\times \LL^q\rightarrow \LL^q$ and an inversion $\tau:\LL^q\rightarrow \LL^q$. Like the symmetric case, we have 

\begin{lemma}
\[
a^*k^q=1\times k^q+k^q\times 1
\]
\[
a^*l^q=1\times l^q+l^q\times 1
\]
\[
\tau^* k^q= k^q
\]
\[
\tau^* l^q=-l^q
\]
\end{lemma}

However, historically the product structure $\widetilde{m}^q:\LL^q\times\LL^q\rightarrow \LL^q$ people used is a different one, namely, the Whitney sum of trivializations of bundles on $G/TOP\simeq \LL^q$, or equivalently, the product of surgery problems $M\times N\rightarrow K\times L$.

The the product structures $\widetilde{m}^q$ and $m^q$ are not the same, but we can reproduce $\widetilde{m}^q$ by the module structure $m^{q,s}$.

First, define the map
\[
i_1:\LL^q\simeq \LL^q\times \pt\xrightarrow{i} \LL^s_0\times \pt\rightarrow \LL^s_0\times\LL^s_1\xrightarrow{a} \LL^s_1
\]
Then $\widetilde{m}^q$ is indeed a composition of maps
\begin{eqnarray*}
\LL^q\times\LL^q\xrightarrow{(i_1\times 1)\times (1\times i_1)\times (1\times 1)} (\LL^s_1\times\LL^q)\times(\LL^q\times\LL^s_1)\times(\LL^q\times\LL^q) \\
\xrightarrow{m^{s,q}\times m^{q,s}\times (\tau\circ m^q)} \LL^q\times\LL^q\times \LL^q\xrightarrow{a} \LL^q
\end{eqnarray*}

Then we reprove the coproducts of the characteristic classes of $\LL^q$ in \cite{Sullivan&Rourke}\cite{Sullivan&Morgan}.

\begin{corollary}[\cite{Sullivan&Rourke}*{p.~407};\cite{Sullivan&Morgan}*{p.~539 and Theorem 8.8}]

(1)$(\widetilde{m}^q)^*k^q=1\times k^q+k^q\times 1$

(2)$(\widetilde{m}^q)^* l^q=1\times l^q+l^q\times 1+8\cdot l^q\times l^q$
\end{corollary}

Lastly, let us calculate coproducts of the characteristic classes of $\LL^n$. Before we prove the coproduct formulae, we need two lemmas.

\begin{lemma}
Let $M$ be a $\ZZ/2$ manifold with a map $f:M\rightarrow \LL^n_t$ and let $(\mathcal{D}_f,\mathcal{E}_f)$ be the associated presheaf of symmetric-quadratic pairs. Then
\[
\sigma^n_{0,2}(M,f)=\langle V^2_M\cdot f^*\rho_2l^n_t,[M]\rangle
\]
where $\rho_2:\ZZ/8\rightarrow \ZZ/2$.
\end{lemma}

\begin{proof}
Recall the definition of $\sigma^n_{0,2}$. Four copies $4M$ is a $\ZZ/8$ manifold and 
\[
\sigma^n_{0,2}(\mathcal{D}_f(M),\mathcal{E}_f(M)) =\sigma^n_0(4(\mathcal{D}_f(M),\mathcal{E}_f(M)))
\]
Also,
\[
\sigma^n_0(4M,4f)=\sigma^n_0(4(\mathcal{D}(M),\mathcal{E}(M)))-j_8\langle V_{4M}Sq^1V_{4M}\cdot f^*k^n_t,4[M]\rangle\in 4\ZZ/8\simeq \ZZ/2
\]
Obviously, 
\[
\langle V_{4M}Sq^1V_{4M}\cdot f^*k^n_t,4[M]\rangle= 4\cdot \langle V_{M}Sq^1V_{M}\cdot f^*k^n_t,[M]\rangle=0\in\ZZ/2
\] 
Hence, the equation holds.
\end{proof}

For the same reason, we also have
\begin{lemma}
Let $M$ be a $\ZZ/4$ manifold with a map $f:M\rightarrow \LL^n_t$ and let $(\mathcal{D}_f,\mathcal{E}_f)$ be the associated presheaf of symmetric-quadratic pairs. Then
\[
\sigma^n_{0,4}(M,f)=\langle \rho_4 L_M\cdot f^*\rho_4l^n_t,[M]\rangle
\]
where $\rho_4:\ZZ/8\rightarrow \ZZ/4$.
\end{lemma}

\begin{proposition}
$(m^n_{t,t'})^*k^n_{tt'}=k^n_t\times \rho_2 l^n_{t'}+\rho_2 l^n_t\times k^n_{t'}$, where $\rho_2:\ZZ/8\rightarrow\ZZ/2$
\end{proposition}

\begin{proof}
Let $M$ and $N$ be $\ZZ/2$ manifolds with maps $f:M\rightarrow \LL^n_t$ and $g:N\rightarrow\LL^n_{t'}$. Let $(\mathcal{D}_f,\mathcal{E}_f)$ and $(\mathcal{D}_{g},\mathcal{E}_g)$ be the associated presheaves. There is also a presheaf $(\mathcal{D}_{f\otimes g},\mathcal{E}_{f\otimes g})$ associated to $m^n_{t,t'}\circ (f\otimes g)$.

Then 
\[
(\mathcal{D}_{f\otimes g}(M),\mathcal{E}_{f\otimes g}(N))=(\mathcal{D}_f(M),\mathcal{E}_f(M))\otimes (\mathcal{D}_{g}(N),\mathcal{E}_g(N))
\]

On one hand, 
\[
\sigma^n_3(M\otimes N,m^n_{t,t'}\circ (f\otimes g))=\langle V^2_{M\otimes N}\cdot (f\otimes g)^*(m^n_{t,t'})^*k^n_{tt'},[M\otimes N]\rangle 
\]

On the other hand,
\begin{eqnarray*}
\sigma^n_3(M\otimes N,m^n_{t,t'}\circ (f\otimes g))
& = & \langle V^2_M \cdot f^\rho_2l^n_t,[M]\rangle \cdot \langle V^2_N \cdot g^*k^n_{t'},[N]\rangle \\
& & +\langle V^2_M \cdot f^*k^n_t,[M]\rangle \cdot\langle V^2_N \cdot g^*\rho_2l^n_{t'},[N]\rangle    
\end{eqnarray*}
\end{proof}

To show the coproduct of the $\ZZ/8$ class $l^n$, we need the following lemma in \cite{Brumfiel&Morgan}.

\begin{lemma}[\cite{Brumfiel&Morgan}*{Lemma 9.3}]
$H_*(X\times Y,\ZZ/2^k)$ is generated by the Hurewicz image of the followings.

(1) $j_{2^k}(f_*[M]\times g_*[N])$, where $M$ and $N$ are $\ZZ/2^l$ manifolds with $l\leq k$, $f:M\rightarrow X$, $g:N\rightarrow Y$ and $j_{2^k}:\ZZ/2^l\rightarrow \ZZ/2^k$.

(2) $\rho_{2^k}\delta (f_*[P]\times g_*[Q])$, where $P$ and $Q$ are $\ZZ/2^l$ manifolds with $l<k$, $f:P\rightarrow X$, $g:Q\rightarrow Y$, $\rho_{2^k}:\ZZ\rightarrow \ZZ/2^k$.
\end{lemma}

\begin{lemma}
Let $M$ and $N$ be two $\ZZ/8$ manifolds with maps $f:M\rightarrow \LL^n_t$ and $g:N\rightarrow \LL^n_{t'}$.
Then
\begin{multline*}
\langle L_{M\otimes N}\cdot (f\otimes g)^*(m^n_{t,t'})^*l^n_{tt'},[M\otimes N]\rangle= \\
\langle L_{M\otimes N}\cdot (f^*l^n_t\otimes g^*l^n_{t'}),[M\otimes N]\rangle +j_8\langle V^2_{M\otimes N}\cdot (f^*r^n_t\otimes g^*k^n_{t'}+f^*k^n_t\otimes g^*r^n_{t'})),[M\otimes N]\rangle  
\end{multline*}
where $j_8:\ZZ/2\rightarrow \ZZ/8$.
\end{lemma}

\begin{proof}
Let $(\mathcal{D}_f,\mathcal{E}_f)$ and $(\mathcal{D}_g,\mathcal{E}_g)$ be the associated presheaves. Let $(\mathcal{D}_{f\otimes g},\mathcal{E}_{f\otimes g})$ the presheaf associated to $m^n_{t,t'}\circ (f\otimes g)$.

On one hand,
\begin{eqnarray*}
& & \sigma^n_0(\mathcal{D}_{f\otimes g}(M\otimes N),\mathcal{E}_{f\otimes g}(M\otimes N)) \\
& = & \langle L_{M\otimes N}\cdot (f\otimes g)^*(m^n_{t,t'})^*l^n_{tt'},[M\otimes N]\rangle \\
& & + j_8\langle (V_{M}Sq^1V_{M}\otimes L_N+L_M\otimes V_N Sq^1V_N)\cdot (f^*k^n_t\otimes g^*l^n_{t'}+f^*l^n_t\otimes g^*k^n_{t'}),[M\otimes N]\rangle    
\end{eqnarray*}

On the other hand, applying the chain-level product formula to $\sigma^n_0(\mathcal{D}_{f\otimes g}(M\otimes N),\mathcal{E}_{f\otimes g}(M\otimes N))$ and noticing that 
\[
j_8\langle V_{M} Sq^1V_M f^*k^n_{t},[M]\rangle\cdot j_8\langle V_{N}Sq^1V_N g^*k^n_{t'},[N]\rangle=0\in \ZZ/8
\]
(since $4\times 4=0\in \ZZ/8$), we can check that the equation holds.
\end{proof}

We can prove the $\ZZ/2$ and $\ZZ/4$ cases by reducing to the $\ZZ/8$ case. Notice that the $j_8$-term vanishes in these two cases.

\begin{lemma}
Let $M$ and $N$ be two $\ZZ/2$ manifolds of dimension $m$ and $n$ with maps $f:M\rightarrow \LL^n_t$ and $g:N\rightarrow \LL^n_{t'}$. Then
\[
\langle L_{M\otimes N}\cdot (f\otimes g)^*(m^n_{t,t'})^*\rho_2 l^n_{tt'},[M\otimes N]\rangle=\langle L_{M\otimes N}\cdot (f^*\rho_2 l^n_t\otimes g^*\rho_2 l^n_{t'}),[M\otimes N]\rangle
\]
where $\rho_2:\ZZ/8\rightarrow \ZZ/2$.
\end{lemma}

\begin{lemma}
Let $M$ and $N$ be two $\ZZ/4$ manifolds of dimension $m$ and $n$ with maps $f:M\rightarrow \LL^n_k$ and $g:N\rightarrow \LL^n_l$. Then
\[
\langle L_{M\otimes N}\cdot (f\otimes g)^*(m^n_{t,t'})^*\rho_4 l^n_{tt'},[M\otimes N]\rangle=\langle L_{M\otimes N}\cdot (f^*\rho_4 l^n_t\otimes g^*\rho_4 l^n_{t'}),[M\otimes N]\rangle
\]
where $\rho_4:\ZZ/8\rightarrow \ZZ/4$.
\end{lemma}

For the Bockstein case, we have

\begin{lemma}
Let $M$ and $N$ be two $\ZZ/8$ manifolds with maps $f:M\rightarrow \LL^n_t$ and $g:N\rightarrow \LL^n_{t'}$.
Then
\begin{multline*}
\langle L_{M\otimes N}\cdot (f\otimes g)^*(m^n_{t,t'})^*l^n_{tt'},[\delta(M\otimes N)]\rangle \\
= \langle L_{M\otimes N}\cdot f^*l^n_t\otimes g^*l^n_{t'},[\delta (M\otimes N)] \rangle +j_8\langle V^2_{M\otimes N}\cdot (f^*r^n_t\otimes g^*k^n_{t'}+f^*k^n_t\otimes g^*r^n_{t'}),[\delta(M\otimes N)]\rangle    
\end{multline*}
where $j_8:\ZZ/2\rightarrow \ZZ/8$.
\end{lemma}

\begin{proof}
As before, let $(\mathcal{D}_f,\mathcal{E}_f)$ and $(\mathcal{D}_g,\mathcal{E}_g)$ be the associated presheaves. Let $(\mathcal{D}_{f\otimes g},\mathcal{E}_{f\otimes g})$ the presheaf associated to $m^n_{t,t'}\circ (f\otimes g)$.

On one hand,
\begin{eqnarray*}
 & & \sigma^n_0(\mathcal{D}_{\delta(f\otimes g)}(\delta(M\otimes N)),\mathcal{E}_{\delta(f\otimes g)}(\delta(M\otimes N))) \\
& = & \langle L_{M\otimes N}\cdot (f\otimes g)^*(m^n_{t,t'})^*l^n_{tt'},[\delta (M\otimes N)]\rangle \\
& & + j_8\langle (V_{M}Sq^1V_{M}\otimes V^2_N+V^2_M\otimes V_N Sq^1V_N)\cdot (f^*k^n_t\otimes g^*l^n_{t'}+f^*l^n_t\otimes g^*k^n_{t'}),[\delta (M\otimes N)]\rangle
\end{eqnarray*}

Notice that
\[
[\delta (M\otimes N)]= [\delta M]\otimes [N]+[M] \otimes [\delta N] \in H_{m+n-1}(M\otimes N;\ZZ/8)   
\]

On the other hand, applying the chain-level product formula to $\sigma^n_0(\mathcal{D}_{\delta(f\otimes g)}(\delta(M\otimes N)),\mathcal{E}_{\delta(f\otimes g)}(\delta(M\otimes N)))$, we can check that the equation holds.
\end{proof}

The $\ZZ/2$ and $\ZZ/4$ Bockstein cases reduce to the $\ZZ/8$ Bockstein case as well. It follows that

\begin{proposition}
\[
(m^n_{t,t'})^*l^n_{tt'}=l^n_k\times l^n_{t'}+j_8 (r^n_t\times k^n_{t'}+k^n_t\times r^n_{t'})
\]
where $j_8:\ZZ/2\rightarrow\ZZ/8$.
\end{proposition}

In particular,
\begin{proposition}
\[
(m^n_{t,t'})^*\rho_4 l^n_{tt'}=\rho_4 l^n_t\times \rho_4 l^n_{t'}
\]
\[
(m^n_{t,t'})^*\rho_2 l^n_{tt'}=\rho_2 l^n_t\times \rho_2 l^n_{t'}
\]
where $\rho_4:\ZZ/8\rightarrow\ZZ/4$ and $\rho_2:\ZZ/8\rightarrow\ZZ/2$.
\end{proposition}

\begin{proposition}
\[
(m^n_{t,t'})^*r^n_{tt'}=\rho_2 l^n_t\times r^n_{t'}+r^n_t\times\rho_2 l^n_{t'}
\]
where $\rho_2:\ZZ/8\rightarrow \ZZ/2$.
\end{proposition}

\begin{proof}
Let $M$ and $N$ be $\ZZ/2$ manifolds with maps $f:M\rightarrow \LL^n_t$ and $g:N\rightarrow \LL^n_{t'}$. Let $(\mathcal{D}_f,\mathcal{E}_f)$ and $(\mathcal{D}_g,\mathcal{E}_g)$ be the associated presheaves. Let $(\mathcal{D}_{f\otimes g},\mathcal{E}_{f\otimes g})$ be the presheaf associated to $m^n_{t,t'}\circ (f\otimes g)$.

On one hand,
\begin{eqnarray*}
& & \sigma^n_1(\mathcal{D}_{f\otimes g}(M\otimes N),\mathcal{E}_{f\otimes g}(M\otimes N)) \\
& = & \langle V^2_{M\otimes N}\cdot (f\otimes g)^*(m^n_{k,l})^*r^n_{tt'},[M\otimes N]\rangle \\
& & + \langle (V_{M}Sq^1V_{M}\otimes V^2_N+V^2_M\otimes V_NSq^1V_N)\cdot (f^*\rho_2 l^n_t\otimes g^*\rho_2 l^n_{t'}),[M\otimes N]\rangle
\end{eqnarray*}

On the other hand,
\begin{eqnarray*}
& & \sigma^n_{0,2}(\mathcal{D}_{f}(M),\mathcal{E}_{f}(M))\cdot \sigma^n_1(\mathcal{D}_{g}(N),\mathcal{E}_{g}(N)) 
 +\sigma^n_{1}(\mathcal{D}_{f}(M),\mathcal{E}_{f}(M))\cdot \sigma^n_{0,2}(\mathcal{D}_{g}(N),\mathcal{E}_{g}(N)) \\
 & = & \langle V^2_M\cdot f^*\rho_2l^n_t,[M]\rangle \langle V^2_N \cdot g^*r^n_{t'} + V_NSq^1V_N \cdot g^*\rho_2l^n_{t'},[N]\rangle) \\
 & & + \langle V^2_M \cdot f^*r^n_t + V_M Sq^1V_M \cdot f^*\rho_2l^n_t,[M]\rangle\cdot\langle V^2_N\cdot g^*\rho_2l^n_{t'},[N]\rangle
\end{eqnarray*}

\end{proof}

\subsection{Characteristic Classes of Bundle Theory}

Because of the equivalence of Ranicki's formulation and Wall's formulation of $L$-groups, it is quite obvious that the classes $k^q$ and $l^q$ of $\LL^q$ are equivalent to the Kervaire class and the $l$-class for the surgery space $G/TOP$.

Recall from Section $2$, there is a graded characteristic class $l^{TOP}\in H^{4*}(B;\ZZ_{(2)})$ for any $TOP$ bundle over $B$. Moreover, there are graded characteristic classes $l^{G}\in H^{4*}(B;\ZZ/8)$ and $k^G\in H^{4*+3}(B;\ZZ/2)$ for any spherical fibration over $B$. In the $2$-local sense, the spherical fibration has a $TOP$ bundle structure if and only if the $k^G$-class vanishes and the $\ZZ/8$ class $l^G$ has a $\ZZ_{(2)}$ lifting. For $2$-local $TOP$ bundles, the $\ZZ/8$ reduction of $l^{TOP}$ is $l^G$.

In Section $2$, we said that the bundle theory also has an integral description.

\begin{theorem}(\cite[Proposition 16.1]{Levitt&Ranicki})
\label{RanickiLevitt}
Suppose $X$ is a finite simplicial complex.

(1) A (k-1)-spherical fibration $\nu:X\rightarrow BSG(k)$ has a canonical $\LL^n$-orientation $U^n(\nu):T(\nu)\rightarrow \Sigma^k\LL^n$, where $T(\nu)$ is the Thom space of $\nu$.

(2) A topological block bundle $\mu:X\rightarrow B\widetilde{STOP}(k)$ has a canonical $\LL^s$-orientation $U^s(\mu):T(\mu)\rightarrow \Sigma^k\LL^s$ so that its $\LL^n$-reduction is $U^n(\mu)$.

(3) A difference between any two stable topological bundle liftings $\mu,\mu'$ of the same spherical fibration $\nu$ is represented by an element $d(\mu,\mu')\in (\LL^q)^0(X)$.
\end{theorem}

\begin{remark}
The universal $\LL^s$-orientation $U^s:M\widetilde{STOP}(k)\rightarrow \Sigma^k\LL^s$ of the block bundle theory induces a universal $\LL^s$-orientation $U^s:MSTOP(k)\rightarrow \Sigma^k\LL^s$ for the (micro-)TOP bundle theory, by the natural inclusion $STOP(k)\rightarrow \widetilde{STOP}(k)$. Moreover, in the stable range $STOP\rightarrow \widetilde{STOP}$ is a homotopy equivalence (\cite[Corollary 4.11]{Rourke-Sanderson1970II}).
\end{remark}

In this subsection, we prove that the localization at prime $2$ of Levitt-Ranicki's theory is equivalent to Brumfiel-Morgan and Morgan-Sullivan's theories \cite{Sullivan&Morgan}\cite{Brumfiel&Morgan}. That is, we will prove that the classes of $\LL^s,\LL^n$ we constructed before correspond to the characteristic classes defined in \cite{Sullivan&Morgan}\cite{Brumfiel&Morgan} under Levitt-Ranicki' $\LL$-theory orientations. 

Before doing that, let us briefly review how to construct the $\LL$-orientations $U^n$ and $U^s$. 

Let $\xi\rightarrow B$ be an oriented spherical fibration. Let $D(\xi)$ be the corresponding disc bundle (the mapping cylinder) and let $Th(\xi)$ be the Thom space. In the singular simplicial complex $S(Th(\xi))$, there is a subcomplex $N(Th(\xi))$ which consists of maps $f:\Delta^n\rightarrow Th(\xi)$ such that $f^{-1}(D(\xi))$ is an $n$-ad of normal spaces of dimension $n-k$ with respect to the pullback bundle. Notice that the inclusion map $N(Th(\xi))\rightarrow S(Th(\xi))$ is a homotopy equivalence. Then the singular chains of $n$-ads of normal spaces induce a simplicial map $N(Th(\xi))\rightarrow \Sigma^k \LL^n_1$, which is the canonical $\LL^n$-orientation.

In $N(Th(\xi))$, let $T(Th(\xi))$ be the subcomplex consisting of maps $f:\Delta^n\rightarrow Th(\xi)$ such that $f$ is Poincar\'e transversal, namely, $f^{-1}(D(\xi))$ is an $n$-ad of $\ZZ$-coefficient homology Poincar\'e space of dimension $n-k$ with the fundamental class induced from the normal structure. Similarly, the singular chains of $n$-ads of homology Poincar\'e spaces induce a simplicial map $T(Th(\xi))\rightarrow \LL^s_1$.

There exists a $TOP$ structure of $\xi$ if and only $\xi$ has a theory of transversality (see section $2$ or \cite{Levitt&Morgan}), if and only if there is a canonical homotopy inverse of the inclusion $T(Th(\xi))\rightarrow N(Th(\xi))$ (\cite[Theorem 1.11]{Levitt&Ranicki}). Furthermore, $T(Th(\xi))\rightarrow \LL^s_1$ is the canonical $\LL^s$-orientation.
 
\begin{theorem}\label{characteristic-classes-4}
\[
(U^s)^*l^s_1=l^{TOP} \cdot U\in \widetilde{H}^{4*}(M\widetilde{STOP};\ZZ_{(2)})
\]
\end{theorem}

\begin{proof}
Let $M$ be a $\ZZ/2^q$ $PL$ manifold with a map $f:M\rightarrow M\widetilde{STOP}(h)$. Due to the transversality theorem of topological manifolds (\cite[(9.6C)]{Freedman&Quinn}), we can homotope $f$ so that it is transversal to the zero section over each simplex, i.e., for each simplex $\Delta^n$ of $M$, $f^{-1}(B\widetilde{STOP}(h))$ is an $n$-ad of $\ZZ/2^q$ topological manifolds of dimension $n-h$. The assembly of the $n$-ads is a $\ZZ/2$ topological submanifold $L$ of $M$.

By our construction,
\[
\Sign(L)=\sigma^s_0(M,U^s\circ f)=\langle L_M\cdot f^* (U^s)^*\Sigma^h l^s_1,[M]\rangle
\]

Also,
\[
\Sign(L)=\langle L_M\cdot f^* (l^{TOP}\cdot  U),[M]\rangle
\]
where $U\in \widetilde{H}^h(M\widetilde{STOP}(h))$ is the universal Thom class and $l^{TOP}\in H^{4*}(B\widetilde{STOP};\ZZ_{(2)})$ is the class defined in\cite{Sullivan&Morgan}.

After passage to the stable range, we complete the proof.    
\end{proof}

\begin{theorem}\label{characteristic-classes-5}
\[
(U^s)^* r^s_1= VSq^1V\cdot U \in \widetilde{H}^{4*+1}(M\widetilde{STOP};\ZZ/2)
\]
In particular, $(U^s)^* r^s_{1,1}=0$.
\end{theorem}

\begin{proof}
Let $M$ be a $\ZZ/2$ $PL$ manifold. Then
\begin{eqnarray*}
\dR(L) & = & \sigma^s_1(M,U^s\circ f) \\
 & = & \langle V^2_M\cdot f^*(U^s)^*\Sigma^h r^s_1,[M]\rangle
+\langle V_MSq^1V_M\cdot f^*\rho_2(l^{TOP} \cdot U),[M]\rangle
\end{eqnarray*}

Let $\tau_M$ be the tangent bundle of $M$ and $\nu$ be the normal bundle of $L\subset M$. Then 
\begin{eqnarray*}
\dR(L) & = & \langle V_LSq^1V_L,[L]\rangle \\
 & = & \langle V^2_{\tau_M\vert_L}\cdot f^*(V_{\nu}Sq^1V_{\nu}),[L]\rangle
+\langle V_{\tau_M\vert_L}Sq^1V_{\tau_M\vert_L}\cdot f^*L_{\nu},[L]\rangle \\
 & = & \langle L_M\cdot f^*(VSq^1V\cdot U),[M]\rangle
+\langle V_MSq^1V_M\cdot f^*\rho_2(l^{TOP} \cdot U),[M]\rangle
\end{eqnarray*}   
\end{proof}

These two theorems prove the following.

\begin{theorem}\label{main-result-1}
At prime $2$, Ranicki-Levitt's symmetric $L$-theory orientation for $TOP$ bundles is equivalent to Morgan-Sullivan's $2$-local characteristic classes.
\end{theorem}

\begin{theorem}\label{characteristic-classes-1}
\[
(U^n)^* k^n_1=k^G\cdot U\in \widetilde{H}^{4*+3}(MSG;\ZZ/2)
\]
\end{theorem}

\begin{proof}
Let M be a $\ZZ/2^q$ manifold of dimension $m+h$ with a map $f:M\rightarrow MSG(h)$. By a slight homotopy we can assume that $f$ is transversal to the spherical fibration $S(ESG(h))\subset  MSG(h)$ and let $I=f^{-1}(D(ESG(h)))$.

Embed $M$ in a sphere $D^{N+m+h}$, where there is a $\ZZ/2^q$-action on the boundary $S^{N+m+h}$ so that $\partial M=\bigcup_{2^q}\delta M$ is equivariantly embedded in $S^{N+m+h}$. Consider the corresponding Pontryagin-Thom construction $F:D^{N+m+h}\rightarrow MSG(h)\wedge MSPL(N)$. Let $N(M)$ be the tubular neighborhood of $M$ in $D^{N+m+h}$ such that the preimage of the disc bundle in the $MSPL(N)$ is $N(M)$ under $F$. Since $MSG(h)\wedge MSPL(N)\simeq M(D(ESG(h))\times ESPL(N))$ is again the Thom space of some spherical fibration, the preimage of the total disc bundle in $MSG(h)\wedge MSPL(N)$ is the restriction $N(M)\vert_I$.

Let $U_{MSG(h)}\in \widetilde{H}^{h}(MSG(h);\ZZ)$ and $U_{MSG(h)\wedge MSPL(N)}\in \widetilde{H}^{h+N}(MSG(h)\wedge MSPL(N);\ZZ)$ both be the Thom class. Let $x=[M]\cap U_{MSG(h)}\in H_m(I,I\bigcap \partial M;\ZZ)$ and $y= [D^{N+m+h}]\cap U_{MSG(h)\wedge MSPL(N)}\in H_m(N(M)\vert_I,N(M)\vert_{I\bigcap \partial M};\ZZ)$. Then $x$ and $y$ induce $\ZZ/2^q$ symmetric structures on $C_*(I)$ and $C_*(N(M)\vert_I)$ respectively. But the natural inclusion $I\rightarrow N(M)\vert_I$ induces a chain homotopy equivalence between the two $\ZZ/2^q$ symmetric chains.
 
Brumfiel-Morgan's obstruction for cobording $f$ to be Poincar\'e transversal is their obstruction class for $F:S^{N+m+h}\rightarrow MSG(h)\wedge MSPL(N)$ (see section $2$). The obstruction is equivalent to whether $C_*(N(M)\vert_I)\simeq C_*(I)$ satisfies $\ZZ/2^q$ Poincar\'e duality. By Ranicki's miracle lemma \ref{Ranickimiracle}, it is equivalent to the bordism class of $\partial C_*(I)$ in $L^q_{m-1}(\ZZ,\ZZ/2^q)$ (for the quadratic structure on $\partial C_*(I)$, see \cite[Proposition 7.4.1]{Ranicki1981}).

On the other hand, associated to the composition map $U^n\circ f$, there is a presheaf $(\mathcal{D}_f,\mathcal{E}_f)$ of Poincar\'e symmetric-quadratic pair over $M$, whose assembly is exactly $(C^{m-*}(I),\partial C_*(I))$.

When $m\equiv 2\pmod{2}$, the bordism class of $\partial C_*(I)$ in $L^q_{m-1}$ is determined by the Kervaire invariant $\sigma^n_3(\mathcal{D}_f,\mathcal{E}_f)$. Hence, we may assume that $M$ is a $\ZZ/2$ manifold.

Then the rest is obvious by recalling the construction of $k^G\in H^{4*+3}(BSG;\ZZ/2)$ in \cite[Theorem 5.4]{Brumfiel&Morgan}
\end{proof}

\begin{theorem}\label{characteristic-classes-2}
\[
(U^n)^* l^n_1=l^G\cdot U\in \widetilde{H}^{4*}(MSG;\ZZ/8)
\]

\end{theorem}

\begin{proof}
Recall the construction of $l^G$ in \cite[Section 8]{Brumfiel&Morgan}. Let $M$ be a $\ZZ/8$ manifold. When the map $f:M\rightarrow MSG(h)$ is Poincar\'e transversal, i.e., the preimage $I$ is a presheaf of ads of $\ZZ/8$ homology Poincar\'e spaces over $M$. Then
\[
\Sign(I)=\langle L_M \cdot f^*(l^G\cdot U),[M]\rangle+j_8\langle V_MSq^1V_M \cdot f^*(k^G\cdot U),[M]\rangle \in\ZZ/8
\]
where $j_8:\ZZ/2\rightarrow \ZZ/8$.

We also know that
\begin{eqnarray*}
\Sign(I) & = & \sigma^n_0(\mathcal{D}_f(M),\mathcal{E}_f(M))+j_8\langle V_MSq^1V_M\cdot f^*(U^n)^*\Sigma^hk^n_1,[M]\rangle \\
 & = & \langle L_M\cdot f^*(U^n)^*\Sigma^hl^n_1,[M]\rangle + j_8\langle V_MSq^1V_M\cdot f^*(k^G\cdot U),[M]\rangle
\end{eqnarray*}

Hence, under the assumption of Poincar\'e transversality for $f$,
\[
\langle L_M f^*(l^G\cdot U),[M]\rangle=\langle L_M f^*(U^n)^*\Sigma^hl^n_1,[M]\rangle
\]

Recall from section $2$ that \cite[p.~61]{Brumfiel&Morgan} constructed a map $a:K^{h+4}\rightarrow MSG(h)$ for a $\ZZ/2$ manifold $K^{h+4}=S^{h+3}\times I/ (x,0)\sim (-x,1)$ such that

(1) the Kervaire obstruction to the Poincar\'e transversality of $a\vert_{\delta K}$ is $1\in\ZZ/2$;

(2) $\langle a^*(V^2\cdot U),[K]\rangle=0\in\ZZ/2$, where $V$ is the sum of even Wu classes. 

We are left to prove that the assembled $\ZZ/2$ symmetric-quadratic Poincar\'e chain pair $(\mathcal{D}_a(K),\mathcal{E}_a(K))$ is bordant to the chosen $(D'_0,E'_0)$ in the previous section. It suffices to prove that $\sigma^n_0(K,a)=0=\sigma^n_0(D'_0,E'_0)$ since $\sigma^q_2(\mathcal{E}_a(K))=\sigma^q_2(E'_0)=1$.

Let $N$ be a $\ZZ/2$ manifold of dimension congruent to $h'+3$ modulo $4$ together with a map $b:N\rightarrow MSG(h')$ such that it has nonvanishing Kervaire obstruction to Poincar\'e transversality, i.e., $\sigma^n_3(N,b)=1\in\ZZ/2$. Due to the product formulae of $L^n_*(\ZZ,\ZZ/2)$, it suffices that the Kervaire obstruction for cobording the following map to Poincar\'e transversality vanishes 
\[
K\otimes N\xrightarrow{a\otimes b}MSG(h)\times MSG(h')\xrightarrow{\Delta} MSG(h+h')
\]
Recall the coproduct formula for $k^G$ in \cite[9.2]{Brumfiel&Morgan}
\[
\Delta^* (k^G\cdot U)=(k^G\cdot U)\times (V^2\cdot U)+(V^2\cdot U)\times (k^G\cdot U)
\]
Then 
\begin{eqnarray*}
\sigma^n_3(K\otimes N,\Delta\circ(a\otimes b)) & = & \langle V^2_{K\otimes N} \cdot (a\otimes b)^*(\Delta)^*(k^G\cdot U),[K\otimes N]\rangle \\
& = & \langle (V^2_{K}\otimes V^2_N)\cdot (a^* (V^2\cdot U)\otimes b^* (k^G\cdot U)),[K\otimes N]\rangle \\
& = & \langle a^* (V^2\cdot U),[K]\rangle\cdot \langle V^2_N\cdot b^* (k^G\cdot U),[N]\rangle=0
\end{eqnarray*}

But 
\begin{eqnarray*}
\sigma^n_3(K\otimes N,\Delta\circ(a\otimes b))=\sigma^n_0(K,a)\cdot \sigma^n_3(N,b)
\end{eqnarray*}

Hence, $\sigma^n_0(K,a)=0=\sigma^n_0(D'_0,E'_0)$.

Recall from section $2$ and \cite[p.~61]{Brumfiel&Morgan} that, for a non Poincar\'e transversal $f$, the key step to define $l^G\in H^{4*}(BSG;\ZZ/8)$ is to subtract the bordism class $[K,a]\cdot [\CC P^{\frac{m-4}{2}}]$ from $[M,f]$ as a modification, which is exactly the same modification as what we did in the chain level.    
\end{proof}

Recall from \cite[8.1(ii)1]{Brumfiel&Morgan} that $\rho_2 l^G=V^2$.

\begin{corollary}
\[
\rho_2 (U^n)^* l^n_1=V^2 \cdot U \in \widetilde{H}^{4*}(MSG;\ZZ/2)
\]
\end{corollary}

Finally, We prove that $r^n_1$ corresponds to $VSq^1 V\in H^{4*+1}(BSG;\ZZ/2)$. 

Let $M$ be a $\ZZ/2$ manifold. First assume that $(M,f:M\rightarrow MSG(h))$ is Poincar\'e transversal. Then $I$ is a presheaf of $\ZZ/2$ homology Poincar\'e spaces over $M$. The presheaf is assembled to a $\ZZ/2$ Poincar\'e space $L$. The preimage of $ESG(h)$ is a spherical fibration of $L$ which induces a normal structure compatible with the Poincar\'e duality. 

Then
\begin{eqnarray*}
\dR(L) & = & \langle V_LSq^1V_L,[L]\rangle \\
 & = & \langle V^2_M \cdot f^*(VSq^1V \cdot U),[M]\rangle +\langle V_MSq^1V_M \cdot f^*(V^2 \cdot U), [M]\rangle
\end{eqnarray*}

So, under the assumption of Poincar\'e transversality for $f$,
\[
\langle V^2_M \cdot f^*(VSq^1V \cdot U),[M]\rangle=\langle V^2_M\cdot  f^*(U^n)^*r^n_1,[M]\rangle
\]

For the general case, there exists a $\ZZ/2$ manifold $P$ of dimension $h+5$ with a map $b:P\rightarrow MSG(h)$ such that 
\[
\langle b^*\beta (l^G\cdot U),[P]\rangle=1\in\ZZ/2
\]
where $\beta$ is the $\ZZ/2\rightarrow\ZZ_{(2)}$ Bockstein. It means that the map $b$ does not admit Poincar\'e transversality. 

We can assume that the associated presheaf $(\mathcal{D}_b,\mathcal{E}_b)$ has vanishing $\sigma^n_1$. Otherwise, note that the de Rham invariant of the Wu manifold $SU(3)/SO(3)$ is $1$. Then we may subtract $[P,b]$ by a map $S^{h+5}\rightarrow MSTOP(h)$ which is the Thom-Pontryagin construction associated to the Wu manifold $SU(3)/SO(3)$.

\begin{lemma}
\[
\langle V^2_P\cdot b^*(VSq^1V\cdot U),[P]\rangle=\langle V^2_P\cdot f^*(U^n)^*r^n_1,[P]\rangle\in \ZZ/2
\]
\end{lemma}

\begin{proof}
Considering the definition of $\sigma^n_1$ and the assumption that $\sigma^n_1(\mathcal{D}_b,\mathcal{E}_b)=0$, it suffices to prove that 
\[
\langle V^2_P\cdot b^*(VSq^1 V\cdot U),[P]\rangle+\langle V_PSq^1V_P\cdot b^*(V^2 \cdot U),[P]\rangle=0
\]
According to \cite{Brumfiel&Morgan}, there exists a map $c:S^{h'+3}\rightarrow MSG(h')$ such that 
\[
\langle c^*(k^G\cdot U),[S^{h'+3}]\rangle=1
\]

Also recall the coproduct formula of $l^G$ (\cite[9.1]{Brumfiel&Morgan}), i.e.,
\[
\Delta^* (l^G\cdot U)=(l^G\cdot U)\times (l^G\cdot U)+j_8((k^G\cdot U)\times (VSq^1V\cdot U)+(VSq^1V\cdot U)\times (k^G\cdot U))
\]
where $j_8:\ZZ/2\rightarrow \ZZ/8$.

Considering the product formulae for $L^n_*(\ZZ,\ZZ/8)$, we have
\begin{eqnarray*}
0 & = & \sigma^n_0(\mathcal{D}_{b\otimes c}(P\otimes S^{h'+3}),\mathcal{E}_{b\otimes c}(P\otimes S^{h'+3})) \\
 & = & j_8(\langle V^2_P\cdot b^*(VSq^1V\cdot U),[P]\rangle+\langle V_PSq^1V_P \cdot b^*(V^2\cdot U),[P]\rangle) 
\end{eqnarray*}
\end{proof}

If $f:M\rightarrow MSG(h)$ is not Poincar\'e transversal, modify it by the map 
\[
b\otimes \pt:P\otimes \CC P^{2i}\rightarrow MSG(h)
\]
when the dimension of $M$ is $m=h+4i+1$, so that the new map is conbordant to be Poincar\'e transversal. Therefore,

\begin{theorem}\label{characteristic-classes-3}
\[
(U^n)^* r^n_1=VSq^1V\cdot U\in \widetilde{H}^{4*+1}(MSG;\ZZ/2)
\]

In particular, $(U^n)^* r^n_{1,1}=0$.
\end{theorem}

Therefore, we have proved the following:

\begin{theorem}\label{main-result-2}
At prime $2$, Ranicki-Levitt's normal $L$-theory orientation for $TOP$ bundles is equivalent to Brumfiel-Morgan's $2$-local characteristic classes.
\end{theorem}

\subsection{\texorpdfstring{$L$}{Lg}-theory at Odd Primes}

To complete the story, let us briefly consider the odd-prime localization of $L$-theories and the bundle theory at odd primes. 

Since the homotopy groups of $\LL^n_0$ have only $2$-primary torsion, the odd-localization of $\LL^n$ is simply homotopy equivalent to $\ZZ$.

Sullivan (\cite[p.~218]{Sullivan2009}) proved that, localized at odd primes, $G/PL$ is homotopy equivalent to $BSO$. Since $G/PL$ and $G/TOP$ are differed by a $\ZZ/2$ twisting, $\LL^q\simeq G/TOP$ is also homotopy equivalent to $BSO$. Therefore, the same is true for $\LL^s$.

However, we want to use the Poincar\'e chain description of $\LL$-theory to reprove the same result. Logically, the reproof is essentially the same as Sullivan's proof, which is based on the a priori invariant method for $K$-theory. 

In the rest of this section, we only consider $\LL^s$. The first goal is to construct an $H$-space map $\sigma^s_{\odd}:\LL^s_1\rightarrow BSO^{\otimes}_{(\odd)}$, where the superscript $\otimes$ means that the product structure on $BSO^{\otimes}$ is induced by the tensor product of vector bundles.

Let $M$ be a $\ZZ$ or $\ZZ/n$ manifold with a map $f:M\rightarrow \LL^s_1$. Let $\mathcal{C}_f$ be the associated presheaf. Define 
\[
\sigma^s_{\odd}(M,f)=\sigma^s_0(\mathcal{C}_f(M))\in\ZZ \,\,\text{or}\,\, \ZZ/n
\]

Then the product formula of $\sigma^s_{\odd}$ holds, i.e.,
\[
\sigma^s_{\odd}((M,f)\cdot(N,g))=\sigma^s_{\odd}(M,f)\cdot \sigma^s_{\odd}(N,g)
\]
where $N$ is another $\ZZ$ or $\ZZ/n$ manifold with a map $g:N\rightarrow \LL^s_1$

In particular, we get a map 
\[
\sigma^s_{\odd}:\Omega^{SO}_*(\LL^s_1)\otimes_{\Omega_*^{SO}}\ZZ_{(\odd)} \rightarrow \ZZ_{(\odd)}
\]
where $\Omega^{SO}(*)\rightarrow  \ZZ_{(\odd)}$ is the signature map. 

Under the same argument as \cite[Lemma 4.26]{Madsen&Milgram}, $\sigma^s_{\odd}$ induces a map $\sigma^s_{\odd}:\LL^s_1\rightarrow BSO^{\otimes}_{(\odd)}$. The proof of the following lemma is essentially the same as that of \cite[Lemma 4.27]{Madsen&Milgram}

\begin{lemma}
\[
\ph(\sigma^s_{\odd})=l^s_1\otimes \QQ\in H^*(\LL^s_1;\QQ)
\]
\end{lemma}

\begin{proof}
Let $M^m$ be a closed manifold with a map $f:M^{4m}\rightarrow \LL^s_1$. Suppose $M$ is embedded in some sphere $S^{4(m+N)}$. Then consider the diagram

\begin{tikzcd}
 M \arrow[r,"f\times \nu_M"] \arrow[d] & \LL^s_1\times BSO(4N) \arrow[d]  \\
 S^{4(m+N)} \arrow[r] & (\LL^s_1)^{+}\wedge MSO(4N) \arrow[r,"\sigma^s_{\odd}\wedge \Delta"] & BSO_{(\odd)}\wedge BSO_{(\odd)} \arrow[r,"\otimes"] & BSO_{(\odd)}
\end{tikzcd}

Let $g:S^{4(m+N)}\rightarrow BSO_{(\odd)}$ be the composition of the lower horizontal arrows. Then 
\begin{multline*}
\sigma^s_{\odd}(M,f)= \langle g^*\ph,[S^{4(m+N)}]\rangle =\langle f^*\ph(\sigma^s_{\odd})\cdot L_M\cdot U_{\nu_M},[S^{4(m+N)}]\rangle \\ =\langle f^*\ph(\sigma^s_{\odd})\cdot L_M,[M]\rangle   
\end{multline*}

On the other hand,
\[
\sigma^s_{\odd}(M,F)=\sigma^s_0(\mathcal{C}_f(M))=\langle L_M\cdot f^*l^s_1,[M]\rangle
\]
\end{proof}

It remains to show that $\sigma^s_{(\odd)}$ induces an isomorphism of homotopy group. Take a generator $f:S^{4n}\rightarrow \LL^s_0$ of $\pi_{4n}(\LL^s_1)$ so that the signature of the assembly of the associated presheaf is $1$. Then
\[
1=\sigma^s_{\odd}(S^{4n})=\langle L_{S^{4n}}\cdot f^*l^s_1,[S^{4n}]\rangle=\langle f^*\ph(\sigma^s_{\odd}),[S^{4n}]\rangle
\]
So $f\circ \sigma^s_{\odd}:S^{4n}\rightarrow BSO_{(\odd)}$ is also a generator of $\pi_{4n}(BSO_{(\odd)})$

\begin{theorem}
Localized at odd primes, there is an $H$-space homotopy equivalence 
\[
\sigma^s_{\odd}:\LL^s_1\rightarrow BSO^{\otimes}_{(\odd)}
\]
\end{theorem}

\begin{proof}
We already proved the homotopy equivalence. Note that the chain-level tensor product makes the restriction of the multiplicative structure of $\LL^s$ to $\LL^s_1$. The map $\sigma^s_{\odd}$ is an $H$-space map since $\ph(\sigma^s_{\odd})=l^s_1\otimes \QQ$ is multiplicative.
\end{proof}

Hence, we finish the proof of the following:

Therefore, we have proved the following:

\begin{theorem}\label{main-result-3}
At odd primes, Ranicki-Levitt's $L$-theory orientations for $TOP$ bundles and spherical fibrations are equivalent to Sullivan's real $K$-theory orientation.
\end{theorem}

We can also restate the theorem in terms of presheaves.

\begin{proposition}
For any presheaf $\mathcal{S}$ of $0$-connective Poincar\'e symmetric chains over $X$, there exists an odd-prime  real $K$-theory characteristic class $\gamma^s(\mathcal{S})\in \widetilde{KO}(X)_{(\odd)}$,
which is invariant under bordism.
\end{proposition}

Analogously, 

\begin{proposition}
For any presheaf $\mathcal{Q}$ of $0$-connective Poincar\'e quadratic chains over $X$, there exists an odd-prime real $K$-theory characteristic class $\gamma^q(\mathcal{Q})\in \widetilde{KO}(X)_{(\odd)}$,
which is invariant under bordism.  
\end{proposition}

With considerations of the $L$-theory orientations for spherical fibrations and $TOP$ bundles, we reprove the following (\cite[Theorem 6.5]{Sullivan2009}).

\begin{corollary}
Localized at odd primes, the obstruction for lifting a spherical fibration to a $TOP$ bundle is the existence of a real $K$-theory orientation. 
\end{corollary}

\bibliographystyle{amsalpha}
\bibliography{ref}

\Addresses

\end{document}